\documentclass[reqno,11pt]{article}
\usepackage{amsmath, latexsym, amsfonts, amssymb, amsthm,appendix}
\usepackage{graphicx, color,hyperref,dsfont,tikz,caption,wrapfig,subfig}
\usepackage{makeidx}
\usepackage{pstricks}

\hypersetup{
    linktoc=page,
    linkcolor=red,          
    citecolor=blue,        
    filecolor=blue,      
    urlcolor=cyan,
    colorlinks=true           
}

\numberwithin{equation}{section}
\newtheorem{theorem}{Theorem}[section]
\newtheorem{lemma}[theorem]{Lemma}
\newtheorem{proposition}[theorem]{Proposition}
\newtheorem{corollary}[theorem]{Corollary}
\newtheorem{rem}[theorem]{Remark}
\newtheorem{definition}[theorem]{Definition}

\newtheorem{conjecture}[theorem]{Conjecture}

\newtheorem*{assA}{Assumption (A)}
\newtheorem*{assB}{Assumption (B)}

\newcommand{\bbE}{{\ensuremath{\mathbb E}} }

\newcommand{\bbL}{{\ensuremath{\mathbb L}} } 
 
\newcommand{\bbN}{{\ensuremath{\mathbb N}} } 
 
\newcommand{\bbP}{{\ensuremath{\mathbb P}} } 
 
\newcommand{\bbR}{{\ensuremath{\mathbb R}} } 
\newcommand{\bbS}{{\ensuremath{\mathbb S}} }

\renewcommand{\P}{{\ensuremath{\bbP}} } 
\makeindex
\renewcommand{\tilde}{\widetilde}          
\DeclareMathSymbol{\leqslant}{\mathalpha}{AMSa}{"36} 
\DeclareMathSymbol{\geqslant}{\mathalpha}{AMSa}{"3E} 
\DeclareMathSymbol{\eset}{\mathalpha}{AMSb}{"3F}     
\renewcommand{\leq}{\;\leqslant\;}                   
\renewcommand{\geq}{\;\geqslant\;}                   
\newcommand{\dd}{\mathrm{d}}             

\DeclareMathOperator*{\Cov}{\mathrm{Cov}}  

\newcommand{\C}{\mathbb{C}}

\newcommand{\R}{\mathbb{R}}

\newcommand{\N}{\mathbb{N}}

\newcommand{\E}{\mathds{E}}
\newcommand{\Pb}{\mathds{P}}
\newcommand{\ind}{\mathds{1}}

\renewcommand{\ge}{\geq}
\renewcommand{\le}{\leq}

\newcommand{\gga}{\gamma}
\newcommand{\gl}{\lambda}
\newcommand{\gb}{\beta}
\newcommand{\gep}{\varepsilon}
\newcommand{\go}{X}

\renewcommand{\log}{\ln}

\title{Complex Gaussian multiplicative chaos}

\date{}

\begin{document}

\maketitle
\begin{center}

{ Hubert Lacoin \footnotemark[1],\ R\'emi Rhodes \footnotemark[2]\footnotemark[4],\ 
Vincent Vargas \footnotemark[3]\footnotemark[4]}

\bigskip

\footnotetext[1]{IMPA, Intitudo Nacional de Matem\`atica Pura e Aplicada,
Estrada Dona Castorina 110,
Rio de Janeiro / Brasil 22460-320} 
\footnotetext[2]{Universit{\'e} Paris-Est Marne la Vall\'ee, LAMA, Champs sur Marne, France.}
\footnotetext[3]{\'Ecole Normale Supérieure, DMA, 45 rue d'Ulm,  75005 Paris, France.}
 \footnotetext[4]{Partially supported by grant ANR-11-JCJC  CHAMU}

\end{center}

\begin{abstract}
In this article, we study complex Gaussian multiplicative chaos. More precisely, we study the renormalization theory and the limit of the exponential of a complex log-correlated Gaussian field in all dimensions (including  Gaussian Free Fields in dimension 2). Our main working assumption is that the real part and the imaginary part are independent. We also discuss applications in $2D$ string theory; in particular we give a rigorous mathematical definition of the so-called Tachyon fields, the conformally invariant operators in critical Liouville Quantum Gravity with a $c=1$ central charge, and derive the original KPZ formula for these fields.
\end{abstract}
\vspace{0.3cm}
\footnotesize


\noindent{\bf Key words or phrases:} Random measures, complex Gaussian multiplicative chaos, tachyon fields, multifractal.

\medskip

\medskip

\normalsize


\section{Introduction}

In dimension $d$, a real Gaussian multiplicative chaos is a random measure on a given domain $D$ of $\R^d$ that can be formally written,  for any Borel set $A\subset D$ as:
\begin{equation}\label{measintrorev}
M^\gamma(A)=\int_Ae^{\gamma X(x)-\frac{\gamma^2}{2}\E[X^2(x)]}\,\dd x,
\end{equation}
 where $\dd x$ stands for the Lebesgue measure on $D$ (more generally Gaussian multiplicative chaos can also be constructed when $\dd x$ is replaced 
 by any Radon measure:  see \cite{review} 
 for a recent review on the subject) and $X$ is a centered Gaussian distribution possessing a covariance kernel of the form:
\begin{equation}\label{Kintrorev}
\E[X(x)X(y)]=\ln _+\frac{1}{|x-y|}+g(x,y),
\end{equation} with $\ln _+(u)=\max(\ln u , 0)$ and $g$  a continuous  function over $D\times D$. The covariance kernel thus possesses a singularity along the diagonal and it is clear that
making sense of \eqref{measintrorev} is not straightforward (how do you define the exponential of a distribution?). The standard approach consists in applying a ``cut-off " to the distribution $X$, that is in regularizing the field $X$ in order to get rid of the singularity of the covariance kernel. The regularization usually depends on a small parameter, call it $\gep$, that stands for the extent to which the field has been regularized. The measure \eqref{measintrorev} is naturally understood as the 
limit of the random measures:
\begin{equation}\label{limitreal}
M^\gamma_\gep(A)\index{Mgamma@$M^\gamma_\gep$}=\int_Ae^{\gamma X_\gep(x)-\frac{\gamma^2}{2}\E[X_\gep^2(x)]}\,\dd x.
\end{equation}
 when the regularization parameter $\gep$ goes to $0$. It is well known \cite{cf:Kah,Bar,review} that this procedure  produces non trivial limiting objects when the real parameter $\gamma$ is strictly less than the critical value $\gamma_c=\sqrt{2d}$.  

The critical case, i.e. $\gamma= \gamma_c$, has been investigated in \cite{Rnew7,Rnew12}. An extra renormalization term is necessary to obtain
 a non-degenerate random limit $M'$ \index{Mprime@$M'$}. This limiting measure is also called the \textit{derivative Gaussian multiplicative chaos} because it can also be obtained by differentiating \eqref{measintrorev} with respect to the parameter $\gamma$:
\begin{equation}\label{deriv}
M'(A)=\lim_{\gep\to 0} \int_A\big(\gamma_c\E[X_\gep^2(x)]-X_\gep(x)\big)e^{\gamma_c X_\gep(x)-\frac{\gamma_c^2}{2}\E[X_\gep^2(x)]}\,\dd x.
\end{equation}
This object has recently received much attention because of its fundamental role in analyzing the behaviour of the maximum of log-correlated Gaussian fields. The reader is referred to the papers \cite{DZ,Louisdor,maxmadaule} for substantial recent advances on this topic.
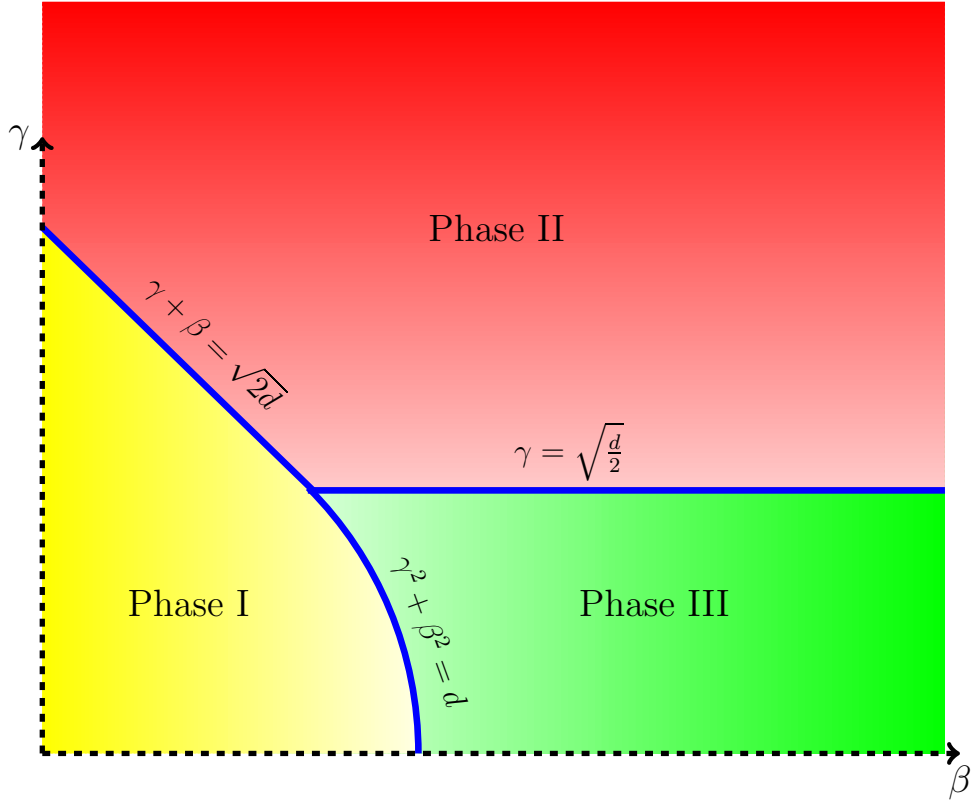
\begin{figure}[t]
\centering
\begin{tikzpicture}
\shade[left color=green!10,right color=green] (3,0) -- (12,0) -- (12,6) -- (3,6) ;
\shade[top color=red, bottom color=red!20] (3.5,3.5) -- (12,3.5) -- (12,10) -- (0,10) -- (0,7);
\shade[left color=yellow,right color=yellow!10] (5,0) arc (0:45:5) -- (0,7) -- (0,0) -- (5,0);
\draw [line width=2.5pt,color=blue](3.53,3.5) -- (12,3.5);
\draw [line width=2.5pt,color=blue](3.6,3.5) -- (0,7);
\draw [line width=2.5pt,color=blue](5,0)  arc (0:45:5);
\draw[style=dashed,line width=2pt,->] (0,0) -- (0,8.2) node[left]{{\Large $\gamma$}};
\draw[style=dashed,line width=2pt,->] (0,0) -- (12.2,0)node[below]{{\Large $\beta$}};
\draw (7,3.5)  node[above]{{\large $\gamma=\sqrt{\frac{d}{2 }}$}};
\draw (2,5.2)  node[above,rotate=-45,line width=2pt]{{\large $\gamma+\beta=\sqrt{2d}$}};
\draw (4.8,1.5)  node[above,rotate=-70,line width=2pt]{{\large $\gamma^2+\beta^2=d$}};
\draw (1,2)  node[right]{{\Large Phase I}};
\draw (5,7)  node[right]{{\Large Phase II}};
\draw (7,2)  node[right]{{\Large Phase III}};
\end{tikzpicture}
\caption{Phase diagram}
\label{diagram}
\end{figure}
The super-critical case $\gamma>\gamma_c$ has been investigated in \cite{MRV}: we will recall the obtained results at the beginning of section \ref{sec:conj}.

The standard theory of Gaussian multiplicative chaos has found many applications in finance, Liouville Quantum Gravity \index{Liouville Quantum Gravity} or turbulence (see \cite{review} and references therein). Yet, the need of understanding the renormalization theory of Gaussian multiplicative chaos with a 
complex value of the parameter $\gamma$ has emerged. This is for instance the case in $2D$-string theory and more precisely when looking at conformal field theories coupled to gravity (see below). 

In this paper, we consider two independent identically distributed centered Gaussian distributions $X$ and $Y$, each of which with covariance kernel of the type \eqref{Kintrorev}. By considering their respective regularizations $(X_\gep)_\gep$ and $(Y_\gep)_\gep$, the problem addressed here is to find a proper renormalization as well as the limit of the family of complex random measures \index{Mgammabeta@$M^{\gamma,\beta}_\gep$}:
\begin{equation}\label{intromeps}
M^{\gamma,\beta}_\gep(A)=\int_A e^{\gamma X_\gep(x)+i\beta Y_\gep(x)}\,\dd x
\end{equation}
where $\gamma,\beta$ are real constants. The convention for $M^{\gamma,\beta}_\gep$ we will use in this paper is thus different than the one sometimes used in the real case, i.e. \eqref{limitreal} where one already renormalizes by the mean in the definition. Notice that we may restrict to the case when $\gamma,\beta$ are nonnegative by symmetry of the Gaussian law. We will see that the renormalization theory of these measures presents three phases, summarized in Figure \ref{diagram}, depending on the considered values of $\gamma$ and $\beta$.

 It is of course natural to consider other frameworks than the one considered in this paper. For example, one can consider the general case where for all fixed $x$ the variables $X_\gep(x)$ and $Y_\gep(x)$ have some fixed correlation $\rho \in [-1,1]$. In this case, one expects a similar phase diagram to hold, at least at the level of the free energy of the system. More specifically, if $A$ is a fixed set, the following limit  
\begin{equation*}
\underset{\gep \to 0}{\lim} \: \frac{1}{\gep}  \,   \ln |M^{\gamma,\beta}_\gep(A)|
\end{equation*}
should be independent from $\rho$. However, we are not aware of a solid argument which would say that the proper renormalization of $M^{\gamma,\beta}_\gep$ (which is the object of this work in the case $\rho=0$) is independent from $\rho$ (except in the inner phase I where it is a consequence of the results of \cite{bar:comp1,bar:comp2} in dimension 1). Let us stress that the general case is of course interesting (and natural to look at in the case $\rho=1$)  but its study requires new ideas beyond the methodology of this paper.

\subsection{Previous related works}

We first mention the recent work \cite{KaKli} where the authors conduct a thorough study of all phases (inner and frontier) in the simpler context of the complex Random Energy Model (REM) partition function. They give the precise asymptotics of all phases. Note that in this context, phase I (inner and frontier) is trivial at order 1 as the (mean) renormalized partition function converges to a non vanishing constant: this is due to the lack of correlations in the model. Hence, in \cite{KaKli}, the authors go in fact one step further as they give the fluctuations.

As is now well known, correlations may be added for instance on a tree structure like Mandelbrot multiplicative cascades. In this context, this problem is investigated in \cite{bar:comp1,bar:comp2,derrida}. In the pioneering work \cite{derrida}, the authors computed the free energy and deduced a phase diagram similar to our Figure \ref{diagram}. In \cite{bar:comp1}, the authors treat the case of dyadic multiplicative cascades. Since this model is $1$-dimensional, they consider the complex random measure $M^{\gamma,\beta}_\gep$ as a random function $t\mapsto M^{\gamma,\beta}_\gep([0,t])$. Translated in our context, the results in \cite{bar:comp1} are the following:
 \begin{itemize}
 \item
 Phase I: the authors prove  that there is almost sure convergence in the space of continuous functions when $M^{\gamma,\beta}_\gep$ is renormalized by its mean. In fact, in the companion paper \cite{bar:comp2}, the authors show almost sure convergence in the space of continuous functions for a class of models that includes the one we consider here (except for Gaussian free fields).   
 \item
 Phase II, frontier  I/II, frontier II/III and triple point: not investigated.
 \item
 Phase III and frontier I/III::  the authors show that the sequence is tight when $M^{\gamma,\beta}_\gep$ is properly renormalized (in this case, the renormalization is different than the one in phase I).  Convergence is not investigated.
 \end{itemize}
We further stress that the authors in \cite{bar:comp1,bar:comp2} do not prove convergence in law  in phase III but claim that if convergence holds then every possible limit is a Brownian motion in multifractal time. The argument is based on the uniqueness property of the solution of some fixed point equation, the star equation for multiplicative cascades.  The corresponding equation for Gaussian multiplicative chaos has been introduced in  \cite{Rnew1} but uniqueness has only been established in the real subcritical case so far. So the same argument cannot be used in the context of Gaussian multiplicative chaos. 


\subsection{Content of the paper} 
The purpose of this manuscript is not only to investigate the phase diagram in the context of Gaussian multiplicative chaos but also to describe the limiting object  that we obtain when  renormalizing properly the family $(M^{\gamma,\beta}_\gep)_\gep$ as $\gep\to 0$. We will first show that the model exhibits three phases, which are represented in Figure \ref{diagram}. We will also describe the limiting objects after renormalization. Apart from  the inner phase I and the frontier of phases I/II where the renormalization procedure produces new objects, we  prove that the renormalization procedure leads to objects that can be described in terms of the limiting measures that we get on the real line $\beta=0$. Roughly speaking, in phase II (excluding frontier I/II and the real line $\beta=0$) and in phase III (and the frontier I/III), we get a 
complex Gaussian white noise (see \eqref{defwhitenoise} for a definition) with a random intensity: the real part $X_\gep$ governs the description of the intensity whereas all the information about the field $Y$ is lost into a white noise (a similar phenomenon is observed in \cite{DoLeWi} for a different model). Figure \ref{limit} is a brief yet complete description of the picture we draw. Actually, we do not treat the triple point and the description 
we give in the inner phase II is a conjecture: our method to explain how the complex Gaussian random measure appears indicates that the inner phase II can be described by a complex Gaussian random measure with random intensity given by the objects we get  in the real supercritical case ($\gamma>\sqrt{2d}$): see section \ref{sec:conj} for more on this. We will also detail applications in $2D$-string theory, which are summarized below.

\begin{figure}[t]
\centering
\begin{tikzpicture}
\shade[left color=green!10,right color=green] (3,0) -- (12,0) -- (12,6) -- (3,6) ;
\shade[top color=red, bottom color=red!20] (3.5,3.5) -- (12,3.5) -- (12,10) -- (0,10) -- (0,7);
\shade[left color=yellow,right color=yellow!10] (5,0) arc (0:45:5) -- (0,7) -- (0,0) -- (5,0);
\draw [line width=23pt,color=blue!50](3.53,3.5) -- (12,3.5);
\draw [line width=23pt,color=blue!50](3.6,3.5) -- (-0.2,7.2);
\draw [line width=23pt,color=blue!50](5,0)  arc (0:45:5);
\draw (8.5,3.2)  node[above]{Limit: $W_{\sigma^2 M'}$ $(|\ln\gep|^{1/4})$};
\draw (1.6,4.9)  node[above,rotate=-45,line width=2pt]{new family ($\gep^{ \frac{\gamma^2}{2}-\frac{\beta^2}{2}}$)};
\draw (4.1,1.5)  node[above,rotate=-70,line width=2pt]{$W_{\sigma^2M^{2\gamma}}$ ($\frac{\gep^{ \gamma^2-\frac{d}{2}}}{|\ln\gep|^{\frac{1}{2}}}$)};
\draw (1,3)  node[right]{{\Large Phase I}};
\draw (0.2,2.3)  node[right]{Limit: new family };
\draw (0.9,1.7)  node[right]{($\gep^{ \frac{\gamma^2}{2}-\frac{\beta^2}{2}}$)};
\draw (7,7.5)  node[right]{{\Large Phase II}};
\draw (6.2,6.8)  node[right]{Limit: $W_{N^\alpha_{M'}}$ with $\alpha=\sqrt{d/2}\gamma^{-1}$ };
\draw (7,6.1)  node[right]{$(|\ln\gep|^{\frac{3\gamma}{2\sqrt{2d}}}\gep^{\gamma\sqrt{2d}-d})$ };
\draw (7,2)  node[right]{{\Large Phase III}};
\draw (7,1.3)  node[right]{Limit: $W_{\sigma^2 M^{2\gamma,0}}$ ($\gep^{ \gamma^2-\frac{d}{2}}$)};
\node[draw,fill=white] (M) at (3.5,6.9) {Limit: $\sigma^2 M'$ $(\gep^{d} |\ln\gep|^{1/2})$};
\draw[->,line width=1.5pt] (M) -- (0.1,7);
\node[draw,fill=white] (R) at (5,9) {Limit: $\sigma^2 N^\alpha_{M'}$ with $\alpha=\frac{\sqrt{2d}}{\gamma}$ $(|\ln\gep|^{\frac{3\gamma}{2\sqrt{2d}}}\gep^{\gamma\sqrt{2d}-d})$};
\draw[->,line width=1.5pt] (R) -- (0.05,9.8);
\draw[->,line width=1.5pt] (R) -- (0.05,8.5);
\draw[->,line width=1.5pt] (R) -- (0.05,7.5);
\node[draw,fill=white] (M) at (7.3,4.6) {Triple point:  $W_{\sigma^2M'}$ ($|\ln \gep|^{-1/4}$)};
\draw[->,line width=1.5pt] (M) -- (3.8,4);
\draw[color=white,fill=white] (-0.5,4.5) -- (0,4.5) -- (0,8) -- (-0.5,8);
\draw[style=dashed,line width=2pt,->] (0,0) -- (0,10.2) node[left]{{\Large $\gamma$}};
\draw[style=dashed,line width=2pt,->] (0,0) -- (12.2,0)node[below]{{\Large $\beta$}};
\end{tikzpicture}
\caption{Limiting measure diagram. 
We indicate between brackets how to renormalize the field 
$M^{\gamma,\beta}_\gep$ and the limiting field. For instance in phase III, ``Limit $W_{\sigma^2M^{2\gamma,0}}$ ($\gep^{ \gamma^2-\frac{d}{2}}$)" means that the field $\gep^{ \gamma^2-\frac{d}{2}}M^{\gamma,\beta}_\gep$ converges as $\gep\to 0$ towards $W_{\sigma^2M^{2\gamma,0}}$. Now we explain the description of the limiting law: conditionally on $\mu$, $W_{\mu}$ \index{Wmu@$W_\mu$}stands for a complex Gaussian random measure with intensity $\mu$. Conditionally on $\mu$, $N^\alpha_\mu$ is a $\alpha$-stable Poisson random measure with intensity $\mu$. $M'$ is the derivative martingale and $M^{2\gamma,0}$ is a standard Gaussian multiplicative chaos with intermittency parameter $2\gamma$. The constant $\sigma^2$ depends on $\gamma$ and $\beta$.}
\label{limit}
\end{figure}
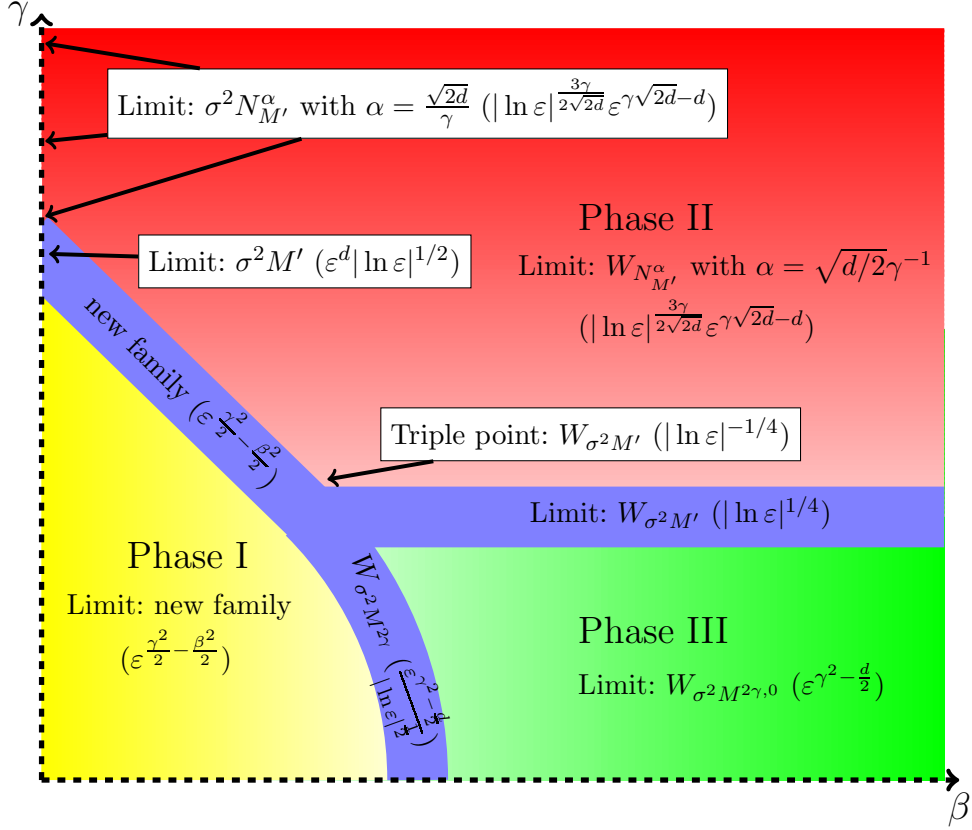

\subsection{Applications in $2D$-string theory} 
$2D$ string theory is the coupling of the Polyakov action with the action of $2D$-gravity on a two-dimensional surface $\Sigma$. Polyakov \cite{Pol} showed that $2D$ string theory could be interpreted as a theory of two dimensional Liouville quantum field theory (LFT for short, see \cite{Nak} for an exhaustive review on the topic and \cite{DKRV} for a mathematical construction) together with an independent Gaussian Free Field \index{Gaussian free field (GFF)} $Y$. The metric on $\Sigma$ is thus a random variable, which roughly takes on the form \cite{Pol,cf:KPZ,cf:Da} (we consider an Euclidean background metric for simplicity):
\begin{equation}\label{tensorintro}
g(z)=e^{2 X(x)}dx^2,
\end{equation}
  and $X$ is a Liouville field, the fluctuations of which   are governed by the Liouville action. If one neglects the cosmological effects in LFT, the field $X$ becomes a Free Field, with appropriate mean and boundary conditions. This framework is also conjectured to be  the  scaling limit of critical statistical physical models having a $c=1$ central charge (like the $O(n=2)$ loop model or the $4$-state Potts model) defined on random lattices. We do not review here the huge amount of work on this topic and we refer the reader to \cite{david-hd,cf:Da,cf:DuSh,DFGZ,DistKa,bourbaki,GM,glimm,Kle,cf:KPZ,Nak,Pol} for further insights.

In $2D$-string theory, the so-called {\bf tachyon fields} \index{tachyon field} $T$ are the operators which are conformally invariant within the theory (see the excellent reviews \cite{Kle,Nak}).   In this paper, we will mathematically construct the tachyon fields
$$e^{\gamma X(x)+i\beta Y(x)}\,\dd x$$ for $\gamma\pm\beta=2$ and $\gamma\in ]1,2[$, where following the above discussion $X$ and $Y$ are two independent Gaussian Free Fields.

Finally, we will also derive the corresponding KPZ formula (see \cite{cf:KPZ})
$$\Delta^0_{i\beta}= \Delta^q_{i\beta}+\Delta^q_{i\beta}(\Delta^q_{i\beta}-1),$$
which is a relation between the conformal dimension $\Delta^0_{i\beta}$ of the 
spinless vertex operator $e^{i\beta Y}$ and the quantum dimension $\Delta^q_{i\beta}$ 
 of this operator gravitationally dressed. 
The reader is referred to Section \ref{LQG} for further details.

\subsection*{Acknowledgements}
We would like to thank F. David, P. Le Doussal, B. Duplantier and S. Sheffield for interesting discussions on these topics. Special thanks are addressed to R. Allez for his help during the early stages of this project. Finally, we would like to thank the anonymous referees for their very careful reading of a first version of this manuscript.

\section{Setup } \label{mainsetup}
Let us introduce a canonical family of log-correlated Gaussian distributions, called star scale invariant\index{star scale invariant}, and their cut-off approximations, which we will work with in the first part of  this paper. Of course, other natural choices are possible and they are discussed in subsection \ref{extension} below.

We consider a positive definite function $k$. For technical reasons 
we assume that $k$ satisfy the following assumption.

\begin{assA} There exists a constant $C_k$ such that
\begin{description}\label{assA}
\item[A1.] $k$ is  normalized by the condition $k(0)=1$,
\item[A2.] $|k(x)|\leq C_k(1+|x|)^{-\nu}$ for some $\nu>d$.
\item[A3.] $|k(x)-k(0)|\leq C_k|x|  $ for all  $x\in\R^d$.
\end{description}
\end{assA}

We set for $\gep\in ]0,1]$ and $x\in\R^d$
 \begin{equation}\label{deflenoyau}
K_\gep(x)\index{Keps@$K_\gep$}=  \int_1^{\frac{1}{ \gep }}  \frac{k(xu)}{u}\dd u
 \end{equation}
 and  $$G_\gep(x)=e^{-K_\gep(x)}.$$ \index{Geps@$G_\gep$} 
We consider two independent families of centered Gaussian processes $(X_\gep(x))_{x\in\R^d,\gep\in]0,1]}$ and $(Y_\gep(x))_{x\in\R^d,\gep\in]0,1]}$ with covariance kernel given by:
 \begin{equation}\label{ladefcorrel}
\forall\gep,\gep'\in ]0,1],\quad  \E[ X_\gep(x) X_{\gep'}(y)   ]= \E[ Y_\gep(x) Y_{\gep'}(y)   ]
= K_{\gep \vee \gep'}  (y-x ),
 \end{equation}
where $\gep \vee \gep':= \sup (\gep,\gep')$. The construction of such fields is possible via a white noise decomposition as explained in \cite{Rnew1}. For the sake of completeness, we recall this construction here. First, given $\nu(\dd x)$ a locally finite measure one $\bbR^d$, we define $W_\nu$
the complex white noise or standard complex Gaussian measure
with intensity $\nu$ to be the centered Gaussian process indexed by the Borelians of $\bbR^d$, for which real and imaginary part are independent
and whose covariance is given by

\begin{equation}\label{defwhitenoise}
 \E[W_\nu(A)\overline{ W}_\nu(B)]=\nu(A\cap B).
\end{equation}
One can similarly define a real white noise where in definition \eqref{defwhitenoise} there is no complex conjugate.  

Let $F$ be the symmetric spectral measure associated to $k$ by the following relation for all $x \in \R^d$
\begin{equation*}
k(x)= \int_{\R^d} e^{i \xi . x}  \, F(\dd \xi).
\end{equation*}
The process $(X_\gep(x))_{x\in\R^d,\gep\in]0,1]}$ is defined by 
\begin{equation}\label{defprocessus}
X_\gep(x)=  \int_{\R^d}  \int_1^{\frac{1}{\gep}}  \cos (u x. \xi)  W(u, \dd \xi)  + \int_{\R^d}  \int_1^{\frac{1}{\gep}}  \sin (u x. \xi)  W'(u, \dd \xi)
\end{equation}
where $x.\xi$ denotes the standard scalar product between $x$ and $\xi$ and $W,W'$ are two independent real white noises associated to the measure $\frac{\dd u}{u} F(\dd \xi)$. Note that $X_\gep$ is an approximation of the field $X$ defined by: 
\begin{equation*}
X(x)=  \int_{\R^d}  \int_1^{\infty}  \cos (u x. \xi)  W(u, \dd \xi)  + \int_{\R^d}  \int_1^{\infty}  \sin (u x. \xi)  W'(u, \dd \xi),
\end{equation*}
in the sense that $X_\gep$ converges almost surely to $X$ in the space of tempered distributions as $\gep$ goes to $0$.
Now, the covariance of $X_\gep$ given by \eqref{defprocessus} is for $\gep' \leq \gep$
\begin{align*}
\E[ X_\gep(x) X_{\gep'}(y)   ] & =    \int_{\R^d}  \int_1^{\frac{1}{\gep}}  \cos (u x. \xi) \cos (u y. \xi)  \frac{\dd u}{u} F(\dd \xi)  + \int_{\R^d}  \int_1^{\frac{1}{\gep}}  \sin (u x. \xi) \sin (u y. \xi)  \frac{\dd u}{u} F(\dd \xi)  \\
& =    \int_1^{\frac{1}{\gep}}  \int_{\R^d} \cos (u (x-y). \xi)   F(\dd \xi) \frac{\dd u}{u}\\
& =    \int_1^{\frac{1}{\gep}} \int_{\R^d}  e^{ i u (x-y). \xi}   F(\dd \xi) \frac{\dd u}{u}\\
& =     \int_1^{\frac{1}{\gep}}  k( u (x-y))  \frac{\dd u}{u}, 
\end{align*}
hence the covariance of $X_\gep$ satisfies \eqref{ladefcorrel} (therefore the field $X$ has covariance $ \int_1^{\infty}  k( u (x-y))  \frac{\dd u}{u}$ which is of the form \eqref{Kintrorev}).  

 Now, we  set: 
\begin{align*}%
\mathcal{F}^X_{\gep}\index{FXeps@$\mathcal{F}^X_{\gep}$}= \sigma \lbrace X_u(x);x\in\R^d,u \geq \gep\rbrace\quad & \text{and}\quad \mathcal{F}^Y_{\gep}  \index{FYeps@$\mathcal{F}^Y_{\gep}$}= \sigma \lbrace Y_u(x);x\in\R^d,u \geq \gep\rbrace,   \\
\mathcal{F}^X \index{FX@$\mathcal{F}^X$} = \sigma \lbrace X_u(x);x\in\R^d,u \in ]0,1]\rbrace\quad &\text{and}\quad \mathcal{F}^Y \index{$\mathcal{F}^Y$}= \sigma \lbrace Y_u(x);x\in\R^d,u \in ]0,1]\rbrace   
\end{align*} 
and
$$\mathcal{F}_{\gep} = \sigma \lbrace X_u(x),Y_u(x);x\in\R^d,u \geq \gep\rbrace    .$$
An important aspect of the construction is the fact  that, for $u< \gep$, the field $(X_u(x)-X_\gep(x))_{x \in \R^d}$ is independent from $\mathcal{F}_{\gep}$. Then we consider the following locally finite complex measure:
\begin{equation}\label{def:meas}
\forall A\in \mathcal{B}(\R^d),\quad M^{\gamma,\gb}_\gep(A)=\int_Ae^{\gamma X_\gep(x)+ i \gb Y_\gep (x)}\,\dd x.
\end{equation}
 Let us finally notice that $K$ can be approximated as follows for all $R>0$
 \begin{equation}\label{froomk}
 \forall x\in B(0,R),\quad  \big| K_\gep(x)- |\log (|x|\vee \gep) | \big |\le C_R.
 \end{equation}
 
 \begin{rem}
 Note that in \eqref{def:meas}, contrary to the real case \eqref{limitreal}, 
 we have chosen not to
 include a term $\gamma^2\bbE[(X_\gep(x))^2]/2$ (and a similar one for $Y$).
 In particular, note that $M^{\gamma,\gb}_\gep(A)$ is not a martingale in $\gep$.
 \end{rem}

\begin{rem}
Let us comment a bit Assumption (A):
\begin{itemize}
\item
Assumption A.1 is just a normalization convention ensuring that the kernel $K_\gep(x)$ satisfies \eqref{froomk}. Changing the value of $k(0)$ has the same effect as changing the value of $\gamma,\beta$ in \eqref{def:meas} hence it is no restriction.  
\item
Assumption A.2 is a technical condition ensuring that the integral \eqref{deflenoyau} is convergent.
\item
Assumption A.3 is a technical condition which was used in the construction of the critical Gaussian multiplicative chaos in \cite{Rnew7,Rnew12}; however, it can probably be relaxed to a H\"older type condition  in \cite{Rnew7,Rnew12}  but for the sake of clarity we do not go into the details here. Hence, this assumption is only used to deal with the frontier II/III (elsewhere, one can just suppose that $k$ is continuous).    
\end{itemize}

\end{rem}

\subsection{Examples and other frameworks}\label{extension}
Let us also mention here some  important Gaussian fields covered by  our methods:
\begin{enumerate}
\item {\bf Exact scale invariant kernels} like the one studied in \cite{bacry,Rnew9}.
At first sight, these kernels do not satisfy Assumption (A) as 
they cannot be written as $\int_1^{\infty}\frac{k(ux)}{u}\,\dd u$. 
Yet, they can be written as $\int_1^{\infty}\frac{k(ux)}{u}\,\dd u+H(x)$ 
for some continuous translation invariant covariance kernel $H$. 
Adding $H$ in the covariance is equivalent to adding an independent 
smooth field with covariance $H$ to $X_\gep$ and $Y_\gep$.
Our proofs are easy to adapt in that case.

\item {\bf Massive Gaussian Free Field}  (MFF for short) on $\R^2$. 
The whole plane MFF is a centered Gaussian distribution with covariance kernel given by the Green function of the operator $\frac 1{2\pi} (m^2-\triangle)$ on $\R^2$, i.e. by:
\begin{equation}\label{MFFkernel}
\forall x,y \in \R^2,\quad G_m(x,y)=\int_0^{\infty}e^{-\frac{m^2}{2}u-\frac{|x-y|^2}{2u}}\frac{\dd u}{2 u}.
\end{equation}
  The real $m>0$ is called the mass. This kernel is of $\sigma$-positive type in the sense of Kahane \cite{cf:Kah} since we integrate a continuous function of positive type with respect to a positive measure. It is furthermore  a star-scale invariant kernel (see \cite{Rnew1,sohier}): it can be rewritten as 
\begin{equation}\label{MFF3}
G_m(x,y)=\int_{1}^{+\infty}\frac{k_m(u(x-y))}{u}\,\dd u.
\end{equation}
 for the continuous covariance kernel $k_m=\frac{1}{2}\int_0^\infty e^{-\frac{m^2}{2v}|z|^2-\frac{v}{2}}\,dv$.   
\item {\bf Gaussian Free Fields} in a compact domain. In dimension $2$, an important family of Gaussian distributions  is the family of Gaussian free fields (see \cite{glimm,She07} for instance). 
They do not satisfy Assumption (A) (in particular they are not translation invariant) and some modifications are needed to adapt the proofs to this case. 
 Because of the importance of applications in this context, we treat specifically these fields  in section \ref{GFF}.  
Applications are given in section \ref{LQG}.
\item  {\bf Log-correlated Gaussian Fields} (LGF) with covariance kernel given by $(m^2-\triangle)^{-d/2}$ in any dimension $d$ for $m\geq 0$ (see \cite{LCGF,LSSW} for instance). Furthermore, in the case of the whole plane and $m>0$, the Green function of the operator $(m^2-\triangle)^{d/2}$ is a star scale invariant kernel.  
\end{enumerate}

\subsection{Notations}
We will further denote by $C(E,F)$ the space of continuous functions from $E$ to $F$. The notation $f(x)\simeq_{x\to x_0} g(x)$ \index{$\simeq$} means that  
$$\lim_{x\to x_0} \frac{f(x)}{g(x)}=1.$$ 
We let $C^{k}(D)$ \index{$C^{k}(\R^d)$} (resp. $C^{k}_c(D)$ \index{$C^{k}_c(\R^d)$}) denote the space of functions defined on $D \subset \R^d$ that are $k$ times continuously differentiable (resp. $k$ times continuously differentiable with a compact support) equipped with the topology of uniform convergence for the derivatives up to order $k$. The random variables in this paper are defined on a probability space 
$(\Omega,\mathcal{F},\Pb)$ and we denote by $\E$ the corresponding expectation. The space of random variables with integrable $p$-th power is denoted $\bbL_p$ \index{L@$\bbL_p$}. The space of measurable functions defined on a Borel set $D \subset \R^d$ with integrable $p$-th power will be denoted $L^p(D)$.

When we make use of one of the two following inequalities
\begin{equation} 
\left( \sum_{i\in I} a_i\right)^{\theta} \ge  \sum_{i\in I} a_i^{\theta} \text{ when  } \theta \ge 1, \quad \quad \left ( \sum_{i\in I} a_i \right ) ^{\theta} \le \sum_{i\in I} a_i^{\theta} \text{ when  } \theta \le 1
\end{equation}
 which are valid for any collection of positive numbers $(a_i)_{i\in I}$
 we will simply say {\sl by superadditivity} or
 {\sl by subadditivity}.

\subsection{General convention} 

The paper is roughly divided in two parts. The first part is dedicated to the study of the renormalization theory of $M^{\gamma,\gb}_\gep$ given by \eqref{def:meas} where the approximation fields $X_\gep,Y_\gep$ have covariance \eqref{ladefcorrel}: this is the content of sections \ref{phaseI} to \ref{sec:conj}. In the second part, the complex measure $M^{\gamma,\gb}_\gep$ will also be given by expression \eqref{def:meas} where the fields $X_\gep,Y_\gep$ are this time approximations of the GFF in some bounded domain: this is the convention of section \ref{GFF} and the same convention will be used in section \ref{LQG}. In section \ref{GFF}, the approximations $X_\gep,Y_\gep$ can vary from one subsection to another but it should be clear from the context which approximation we are working with.

Finally, in phase I and its I/II boundary, we will always denote by $M^{\gamma,\beta}$ the following limit (which we will show to exist in an appropriate sense):
\begin{equation*}
M^{\gamma,\beta} := \underset{\gep \to 0}{\lim} \: \gep^{\frac{\gamma^2}{2}-\frac{\beta^2} {2}}   \,M_\gep^{\gamma, \beta}. 
\end{equation*}  
 We stress here that when $X_\gep,Y_\gep$ have covariance \eqref{ladefcorrel} the above limit corresponds to renormalizing $M_\gep^{\gamma, \beta}$ by its mean whereas, when $X_\gep,Y_\gep$ are approximations of a GFF, this corresponds to renormalizing $M_\gep^{\gamma, \beta}$ by its mean multiplied by a function of the conformal radius: see section \ref{GFF}.

\subsection{A toolbox of useful results}
Let us introduce here some useful results that we will massively use in the proofs.
The first one is the Girsanov formula expressed for general Gaussian processes, and will be used in many contexts:

\begin{lemma}[Girsanov Formula]
 Let $(X(x))_{x \in D}$ be a centered real Gaussian process indexed by $D \subset \R^d$ and
 $Z$ be a centered Gaussian random variable which is such that $(X,Z)$ is Gaussian (we let $\P$ denote the associated probability law).
 
 \medskip
 
Let $\tilde \P$ be the measure defined by

$$\frac{\dd \tilde \P}{\dd \P}=e^{Z-\bbE[Z^2]/2}.$$

 \medskip
 
 Then under $\tilde \P$, $(X(x))_{x \in D}$ is still a Gaussian process with the same 
 convariance. Its mean is equal to 
 
 $$\tilde \E [X(x)]=h_Z(x)=\Cov[Z, X(x)] $$

\end{lemma}

The following convexity inequality is proved in \cite{cf:Kah} by interpolation and Gaussian integration by parts;
we also present a special consequence of it.

\begin{proposition}[Kahane's convexity inequality]
\label{kahaha}
 
Let $Z_1$ and $Z_2$ be two centered Gaussian fields on $\bbR^d$ (or on any metric space) 
with covariance kernels 
$K_1(x,y)$ and $K_2(x,y)$ respectively.

\medskip

If \,\,\,$\forall x,y \in \bbR^d,\quad  K_1(x,y)\le K_2(x,y)$ then,
for all real convex functions $F$ and all positive measures $\sigma$ on $\bbR^d$, we have
\begin{equation}\label{khan2}
\bbE\left[ F\left(\int_{\bbR^d} e^{Z_1(x)-\bbE[Z^2_1(x)]/2}\sigma(\dd x)\right) \right]\le 
\bbE\left[ F\left(\int_{\bbR^d} e^{Z_2(x)-\bbE[Z^2_2(x)]/2}\sigma(\dd x)\right) \right].
\end{equation}

As a consequence if $\forall x,y \in \bbR^d,\quad K_1(x,y)\le K_2(x,y)+\alpha$  for some $\alpha\in \bbR$ then we have for $p>0$
\begin{equation}\begin{split}\label{khan1}
 \bbE\left[\left(\int_{\bbR^d} e^{Z_1(x)-\bbE[Z^2_1(x)]/2}\sigma(\dd x)\right)^p \right]&\le e^{\frac{1}{2}\alpha p(p-1)}
 \bbE\left[\left(\int_{\bbR^d} e^{Z_2(x)-\bbE[Z^2_2(x)]/2}\sigma(\dd x)\right)^p \right] \quad \text{ if } p>1 \\
 \bbE\left[\left(\int_{\bbR^d} e^{Z_1(x)-\bbE[Z^2_1(x)]/2}\sigma(\dd x)\right)^p \right]&\ge e^{\frac{1}{2}\alpha p(p-1) }
 \bbE\left[\left(\int_{\bbR^d} e^{Z_2(x)-\bbE[Z^2_2(x)]/2}\sigma(\dd x)\right)^p \right]\quad \text{ if } p<1.
\end{split}\end{equation}
\end{proposition}

\begin{proof}[Proof of \eqref{khan1}]
We consider the case $\alpha\ge 0$. In that case we apply \eqref{khan2} to the fields $Z_1$ and $Z_2+N$ where $N$ is a Gaussian of variance $\alpha$
which is independent of $Z_2$ (the kernel of $Z_2+N$ is $K_2+\alpha$).
Then the inequalities \eqref{khan1} are obtained by integrating over the variable $Z$.
When $\alpha<0$ we consider $Z_1+N$ and $Z_2$ instead.\end{proof}

The above proposition allows us 
to compare moments of order $p$ for two different log-normal multiplicative chaos integrated with respect 
to a measure $\sigma$. 
We will sometimes use it to make comparisons with a chaos which presents the nice 
property of being 
stochasticly
scale invariant. This specific chaos is constructed in \cite[Proposition 2.9]{Rnew9}

\begin{proposition}\label{scalinv}
For every dimension $d$ and $T>0$, one can construct a sequence of Gaussian fields $\{ (X_{\gep}(x))_{x\in \bbR^d}, \gep\ge 0\}$
whose covariance structures are given by
\begin{equation}
 \bbE[X_{\gep}(x)X_{\gep}(y)]=\int_{m\in O(d)} g_\gep(m(x-y)) \sigma_d(\dd m),
\end{equation}
where $O(d)$ is the orthogonal group on $\bbR^d$, $\sigma_d$ is the Haar measure on it and
\begin{equation}
 g_{\gep}(x):=\begin{cases}\log (T/\gep)+1-\frac{|x_1|}{\gep} &\text{ when } |x_1|\le \gep, \\
            \log_+(T/|x_1|) & \text{ when } |x_1|\ge \gep.
           \end{cases}
\end{equation}
where $x_1$ is the first real-coordinate of $x$ in $\bbR^d$.

For any fixed $\gl\in (0,1)$, we have the following equality in distribution
\begin{equation}\label{sssscal}
(X_{\gl \gep}(\gl x))_{x\in B(0,T/2)}\stackrel{(\dd)}{=} \Omega_\gl+ (X_{ \gep}( x))_{x\in B(0,T/2)},
\end{equation}
where $\Omega_{\gl}$ is a Gaussian variable of variance $|\log \gl |$ which is independent of $X_{\gep}$.
Finally given $T\ge 0$, $R>0$ there exists a constant $C$ such that for all $z\in B(0,R)$
\begin{equation}\label{crook}
|\log (|z|\vee \gep) |-C \le \bbE[X_{\gep}(x)X_{\gep}(x+z)]\le  |\log (|z|\vee \gep) |+C.
\end{equation}

\end{proposition}

\begin{rem}
The consequence of Equation \eqref{sssscal} is that the (real) multiplicative chaos satisfies the star-equation
\begin{equation}
M^\gamma(\gl \cdot)\stackrel{(\dd)}{=}\gl^{d+\gamma^2/2}e^{\gamma \Omega_\gl} M^\gamma(\cdot).
\end{equation}
For more on this subject see \cite{sohier}.
\end{rem}

\section{Study of phase I and its I/II boundary}\label{phaseI}
This section is concerned with the values of the parameters $(\gamma,\beta)$  for which a renormalization by the mean of $M^{\gamma,\beta}_\gep$  defined in \eqref{intromeps} yields a non trivial limit. Basically, this convergence relies on martingale techniques. Moment estimates in $\bbL_p$ of the martingale will be obtained by conditioning on the real part of the field, i.e. the field $X_\gep$, to be left with establishing capacity estimates for  a real Gaussian multiplicative chaos measure (these estimates of independent interest are the content of lemmas \ref{capI} and \ref{capI_II}). 
 
\subsection{Study of the inner phase}
 We study the inner phase I\index{inner phase I}, namely 
 \begin{equation}\label{innerI}
 \mathcal P_I:=\left\{\gamma+\beta<\sqrt{2d},\gamma\in\left[\sqrt{\frac{d}{2}},\sqrt{2d}\right)\right\}\cup\left\{\gamma^2+\beta^2<d\right\}.
 \end{equation}
Throughout the paper, we will use the terminology "inner phase I" to denote set of pairs 
$(\gamma,\beta)$ satisfying \eqref{innerI} in order to avoid heavy notations.

 We are interested in the martingale $( \gep^{\frac{\gamma^2}{2}-\frac{\beta^2} {2}} M_\gep^{\gamma, \beta}(\dd x) )_{\gep>0}$.
The reader can check that \eqref{innerI} is equivalent to the existence of some $p \in ]1,2[$ such that $\zeta(p)>d$ where:
 \begin{equation*}
 \zeta(p)=\left(d+ \frac{\gamma^2}{2}-\frac{\beta^2} {2}\right)p -\frac{\gamma^2}{2}p^2.
 \end{equation*}
As a warm-up, the reader may check that if  one considers $p\geq 2$ such that  $\zeta(p)>d$ then the martingale is bounded in $\bbL_2$.
This corresponds to the parameters $(\gamma,\beta)$ such that  $\gamma^2+\beta^2<d$. This $\bbL_2$ phase is rather straightforward to study. Also, if one introduces 
\begin{equation}
 \label{pc}
 p_c(\gamma,\gb):= \sup \left\{ p >1 \ | \ \zeta(p) >d \right\},
\end{equation}
 one gets that $p_c \in \left(\frac{\sqrt{2d}}{\gamma}, \frac{2d}{\gamma^2}\right]$.   We have the following behaviour inside phase I:
 
\begin{theorem}{\bf (Convergence)}\label{phase1} 
Let $(\gamma,\beta)$ belong to inner phase I. Consider $p \in ]1,2[$  such that $\zeta(p)>d$. \\
\noindent 1.\,\,\, For all compactly supported bounded measurable functions $f$, 
the martingale
\begin{equation*} 
 \left( \gep^{\frac{\gamma^2}{2}-\frac{\beta^2} {2}}\int_{\R^d} f(x)\,M_\gep^{\gamma, \beta}(\dd x) \right)_{\gep>0}.
\end{equation*}
is uniformly bounded in $\bbL_p$. Furthermore, for all $R>0$, there exists a constant $C_{p,R}$ (only depending on $p,R$) such that for all bounded measurable functions $f$ with compact support in $B(0,R)$:
$$\E\Big[\sup_{\gep\in ]0,1]}\Big| \gep^{\frac{\gamma^2}{2}-\frac{\beta^2} {2}}\int_{\R^d} f(x)\,M_\gep^{\gamma, \beta}(\dd x)\Big|^p\Big]\leq C_{p,R}\|f\|_\infty^p.$$
\noindent 2.\,\,\, The $\mathcal{D}'(\R^d)$-valued martingale:
\begin{equation*} 
 \gep^{\frac{\gamma^2}{2}-\frac{\beta^2} {2}} M^{\gamma,\beta}_\gep: \varphi \rightarrow  \gep^{\frac{\gamma^2}{2}-\frac{\beta^2} {2}}\int_{\R^d} \varphi(x) e^{\gamma X_\gep(x)+ i \gb Y_\gep (x)} \dd x
\end{equation*}
converges almost surely in the space $\mathcal{D}'_{d}(\R^d)$ of distributions of order $d$ towards a non trivial limit $M^{\gamma,\beta}$. More precisely, for each $R>0$, there exists a random variable $Z_R\in\bbL_{p}$ such that for all functions $\varphi\in C^d_c(B(0,R))$:
$$|M^{\gamma,\beta}(\varphi)|\leq Z_R \sup_{x \in B(0,R) } \left|\frac{\partial^d \varphi(x)}{\partial x_1 \cdots \partial x_d }\right|  .$$
\noindent 3.\,\,\, In dimension 1, 
we have convergence  of $\left(\gep^{\frac{\gamma^2}{2}
-\frac{\beta^2} {2}} M^{\gamma,\beta}_\gep[0,t]\right)_{t\in [0,T]}$ 
in the space of continuous functions. 
\end{theorem}

\begin{rem}
For the reader who wishes to skip the proofs, we  stress here that item 1. of Theorem  \ref{phase1}  is proved in dimension $1$ in \cite[Prop. 3.1]{bar:comp2} in greater generality. Actually, their argument is quite elegant and flexible: it may be extended to treat situations like
$$\int_{\R}f(x)e^{\gamma X_\gep(x)+i\beta Y_\gep(x)}\,\sigma(\dd x)$$ for general possibly correlated $X_\gep$ and $Y_\gep$ with $\sigma$ a Radon measure. In our context, their  main assumption is that   the kernel $k$ introduced in section \ref{mainsetup} has compact support. Their proof is written in dimension $1$ but clearly adapts to 
higher dimensions. However, the proofs of their paper can not be adapted to the case of long range correlated fields like Gaussian Free Fields \index{Gaussian free field (GFF)} (see section \ref{GFF}). Furthermore their tightness criterion clearly works in dimension $1$ but an extension to 
higher dimensions does not make sense. Here, we suggest to study tightness  in the space  of distributions of order $d$ as this will turn out to be important in view of the applications in Euclidean Field Theory. This step is carried out via a sharp analysis of the capacity of the involved measures (see Lemma \ref{capI_II}). Other choices of spaces for tightness may be investigated as well. 
\end{rem}

\begin{rem}
Important enough, we point out that the strategy developed in \cite{bar:comp2} 
is not robust enough to treat the frontier of phases 1 and 2. 
This point will be developed in 
Section \ref{frontier12}: we prove (Lemma \ref{capI})
that the martingale is bounded in $\bbL_p$ for some $p>1$ (whose value may not be optimal).
\end{rem}

\begin{rem}
Item 1. Theorem \ref{phase1}  describes a sufficient condition on $p\in (1,2)$ in order for the martingale to be uniformly bounded in $\bbL_p$.  We do not know if this condition is sharp as in the real case $\beta=0$. We prove a weaker statement in Proposition \ref{sharp} by proving that the condition $\xi(p)\geq d$ is necessary when $p\geq 2$.  
\end{rem}
 
The fact that for any $f$
$$\left(\gep^{\frac{\gamma^2}{2}-\frac{\beta^2} {2}}\int_{\R^d} f(x)\,M_\gep^{\gamma, \beta}(\dd x) \right)_{\gep\in[0,1]}$$ 
 is a martingale for decreasing $\gep$ simply follows from our construction of the $X_\gep$, which are sums of independent infinitesimal fields, and the 
 choice of the renormalization which guarantees that the mean is constant.
 The martingale property is not required to prove convergence in $\bbL^p$  (see for instance the circle average construction of the GFF \index{Gaussian free field (GFF)} exponential in Section \ref{GFF}) 
 but it allows to have shorter and perhaps more elegant proofs. An important step in the proof of the Theorem is the uniform 
 control of the capacity of the measure $M^{\gamma,0}_\gep$ which we prove only later in Lemma \ref{capI}.
 
\vspace{2mm}

\noindent {\it Proof of Theorem  \ref{phase1}.}  \\
{\it Item 1.} We do not follow the proof of  \cite[Prop. 3.1]{bar:comp2} as we want to give a proof that is also valid for fields with long range correlations.  Let us consider a bounded measurable function $\varphi:\R^d\to \R$ with support included in $B(0,1)$ and $p\in ]1,2[$ such that $\zeta(p)>d$. We have by Jensen's inequality:
\begin{align*}
\E\left[\Big|\gep^{\frac{\gamma^2}{2}-\frac{\beta^2} {2}}\int_{\R^d} \varphi(x)\,M_\gep^{\gamma, \beta}(\dd x)\Big|^p\right]\leq &
\E\left[\E\Big[\gep^{\gamma^2- \beta^2}\int_{B(0,1)^2} \varphi(x)\varphi(y)\,M_\gep^{\gamma, \beta}(\dd x)M_\gep^{\gamma, \beta}(\dd y)|\mathcal{F}^X\Big]^{p/2}\right]\\
\leq & C \|\varphi\|_\infty^p\E\left[\Big( \gep^{\gamma^2}\int_{B(0,1)^2}\frac{1}{|x-y|^{\beta^2}}\,M_\gep^{\gamma, 0}(\dd x)M_\gep^{\gamma,0}(\dd y)\Big)^{p/2}\right].
\end{align*}
We can then conclude with Lemma \ref{capI}.

\vspace{1mm}  
\noindent {\it Item 2.}   Let $\varphi$ be a smooth test function with support in $]0,1[^d$.  We consider the mapping
$$x=(x_1,\dots,x_d)\in\R^d\mapsto F^{\gamma,\beta}_\gep(x_1,\dots,x_d):= \gep^{ \frac{\gamma^2}{2} -\frac{\beta^2}{2}}\int_{[0,x_1]\times \dots\times[0,x_d]} e^{\gamma X_\gep(x)+ i \gb Y_\gep (x)} \,\dd x.$$
 By integration by parts, we get:
\begin{align*}
& \gep^{\frac{\gamma^2}{2}-\frac{\beta^2} {2}} \int_{ ]0,1[^d } \varphi(x) e^{\gamma X_\gep(x)+ i \gb Y_\gep (x)} \dd x  \\ 
& = (-1)^d  \gep^{\frac{\gamma^2}{2}-\frac{\beta^2} {2}} \int_{ ]0,1[^d } \frac{\partial^d \varphi(x)}{\partial x_1 \cdots \partial x_d } \left( \int_{[0,x_1] \times \cdots \times [0,x_d] } e^{\gamma X_\gep(u)+ i \gb Y_\gep (u)}\dd u_1 \cdots \dd u_d \right )  \dd x,  \
\end{align*}
where $u=(u_1, \cdots, u_d)$.
Therefore, we conclude that: 
\begin{equation*}
|M^{\gamma,\beta}_\gep ( \varphi)| \leq \sup_{x \in ]0,1[^d } \Big|\frac{\partial^d \varphi(x)}{\partial x_1 \cdots \partial x_d }\Big| \,Z_\gep,
\end{equation*} 
where:
\begin{equation*}
Z_\gep=  \int_{ ]0,1[^d } | F^{\gamma,\beta}_\gep(x_1,\dots,x_d) |  \dd x.
\end{equation*}
Observe that $(Z_\gep)_\gep $ is a positive submartingale. Furthermore, from Item 1, we deduce that
$$\E[|Z_\gep|^p]\leq \int_{ ]0,1[^d }\E[ | F^{\gamma,\beta}_\gep(x_1,\dots,x_d) | ^p] \dd x\leq C_{p,1}.$$
 
\vspace{1mm}  
\noindent {\it Item 3.}  One applies \cite[Prop 3.2]{bar:comp2}.\qed

 \begin{proposition}{\bf (Necessary conditions for $\bbL_p$ convergence for $p\geq 2$)}\label{sharp} 

If the martingale $$\big(\gep^{\frac{\gamma^2}{2}-\frac{\beta^2} {2}}  M^{\gamma,\beta}_\gep([0,1]^d)  \big)_\gep$$ is bounded in $\bbL_p$ for some $p\geq 2$ then $\xi(p)\geq d$.

 \end{proposition}

\proof

\noindent  We consider $p \geq 2$ such that $\E  [  |  M^{\gamma,\beta}([0,1]^d)  |^p ] < \infty$. We have the following inequalities:
\begin{align*}
 \E  \left[  \left| \gep^{\frac{\gamma^2}{2}-\frac{\beta^2}{2}} M^{\gamma,\beta}_\gep ([0,1]^d)  \right|^p \right]   
& = \E  \left[  \E  \left[  \left(\left|    \gep^{\frac{\gamma^2}{2}-\frac{\beta^2}{2}} M^{\gamma,\beta}_\gep ([0,1]^d)   \right|^2\right)^{p/2} \ | \ \mathcal{F}^X \right] \right]    \\
&\ge \E  \left[  \E  \left[  \left(\left|    \gep^{\frac{\gamma^2}{2}-\frac{\beta^2}{2}} M^{\gamma,\beta}_\gep ([0,1]^d)   \right|^2\right)  \ | \ \mathcal{F}^X \right]^{p/2} \right]  \\
& =  \E  \left[   \left (   \int_{[0,1]^d \times [0,1]^d} \gep^{\gamma^2} \frac{M_{\gep}^{\gamma,0}(\dd x) M_{\gep}^{\gamma,0}(\dd y)}{G_\gep(x-y)^{\beta^2}}  \right)^{p/2} \right]   \\ 
& \geq n^{d}  \E  \left[    \left(   \int_{[0,1/n]^d \times [0,1/n]^d} \gep^{\gamma^2} \frac{M_{\gep}^{\gamma,0}(\dd x) M_{\gep}^{\gamma,0}(\dd y)}{G_\gep(x-y)^{\beta^2}}  \right)^{p/2}  \right]    
\end{align*}
where the first inequality is Jensen's inequality for the conditional expectation and in the last one we have used super-additivity and stationarity. 
Now from Kahane's inequality \eqref{khan1}, we can, at the cost of a multiplicative constant, replace $X$ in the last line by the scale invariant field given by Proposition 
\ref{scalinv}. This gives
\begin{equation*}
\E  \left[  |  M^{\gamma,\beta} [0,1]  |^p \right]   \geq C n^{d-\zeta(p)} \E  \left[   \left(   \int_{[0,1]^d \times [0,1]^d}  \frac{M^{\gamma,0}(\dd x) M^{\gamma,0}(\dd y)}{|y-x|^{\beta^2}}  \right)^{p/2}  \right]  , 
\end{equation*}
for some fixed constant $C>0$. Hence we get the desired result by letting $n \to \infty$.\qed

\vspace{2mm}
  
\begin{theorem}{\bf (Multifractal spectrum)}\label{coromuzy}
We consider $p\in ]1,2]$ such that   $\zeta(p)>d$. Then   for all $q\in [0,p]$ and some constant $C_q>0$:
\begin{equation*}
\E  \left[  |  M^{\gamma,\beta}(B(0,r))  |^q \right]\underset{r \to 0}{\simeq}  C_q r^{\zeta(q)}.
\end{equation*}

\end{theorem}
  
The above result is conceptually important as it also implies that $M^{\gamma,\beta}$
is non trivial, i.e. is not deterministic.  

\begin{rem}
Note that, if $\beta \not = 0$, then $\zeta(1)=d-\frac{\beta^2}{2}<d$ hence in average $M^{\gamma,\beta}$ will have infinite variation. Therefore, we do not expect it to be a complex random measure.

\end{rem}

\begin{rem}
We stress here that   the above result is standard in the real case
$\beta=0$. Yet, in the general case $\beta\not =0$, the proof is far from straightforward. The difficulty here is that the fluctuations of the process $Y$ in the term $e^{i\beta Y_\gep}$ may cause a faster decay of $M^{\gamma,\beta}_\gep$ than expected, a kind of Riemann-Lebesgue averaging to $0$. We have to make sure that this does not happen. We further stress that this averaging to $0$ occurs outside the phase I and that the phases II and III may be seen as the study of the fluctuations along this averaging.

Furthermore, we stress that the result holds as well if we replace the ball 
$B(0,r)$ by $\varphi(\cdot/r)$ 
for some continuous and compactly supported function $\varphi$.
\end{rem}

\vspace{2mm}

\noindent {\it Proof.} We carry out the proof in dimension $1$.   
Observe  that the martingale  $\left( \gep^{\frac{\gamma^2}{2}-\frac{\beta^2} {2}}M_\gep^{\gamma, \beta}(K)\right)_{\gep>0}$ is uniformly bounded in $\bbL_p$ for all compact sets $K$. Let us fix  $r>0$. Let us consider a family of complex random measures for $r>\gep>0$: $$M^{\gamma,\beta}_{r,\gep}(\dd x)=e^{\gamma(X_{\gep}-X_r)(x)+i\beta (Y_{\gep}-Y_r)(x)-(\frac{\gamma^2}{2}-\frac{\beta^2}{2})\ln\frac{r}{\gep}}\,\dd x.$$
 Observe that $M^{\gamma,\beta}_{r,\gep}(\dd x)$ is independent of the fields 
 $X_r$ and $Y_r$ and has  the same law as $ r \left(\frac{\gep}{r}\right)^{ \frac{\gamma^2}{2}-\frac{\beta^2}{2}}M^{\gamma,\beta}_{\frac{\gep }{r}}(\dd x/r)$. 
For $\gep<r$, we can decompose $M^{\gamma,\beta}_{\gep}$ as  
\begin{equation}\label{star} 
 \gep^{\frac{\gamma^2}{2}-\frac{\beta^2}{2}}M^{\gamma,\beta}_{\gep}(\dd x)=e^{ \gamma X_r(x)+i\beta Y_r(x)-(\frac{\gamma^2}{2}-\frac{\beta^2}{2})\ln\frac{1}{r} }M^{\gamma,\beta}_{r,\gep}(\dd x).
\end{equation} 
Therefore we have:
\begin{align*}
 \E&\big[ \big| \gep^{\frac{\gamma^2}{2}-\frac{\beta^2}{2}}M^{\gamma,\beta}_\gep([0,r])\big|^q \big]\\
 =& \E\Big[ \Big|\int_{[0,r]}e^{ \gamma X_r(x)+i\beta Y_r(x) -(\frac{\gamma^2}{2} -\frac{\beta^2}{2})\ln\frac{1}{r} }M^{\gamma,\beta}_{r,\gep}(\dd x) \Big|^q  \Big]\\
 =& r^{q(\frac{\gamma^2}{2} -\frac{\beta^2}{2})-\frac{q^2\gamma^2}{2}}\E\Big[e^{ q\gamma X_r(0)-\frac{q^2\gamma^2}{2}\ln\frac{1}{r}}  \Big|\int_{[0,r]}e^{ \gamma (X_r(x)-X_r(0))+i\beta (Y_r(x) -Y_r(0)) }M^{\gamma,\beta}_{r,\gep}(\dd x) \Big|^q  \Big] \\
=& r^{q(1+\frac{\gamma^2}{2} -\frac{\beta^2}{2})-\frac{q^2\gamma^2}{2}}\E\Big[e^{ q\gamma X'_r(0)-\frac{q^2\gamma^2}{2}\ln\frac{1}{r}} \Big|\int_{[0,r]}e^{ \gamma (X'_r(x)-X'_r(0))+i\beta (Y'_r(x) -Y'_r(0)) }\Big(\frac{\gep}{r}\Big)^{ \frac{\gamma^2}{2}-\frac{\beta^2}{2}}M^{\gamma,\beta}_{\frac{\gep }{r}}(\dd x/r) \Big|^q  \Big] ,
\end{align*}
where $X'$ and $Y'$ are fields that are independent of $X$ and $Y$ with the same law (and 
hence $(r\big(\frac{\gep}{r}\big)^{ \frac{\gamma^2}{2}-\frac{\beta^2}{2}} M^{\gamma,\beta}_{\frac{\gep }{r}}(\dd x/r),X'_r,Y'_r)$ has the same law as $(M^{\gamma,\beta}_{r,\gep}(\dd x),X_r,Y_r)$ ).

Now we make a change of variables in the integral and the use the Girsanov transform to get: 
\begin{align}
 \E\big[ \big|&\gep^{\frac{\gamma^2}{2}-\frac{\beta^2}{2}}M^{\gamma,\beta}_\gep([0,r])\big|^q \big]\nonumber\\
 =& r^{\zeta(q) }\E\Big[e^{ q\gamma X'_r(0)-\frac{q^2\gamma^2}{2}\ln\frac{1}{r}}  \Big|\int_{[0,1]}e^{ \gamma (X'_r(r x)-X'_r(0))+i\beta (Y'_r( r x) -Y'_r(0)) }\Big(\frac{\gep}{r}\Big)^{ \frac{\gamma^2}{2}-\frac{\beta^2}{2}}M^{\gamma,\beta}_{\frac{\gep }{r}}(\dd x) \Big|^q  \Big]\nonumber\\
  =& r^{\zeta(q) }\E\Big[   \Big|\int_{[0,1]}e^{ \gamma Z^X_r(x)+i\beta Z^Y_r(x)-\frac{\gamma^2-\gb^2}{2}\bbE[Z^X_r(x)^2]+f_r(x)}\Big(\frac{\gep}{r}\Big)^{ \frac{\gamma^2}{2}-\frac{\beta^2}{2}}M^{\gamma,\beta}_{\frac{\gep }{r}}(\dd x) \Big|^q  \Big] \label{eron}
\end{align}
 where we have set $Z^X_r(x)=X'_r(r x)-X'_r(0)$, $Z^Y_r(x)=Y'_r(r x)-Y'_r(0)$ and 
  $$f_r(x)=\left((q-1)\gamma^2+\gb^2\right)\int_r^1\frac{k(u x)-1}{u}\,\dd u.$$ 
 
 To conclude the proof, we have to show that 
\begin{equation}
\lim_{r\to 0} \lim_{\gep\to 0} \E\Big[   \Big|\int_{[0,1]}e^{ \gamma Z^X_r(x)+i\beta Z^Y_r(x)-\frac{\gamma^2-\gb^2}{2}\bbE[Z^X_r(x)^2]+f_r(x)}\gep^{ \frac{\gamma^2}{2}-\frac{\beta^2}{2}}M^{\gamma,\beta}_{\gep}(\dd x) \Big|^q  \Big] \label{eron2}
\end{equation}
 exists and is non-zero.
We do it using a martingale argument.  The covariance kernel of $Z$ is given by,
  \begin{align}\label{covaz}
\E[Z^X_r (x)Z^X_r (x') ]=&\int_1^{1/r}\frac{k(r(x-x')u)-k(rxu)-k(rx'u)+1}{u}\,\dd u\\
=&\int_r^{1}\frac{k((x-x')v)-k(xv)-k(x'v)+1}{v}\,dv.
\end{align}
Due to this structure of covariance, one can construct a process whose marginals have the same law (we also name it $Z=(Z^X,Z^Y)$ as it causes no confusion)
indexed by $r\le 1$ such that for each $x$ and $r'<r$,
$\bbE[ Z_{r'}(x) \ | \ Z_r(x)]=Z_r(x)$. 
With this construction the process
$$A_{\gep,r}=\left(\int_{[0,1]}e^{ \gamma Z^X_r(x)+i\beta Z^Y_r(x)-
\frac{\gamma^2-\gb^2}{2}\bbE[Z^X_r(x)^2]+f_0(x)}\gep^{ \frac{\gamma^2}{2}-\frac{\beta^2}{2}}M^{\gamma,\beta}_{\gep}(\dd x) \right)_{\gep\in[0,1],r\in[0,1]}$$
is a doubly indexed martingale (note that $f_r$ has been changed to $f_0$) and thus for $q\ge 1$ we have
\begin{equation}
 \lim_{r\to 0}\lim_{\gep\to 0} \bbE[|A_{\gep,r}|^q]
=\sup_{\gep,r} \bbE[|A_{\gep,r}|^q]>0
\end{equation}
exists and we just have to show uniform boundedness of $\E[ |A_{\gep,r}|^q]$.
Let $\mathcal F^Z$ denote the sigma algebra generated by $Z$, by Jensen's inequality 
\begin{multline}
\bbE[|A_{\gep,r}|^q]\le \bbE\left[\left(\bbE\left[|A_{\gep,r}|^2 | \mathcal F^Y, \mathcal F^Z \right]\right)^{q/2}\right]\\
\le C \E\Big[\Big( \gep^{\gamma^2}\int_{[0,1]^2}\frac{1}{|x-y|^{\beta^2}}\,M_\gep^{\gamma, 0}(\dd x)M_\gep^{\gamma,0}(\dd y)\Big)^{q/2}\Big].
\end{multline}
for an appropriate constant $C$. To obtain the second inequality, we compute explicitly the average as in the proof of Theorem \ref{phase1}, averaging over $Z$ only yields an extra constant as its  covariance is uniformly bounded in $r$.
We conclude 
using Lemma \ref{capI}.
For $q<1$ we have to prove that

\begin{equation}
 \lim_{r\to 0}\lim_{\gep\to 0} \bbE[|A_{\gep,r}|^q]
=\inf_{\gep,r} \bbE[|A_{\gep,r}|^q]>0,
\end{equation}
which is true as the martingale is uniformly integrable and thus has a non trivial limit.

What remains to show is that replacing $f_0$ by $f_r$ does not change the limit. This is easy: if $\tilde A_{\gep,r}$ denotes 
the process where $f_0$ is replaced by $f_r$ we obtain after redoing the same computation with an extra $e^{f_r-f_0}-1$ factor that
\begin{equation}
 \bbE\left[|A_{\gep,r}-\tilde A_{\gep,r}|^q\right]=o(1)
 \E\Big[\Big( \gep^{\gamma^2}\int_{[0,1]^2}\frac{1}{|x-y|^{\beta^2}}\,M_\gep^{\gamma, 0}(\dd x)M_\gep^{\gamma,0}(\dd y)\Big)^{q/2}\Big]
\end{equation}
when $r$ tends to zero, and conclude.\qed

We also have:
\begin{theorem}\label{corostar}{\bf (Star scale invariance).}
Assume that the kernel $k$ is of class $C^{2d}$ with  derivatives of order $2d$ H\"older. The distribution of order $d$ $M^{\gamma,\beta}$ is star scale invariant \index{star scale invariant} in the sense that it can be written as
$$M^{\gamma,\beta}(\dd x)=e^{ \gamma X_r(x)+i\beta Y_r(x)-(\frac{\gamma^2}{2}-\frac{\beta^2}{2})\ln\frac{1}{r} }\tilde{M}^{\gamma,\beta}_{r}(\dd x)$$
where $\tilde{M}^{\gamma,\beta}_{r}$ is a distribution of order $d$, is independent of the fields $X_r$ and $Y_r$ and has  the same law as $ r^d M^{\gamma,\beta}(\dd x/r)$. 
\end{theorem}

\begin{rem}
The scaling relation \eqref{star} or that of Theorem \ref{corostar} is only valid for star scale invariant kernels as those described in section \ref{mainsetup}. If one wishes to apply this argument to more general situations than those described in sections \ref{extension} or \ref{GFF}, the   difference is that 
we have a decomposition
$$ M^{\gamma,\beta}(\dd t)=r^d e^{\gamma X_r(t)-
(\frac{\gamma^2}{2}-\frac{\beta^2}{2})\E[X_r(t)^2]}\overline{M}^{\gamma,\beta,r}\left(\frac{\dd t}{r}\right)$$
where the measure $\overline{M}^{\gamma,\beta,r}$ is independent of $X_r,Y_r$ and the family of complex valued distributions $(\overline{M}^{\gamma,\beta,r})_r$ is tight in the space of distributions of order $d$.
\end{rem}

\noindent {\it Proof.}  We stick to the notations  of the beginning of the proof of Theorem \ref{coromuzy} and write
\begin{equation}\label{starbis} 
 \gep^{\frac{\gamma^2}{2}-\frac{\beta^2}{2}}M^{\gamma,\beta}_{\gep}(\dd x)=e^{ \gamma X_r(x)+i\beta Y_r(x)-(\frac{\gamma^2}{2}-\frac{\beta^2}{2})\ln\frac{1}{r} }M^{\gamma,\beta}_{r,\gep}(\dd x).
\end{equation} 
where $M^{\gamma,\beta}_{r,\gep}(\dd x)$ is independent of the fields $X_r$ and $Y_r$ and has  the same law as $ r^d \left(\frac{\gep}{r}\right)^{ \frac{\gamma^2}{2}-\frac{\beta^2}{2}}M^{\gamma,\beta}_{\frac{\gep }{r}}(\dd x/r)$. From item 2 of Theorem \ref{phase1}, the left-hand side converges in the sense of distributions of order $d$ towards $M^{\gamma,\beta}(\dd x)$. Concerning the right-hand side,  for all $r>0$, 
the family $ r^d \left(\frac{\gep}{r}\right)^{ \frac{\gamma^2}{2}-\frac{\beta^2}{2}}M^{\gamma,\beta}_{\frac{\gep }{r}}(\dd x/r)$ converges in law as $\epsilon\to 0$ in the  sense of distributions of order $d$ towards $M^{\gamma,\beta}_r(\dd x)$, which has the same law as $r^dM^{\gamma,\beta}(\dd x/r)$ and is independent of $X_r,Y_r$. Because of our assumption on the regularity of $k$, we can apply Proposition \ref{kolm2} to prove that both processes $X_r,Y_r$ are almost surely of class $C^d$. We can then pass to the limit in \eqref{starbis} to complete the proof of Theorem \ref{corostar}.\qed

\subsubsection{Analysis of the capacity}
The following lemma settles the case where $(\gamma,\beta)$ belongs to the inner phase I.
Recall that this implies the existence of $p \in (1,2)$  such that $\zeta(p)>d$. 

\begin{lemma}\label{capI}

Let $(\gamma,\beta)$ belong to the inner phase I. If $p \in (1,2]$  is such that $\zeta(p)>d$, there exists $C>0$ such that we have for all $\gep<1$:
\begin{equation*}
\E  \left [ \left ( \int_{ x,y \in [0,1]^{ d}} \gep^{\gamma^2} \frac{M_{\gep}^{\gamma,0}(\dd x) M_{\gep}^{\gamma,0}(\dd y)}{|y-x|^{\beta^2}}  \right ) ^{p/2} \right ]  \leq C. 
\end{equation*}

\end{lemma}

\proof
For simplicity, we suppose that $d=1$. From Proposition \ref{kahaha}, \eqref{froomk} and \eqref{crook}, it is sufficient to prove the result when
$X_{\gep}(x)$ is the scale invariant Gaussian log-correlated field described in Proposition \ref{scalinv} (we apply \eqref{khan1} to the field $X_\gep(x)+X_{\gep}(y)$ indexed by $(x,y) \in \bbR^2$) with $T=2$. 

Now, by subadditivity of $x \mapsto x^{p/2}$, we get:
\begin{align}\label{spiderman}  
 \E  \left [ \left ( \int_{ x,y \in [0,1]} \left(\frac{\gep}{2}\right)^{\gamma^2} \frac{M_{\gep/2}^{\gamma,0}(\dd x) M_{\gep/2}^{\gamma,0}(\dd y)}{|y-x|^{\beta^2}}  \right ) ^{\frac{p}{2}} \right ]    \leq &2 \E  \left [ \left ( \int_{ x,y \in [0,\frac{1}{2}]} \left(\frac{\gep}{2}\right)^{\gamma^2} \frac{M_{\gep/2}^{\gamma,0}(\dd x) M_{\gep/2}^{\gamma,0}(\dd y)}{|y-x|^{\beta^2}}  \right ) ^{\frac{p}{2}} \right ]  \\
 + &2 \E  \left [ \left ( \int_{x \in [0,\frac{1}{2}],y\in [\frac{1}{2},1]} \left(\frac{\gep}{2}\right)^{\gamma^2} \frac{M_{\gep/2}^{\gamma,0}(\dd x) M_{\gep/2}^{\gamma,0}(\dd y)}{|y-x|^{\beta^2}}  \right ) ^{\frac{p}{2}} \right ].
 \end{align}
We first handle the second term in the above sum. We have by Jensen's Inequality that:
\begin{align*}
 \E  \left [ \left ( \int_{ (x,y) \in [0,\frac{1}{2}]\times [\frac{1}{2},1]} \left(\frac{\gep}{2}\right)^{\gamma^2} \frac{M_{\gep/2}^{\gamma,0}(\dd x) M_{\gep/2}^{\gamma,0}(\dd y)}{|y-x|^{\beta^2}}  \right ) ^{\frac{p}{2}} \right ]  
& \leq C \left ( \int_{ (x,y) \in [0,\frac{1}{2}]\times [\frac{1}{2},1]}  \frac{\dd x \dd y}{|y-x|^{\gamma^2+\beta^2}} \right )^{\frac{p}{2}}   \\
& \leq C \int_{0}^1 \frac{\dd u}{u^{\gamma^2+\beta^2}} \int_0^u dv  .
\end{align*}
This latter quantity is finite since $\gamma^2+\beta^2 < 2$. Hence, we get the existence of some constant $C>0$ such that:
\begin{equation}\begin{split}\label{eqintermediaire} 
& \E  \left [ \left ( \int_{ x,y \in [0,1] } \left(\frac{\gep}{2}\right)^{\gamma^2} \frac{M_{\gep/2}^{\gamma,0}(\dd x) M_{\gep/2}^{\gamma,0}(\dd y)}{|y-x|^{\beta^2}}  \right ) ^{\frac{p}{2}} \right ]       \leq 2 \E  \left [ \left ( \int_{ x,y \in [0,\frac{1}{2}] } \left(\frac{\gep}{2}\right)^{\gamma^2} \frac{M_{\gep/2}^{\gamma,0}(\dd x) M_{\gep/2}^{\gamma,0}(\dd y)}{|y-x|^{\beta^2}}  \right ) ^{\frac{p}{2}} \right ] +C.    
\end{split}\end{equation}
By stochastic scale invariance \eqref{sssscal} for $\gl=1/2$, we have
\begin{align*}
&  \E  \left [ \left ( \int_{ (x,y) \in [0,\frac{1}{2}]^{2}} \left(\frac{\gep}{2}\right)^{\gamma^2} \frac{M_{\gep/2}^{\gamma,0}(\dd x) M_{\gep/2}^{\gamma,0}(\dd y)}{|y-x|^{\beta^2}}  \right ) ^{p/2} \right ]  \\
&   = \frac{1}{2^{ p(1+\gamma^2/2-\beta^2/2)  }} \E  \left [ \left ( \int_{ (x,y) \in [0,1]^{2}} \gep^{\gamma^2} \frac{e^{\gamma X_{\gep/2}(x/2)+ \gamma X_{\gep/2}(y/2)}  \dd x \dd y}{|y-x|^{\beta^2}}  \right ) ^{p/2} \right ] \\
&  = \frac{1}{2^{ p(1+\gamma^2/2-\beta^2/2)  }} \E  [e^{p \gamma \Omega_{1/2}}]    \E  \left [ \left ( \int_{ (x,y) \in [0,1]^{2}} \gep^{\gamma^2} \frac{e^{\gamma X_{\gep}(x)+ \gamma X_{\gep}(y)}  \dd x \dd y}{|y-x|^{\beta^2}}  \right ) ^{p/2} 
\right ] \\
&  = \frac{1}{2^{ \zeta(p)  }}   \E  \left [ \left ( \int_{ (x,y) \in [0,1]^{2}} \gep^{\gamma^2} \frac{e^{\gamma X_{\gep}(x)+ \gamma X_{\gep}(y)}  \dd x \dd y}{|y-x|^{\beta^2}}  \right ) ^{p/2} \right ] .
\end{align*} 
If we set $$u(\gep)= \E  \left [ \left ( \int_{ (x,y) \in [0,1]^{2}} \gep^{\gamma^2} \frac{e^{\gamma X_{\gep}(x)+ \gamma X_{\gep}(y)}  \dd x \dd y}{|y-x|^{\beta^2}}  \right ) ^{p/2} \right ]$$
 then inequality (\ref{eqintermediaire}) amounts to

 \begin{equation}\label{grossier}
 u(\gep/2) \leq \rho u(\gep)+C
 \end{equation}
  where $\rho= \frac{2}{2^{ \zeta(p)  }}  <1$.
 Notice that repeating the computation starting from \eqref{spiderman} 
 with $\alpha\in [1/2,1)$ entails
 \begin{equation}\label{grossier2}
 \forall \alpha \in [1/2,1), \; u(\alpha\gep) \leq \rho u(\gep)+C
 \end{equation}
 Hence by a trivial induction 
 the sequence $(u(2^{-n}))_{n \geq 1}$ is bounded and \eqref{grossier2} allows to conclude for general $\gep$.

 \qed
\subsection{Phase transition I/II}\label{frontier12}

\begin{theorem}\label{th:frontier12} 
Let us consider the frontier I/II: $\beta+\gamma=\sqrt{2d}$ and 
$\gamma \in \left( \sqrt{ \frac{d}{2}} , \sqrt{2d}\right)$. We further consider any $p\in 
\left(1, \frac{\sqrt{2d}}{\gamma}\right)$.

\noindent 1.\,\,\, For all compactly supported   bounded measurable functions $f$, 
the martingale
\begin{equation*} 
 ( \gep^{\frac{\gamma^2}{2}-\frac{\beta^2} {2}}\int_{\R^d} f(x)\,M_\gep^{\gamma, \beta}(\dd x) )_\gep
\end{equation*}
is uniformly bounded in $\bbL_p$. Furthermore, for all $R>0$, there exists a constant $C_{p,R}$ (only depending on $p,R$) such that for all bounded measurable functions $f$ with compact support in $B(0,R)$:
$$\E\Big[\sup_{\gep\in ]0,1]}\Big| \gep^{\frac{\gamma^2}{2}-\frac{\beta^2} {2}}\int_{\R^d} f(x)\,M_\gep^{\gamma, \beta}(\dd x)\Big|^p\Big]\leq C_{p,R}\|f\|_\infty^p.$$
\noindent 2.\,\,\,  The $\mathcal{D}'(\R^d)$-valued martingale:
\begin{equation*} 
M^{\gamma,\beta}_\gep: \varphi \rightarrow  \gep^{\frac{\gamma^2}{2}-\frac{\beta^2} {2}}\int_{\R^d} \varphi(x) e^{\gamma X_\gep(x)+ i \gb Y_\gep (x)} \dd x
\end{equation*}
converges almost surely in the space $\mathcal{D}'_{d}(\R^d)$ of distributions of order $d$ towards a non trivial limit $M^{\gamma,\beta}$. More precisely, for each $R>0$, there exists a random variable $Z_R\in\bbL_{p}$ such that for all functions $\varphi\in C^d_c(B(0,R))$:
$$|M^{\gamma,\beta}(\varphi)|\leq Z_R \sup_{x \in B(0,R) } |\frac{\partial^d \varphi(x)}{\partial x_1 \cdots \partial x_d }|  .$$
\noindent 3.\,\,\, For all $p<\frac{\sqrt{2d}}{\gamma} $, we set $\zeta(p)=(d+ \frac{\gamma^2}{2}-\frac{\beta^2} {2})p -\frac{\gamma^2}{2}p^2= \sqrt{2d} \gamma p -\frac{\gamma^2}{2}p^2$. There exists some  constant $C_p>0$ such that:
\begin{equation*}
\E  [  |  M^{\gamma,\beta}(B(0,r))  |^p ] \underset{r \to 0}{\simeq} C_p r^{\zeta(p)}.
\end{equation*}

\noindent 4.\,\,\, Assume that the kernel $k$ is of class $C^{2d}$ with  derivatives of order $2d$ H\"older. The distribution of order $d$, $M^{\gamma,\beta}$ is star scale invariant \index{star scale invariant} in the sense that it can be written as
$$M^{\gamma,\beta}(\dd x)=e^{ \gamma X_r(x)+i\beta Y_r(x)-(\frac{\gamma^2}{2}-\frac{\beta^2}{2})\ln\frac{1}{r} }\tilde{M}^{\gamma,\beta}_{r}(\dd x)$$
where $\tilde{M}^{\gamma,\beta}_{r}$ is a distribution of order $d$, independent of the fields $X_r$ and $Y_r$ and has  the same law as $ r M^{\gamma,\beta}(\dd x/r)$. 
 \end{theorem}

\begin{rem}\label{rem:frontier}
 In this case, $\zeta$ is increasing on $]0, \frac{\sqrt{2d}}{\gamma}[$ with $\zeta(\frac{\sqrt{2d}}{\gamma})=d$ and $\zeta'(\frac{\sqrt{2d}}{\gamma})=0$. Hence the continuity of $M^{\gamma,\beta}$ in dimension 1 is not obvious.  
 \end{rem}
 
\vspace{1mm}
\noindent {\it Proof.} The argument for item 1 and 2 is the same as in the proof of Theorem \ref{phase1} except that we use Lemma \ref{capI_II} below  instead of Lemma \ref{capI} (indeed, we can choose $p$ such that  $\alpha=p/2\in ]\frac{\sqrt{2d}}{3 \gamma} , \sqrt{\frac{d}{2}} \frac{1}{\gamma}[$ and $3\alpha \gamma/\sqrt{2d}>1$.). Items 3 and 4 are proved as in Theorem \ref{coromuzy} and  \ref{corostar}.\qed

\subsubsection{Analysis of the capacity}
The following lemma settles the case 
$\beta+\gamma=\sqrt{2d}$ and $\gamma \in \left(\sqrt{\frac{d}{2}}, \sqrt{2d}\right)$:

\begin{lemma}\label{capI_II}

Let $\gamma \in \left(\sqrt{\frac{d}{2}}, \sqrt{2d}\right)$. Let $l \geq 0$. 
For all $\alpha \in  \left(\frac{\sqrt{2d}}{3 \gamma} , \sqrt{\frac{d}{2}} \frac{1}{\gamma}\right)$, there exists $C>0$ such that for all $\gep, \gep' \leq \frac{1}{2^l}$:  
\begin{equation*}
\E  \left [ \left ( \int_{ x \in [0,1]^d; \: |y-x| \leq \frac{1}{2^l}} (\gep')^{\gamma^2/2} \gep^{\gamma^2/2} \frac{M_{\gep'}^{\gamma,0}(\dd x) M_{\gep}^{\gamma,0}(\dd y)}{|y-x|^{(\sqrt{2d}-\gamma)^2}}  \right ) ^{\alpha} \right ]  \leq C  \sum_{j \geq l}  \frac{1}{j^{\frac{3}{\sqrt{2d}} \alpha \gamma }} .
\end{equation*}

\end{lemma}

\proof
For simplicity, we suppose that $\gep=\gep_k=\frac{1}{2^k} $ and $\gep'=\gep_{k'}=\frac{1}{2^{k'}} $ with $k \leq k'$. 
 We set $X_k(x)=X_{\gep_k}(x) $ ($X_{k'}(x)=X_{\gep_{k'}}(x) $), 
 $M_{k}( \dd x)=M_{\gep_k}^{\gamma,0}( \dd x)$ ( $M_{k'}( \dd x)=M_{\gep_{k'}}^{\gamma,0}( \dd x)$)
 and $ \mathcal{F}_k=\mathcal{F}_{\gep_k}$ ($ \mathcal{F}_{k'}=\mathcal{F}_{\gep_{k'}}$). 
 We further define 
 $$A_j=\{(x,y)\in ([0,1]^d)^2;|x-y|\leq 2^{-j}\} \quad \quad D_j=\{(x,y)\in ([0,1]^d)^2;2^{-j}<|y-x|\le 2^{-j+1}\} $$
 We have:
\begin{align*}
   \int_{A_l} \frac{1}{2^{k \gamma^2/2}}  \frac{1}{2^{k' \gamma^2/2}} \frac{M_{k}(\dd x) M_{k'}(\dd y)}{|y-x|^{(\sqrt{2d}-\gamma)^2}}  
  \leq & C \sum_{j=l+1}^k  2^{j (\sqrt{2d}-\gamma)^2}  \int_{D_j}  \frac{1}{2^{k \gamma^2/2}} \frac{1}{2^{k' \gamma^2/2}}M_{k}(\dd x) M_{k'}(\dd y) \\
& +   \int_{ A_k} \frac{1}{2^{k \gamma^2/2}} \frac{1}{2^{k' \gamma^2/2}} \frac{M_{k}(\dd x) M_{k'}(\dd y)}{|y-x|^{(\sqrt{2d}-\gamma)^2}} .
\end{align*}
Therefore, if $\alpha<1$, by subadditivity we get the following inequality:
\begin{align}
   \left ( \int_{  A_l } \frac{1}{2^{k \gamma^2/2}} \frac{1}{2^{k' \gamma^2/2}} \frac{M_{k}(\dd x) M_{k'}(\dd y)}{|y-x|^{(\sqrt{2d}-\gamma)^2}} \right ) ^\alpha \nonumber 
  \leq &C \sum_{j=l+1}^k  2^{ \alpha j (\sqrt{2d}-\gamma)^2}  \left ( \int_{D_j} \frac{1}{2^{k \gamma^2/2}} \frac{1}{2^{k' \gamma^2/2}} M_{k}(\dd x) M_{k'}(\dd y) )\right )^\alpha \nonumber  \\
& +  \left (   \int_{A_k} \frac{1}{2^{k \gamma^2/2}} \frac{1}{2^{k' \gamma^2/2}} \frac{M_{k}(\dd x) M_{k'}(\dd y)}{|y-x|^{(\sqrt{2d}-\gamma)^2}} \right ) ^{\alpha} .\label{leporc}
\end{align}
Now, we estimate each quantity in the sum on the right hand side. We start with the last term above.
We have by Jensen's inequality and the martingale property applied to $\frac{1}{2^{k' \gamma^2/2}} M_{k'}(\dd y)$
\begin{align*}
 \E \left [ \left (   \int_{ A_k} \frac{1}{2^{k \gamma^2/2}} \frac{1}{2^{k' \gamma^2/2}} \frac{M_{k}(\dd x) M_{k'}(\dd y)}{|y-x|^{(\sqrt{2d}-\gamma)^2}} \right ) ^{\alpha}  \right ]  
  \leq  &\E  \left [ \left (  \E \left [  \int_{A_k} \frac{1}{2^{k \gamma^2/2}} \frac{1}{2^{k' \gamma^2/2}} \frac{M_{k}(\dd x) M_{k'}(\dd y)}{|y-x|^{(\sqrt{2d}-\gamma)^2}} | \mathcal{F}_k   \right ] \right ) ^{\alpha}  \right ]  \\
  =&  \E \left [ \left (   \int_{ A_k} \frac{1}{2^{k \gamma^2}} \frac{M_{k}(\dd x) M_{k}(\dd y)}{|y-x|^{(\sqrt{2d}-\gamma)^2}} \right ) ^{\alpha}  \right ]  .
\end{align*}
Now we use the inequality $ab\le a^2+b^2/2$ and obtain (for some constant $C(\gamma,d)$)
\begin{align*}
  \E \left [ \left (   \int_{A_k} \frac{1}{2^{k \gamma^2/2}} \frac{1}{2^{k' \gamma^2/2}} \frac{M_{k}(\dd x) M_{k'}(\dd y)}{|y-x|^{(\sqrt{2d}-\gamma)^2}} \right ) ^{\alpha}  \right ]   
& \leq \E \left [ \left (   \int_{ A_k} \frac{1}{2^{k \gamma^2}} \frac{e^{2 \gamma X_k(x)} \dd x\dd y}{|y-x|^{(\sqrt{2d}-\gamma)^2}} \right ) ^{\alpha}  \right ]  \\
& \le C 2^{\alpha k ((\sqrt{2d}-\gamma)^2-d)  } \E \left [ \left (   \int_{[0,1]^d} \frac{1}{2^{k \gamma^2}}  e^{2 \gamma X_k(x)}  \dd x \right ) ^{\alpha}  \right ]  \\ 
 & = C \E \left [ \left (   \int_{[0,1]^d} 2^{d k}  e^{2 \gamma (X_k(x)-\sqrt{2d} \ln 2^k  )  }  \dd x \right ) ^{\alpha}  \right ].
 \end{align*}

To conclude the proof we need the following result which is a refined version of  \cite[Theorem 1.6]{HuShi} coming from  \cite[Prop 2.1]{madaule} or \cite[Lemma 9]{basic}, the proof of which is postponed to the end of this section.

\begin{lemma}\label{superhushi}
For any $ \gamma>\sqrt{\frac{d}{2}}$ and any $\alpha< \sqrt{\frac{d}{2}} \frac{1}{\gamma}$, there exists $C$ (depending on $\gamma$ and $\alpha$) such that for all $\gep$ we have
\begin{equation}\label{cheeez}
\bbE\left[  \left (\int_{[0,1]^d} \gep^{-d} e^{2 \gamma \left(X_{\gep}(x)-\sqrt{2d} \log \frac{1}{\gep}  \right) } \dd x  \right ) ^\alpha \right]\le \frac{C}{|\log \gep|^{\frac{3}{\sqrt{2d}} \alpha \gamma }}.
\end{equation}
\end{lemma}

Let us admit this lemma for a while. We get the bound
\begin{equation*}
\E \left [ \left (   \int_{[0,1]^d} 2^{d k}  e^{2 \gamma (X_k(x)-\sqrt{2d} \ln 2^k  )  }  \dd x \right ) ^{\alpha}  \right ] \leq \frac{C}{k^{\frac{3}{\sqrt{2d}} \alpha \gamma }} .
\end{equation*}

Now, we handle the terms in the sum \eqref{leporc}. Without loss of generality, we may assume that the kernel $k$ appearing in Assumption (A) has a compact support included in the ball $B(0,1)$. Indeed if not, we can use Proposition \ref{kahaha} and \eqref{froomk} to get  a comparison with the compactly supported case. In particular, we will use the fact that, conditionally to $\mathcal{F}_j$, the sigma algebras $\sigma\{M_k(A);A\subset B_1\}$ and $\sigma\{M_k(A);A\subset B_2\}$ are independent as soon as ${\rm dist}(B_1,B_2)>2^{-j}$.
Thus, using first a conditional Jensen (recall $\alpha<1$) and then conditional independence and the martingale property, we have
\begin{align*}
  \E \left [   \left ( \int_{D_j}   \frac{1}{2^{k \gamma^2}} M_{k}(\dd x) M_{k}(\dd y) )\right )^\alpha \right ]  & \leq \E \left [ \left ( \E [ \int_{D_j}   \frac{1}{2^{k \gamma^2}} M_{k}(\dd x) M_{k}(\dd y) )  | \mathcal{F}_j  ]\right )^\alpha \right ]  \\ 
& =\E \left [   \left ( \int_{D_j}   \frac{1}{2^{j \gamma^2}} M_{j}(\dd x) M_{j}(\dd y) )\right )^\alpha \right ]   \\
& \leq   \E \left [   \left ( \int_{ |y-x| < \frac{1}{2^{j-1}}}   \frac{1}{2^{j \gamma^2}} M_{j}(\dd x) M_{j}(\dd y) )\right )^\alpha \right ]   .
\end{align*}
Now, one can conclude similarly to the previous term. \qed
 
\begin{proof}[Proof of Lemma \ref{superhushi}] We assume that $\gep=\frac{1}{2^k}$ (this brings no loss of generality  since by Proposition \ref{kahaha}, one can compare the l.h.s. of \ref{cheeez} for two $\gep$ within a factor of $2$). 

For $r\in \bbN$ let us cover the interval $[0,1]^d$ by the dyadic cubes of the form: 
\begin{equation*}
I^{(r)}_{\bf z}:= \prod_{1 \leq i \leq d} \left[ \frac{z_i}{2^r}+\frac{z_i+1}{2^r}\right[,
\end{equation*}
where ${\bf z}=(z_1, \cdots, z_d) \in \{0,\dots, 2^r-1\}^d$. 
Let $(Y^{(k)}(x))_{x \in  [0,1]^d}$ be a standard Gaussian $2^d$-adic cascade defined on $[0,1]^d$ by the covariance function 
\begin{equation*}
  \E[ Y^{(k)}(x) Y^{(k)}(y) ]:= \ln \frac{1}{d_{2}(x,y)\vee 2^{-k}},
\end{equation*}
where $d_{2}(x,y)$ is the  dyadic distance defined on $[0,1]^d$, i.e. $d_{2}(x,y)$ is an inverse power of $2$ and we have

$$ d_{2}(x,y) \le 2^{-r} \quad \Leftrightarrow \quad  \left(\exists {\bf z}, \ x,y\in I^{(r)}_{\bf z} \right).$$

Note that the process $Y$ is a particular case of branching random walk (studied in e.g. \cite{HuShi, madaule}) with 
i.i.d Gaussian step, and deterministic branching number $2^d$.

From \eqref{froomk}, there exists $C$ such that 
\begin{equation*}
\E[ X_k(x) X_k(y)  ] +C \geq  \: \E[ Y^{(k)}(x) Y^{(k)}(y) ].
\end{equation*}  
 Hence, by Proposition \ref{kahaha}, we get that there exists some $C>0$ such that:
 \begin{align*}
  \E \left [ \left (   \int_{[0,1]^d} 2^{dk} e^{2 \gamma \left (X_k(x) -\sqrt{2d} \ln 2^k  \right)}  \dd x \right ) ^{\alpha}  \right ]  &  \leq C \E \left [ \left (   \int_{[0,1]^d} 2^{dk}  e^{2 \gamma \left( Y^{(k)}(x)  -\sqrt{2d} \ln 2^k  \right)}  \dd x \right ) ^{\alpha}  \right ]    \\
&= C \E\left[  \left( \sum_{{\bf z}\in \{0,\dots, 2^k-1\}^d} e^{  2 \gamma  \left(Y^{(k)}(2^{-k}{\bf z}) - \sqrt{2d} \ln 2^k \right) }     \right)^\alpha   \right] .
\end{align*}
Now, by using  \cite[Prop 2.1]{madaule} or \cite[Lemma 9]{basic}, we know that for all $\beta>1$ and $\alpha<\frac{1}{\beta}$:
\begin{equation*}
 \E\left[  \left( \sum_{{\bf z}\in \{0,\dots, 2^k-1\}^d} e^{  \gb\left(\sqrt{2d}Y^{(k)}(2^{-k}{\bf z}) - 2d \ln 2^k\right)}     \right)^\alpha   \right] \leq \frac{C}{k^{\frac{3}{2} \alpha \beta }}
\end{equation*}
Now, recall that $\gamma>\sqrt{\frac{d}{2}}$ and hence for all $\alpha < \sqrt{\frac{d}{2}} \frac{1}{\gamma}$:
\begin{equation*}  
\E\left[  \left( \sum_{{\bf z}\in \{0,\dots, 2^k-1\}^d} e^{  2 \gamma  \left ( Y^{(k)}(2^{-k}{\bf z}) - \sqrt{2d} \ln 2^k \right ) }     \right)^\alpha   \right]   \leq \frac{C}{k^{\frac{3}{\sqrt{2d}} \alpha \gamma }}.
\end{equation*}
\end{proof}

\section{Study of the phase III and its I/III and II/III boundaries }\label{phaseIII}

\subsection{Results and organization of the proofs} 

In the inner phase III, the limit object describing the renormalized measure is a complex white-noise whose intensity 
depends on $X$ and is given by $M^{2 \gamma,0}(\dd x)$ (all the information about $Y$ is lost in the limit).

\medskip

The frontiers I/III and II/III present similar behavior but there are additional technical difficulties especially for the frontier II/III. Indeed, in that case, the intensity measure $M^{2 \gamma,0}(\dd x)$ is the so-called derivative martingale (see \cite{Rnew7,Rnew12}).

Now, we define the function $\sigma^2$ that will appear throughout the following convergence results
\begin{definition}{\bf Function $\sigma^2$.}\label{sigmasquare}
In dimension $d$, we define the function of the parameters $(\gamma,\beta)$:
\begin{equation}
\sigma^2(\gb^2+\gamma^2) :=\left\{\begin{array}{lll}
&\int_{\R^d}\exp\Big(-(\gamma^2+\beta^2)\int_0^1\frac{1-k(uz)}{u}\,\dd u\Big)\,\dd z & \text{if }\gb^2+\gamma^2>d\\
&\int_{z\in \bbS^{d-1}}  \nu_{d-1}(\dd z) \exp\left(\int_0^\infty \frac{k(uz)-\ind_{[0,1]}(u)}{u} \dd u \right) &   \text{if }\gb^2+\gamma^2=d
 \end{array}\right.
\end{equation}
where $\nu_{d-1}(z)$ denotes the surface measure on $\bbS^{d-1}$ (the $d-1$ dimensional unit sphere). 
\end{definition}
Observe that in dimension $1$, for $\beta^2+\gamma^2=1$, $\sigma^2$ takes the simpler form $$
\sigma^2(1)=2\exp\left(\int_0^\infty \frac{k(u)-\ind_{[0,1]}(u)}{u} \dd u \right).$$

Note that since $k$ decays to zero at infinity sufficiently fast (cf. Assumption ({\bf A.2}))
and is Lipshitz around zero we get

\begin{equation}
 \int_0^1\frac{1-k(uz)}{u}\,du \stackrel{z \to \infty}{\sim} \log |z|.
\end{equation}
This guaranties that $\sigma$ is well defined for $\gb^2+\gamma^2>d$.

We are now in position to state the main results of this section:
\begin{theorem}\label{th3}
\begin{itemize}
\item  When $\gamma\in[0,\sqrt{\frac{d}{2}}[$ and $\gb^2+\gamma^2>d$, we have
 \begin{equation}
\left(\gep^{\gamma^2-\frac{d}{2}} M^{\gamma,\gb}_\gep(A)\right)_{A\subset \R^d} \Rightarrow \left(W_{\sigma^2 M^{2\gamma,0}}(A)\right)_{A\subset \R^d}.
\end{equation}
where  $\sigma^2=\sigma^2(\gb^2+\gamma^2)$ and $W$ is a standard complex Gaussian measure on $\R^d$ with intensity $\sigma^2 M^{2\gamma,0}$. The above convergence holds in the sense of convergence in law of the finite dimensional distributions.
\item When $\gamma\in[0,\sqrt{\frac{d}{2}}[$ and $\gb^2+\gamma^2=d$, we have
 \begin{equation}
\left(\gep^{\gamma^2-\frac{d}{2}}|\log \gep|^{-1/2} M^{\gamma,\gb}_\gep(A)\right)_{A\subset \R^d} \Rightarrow \left(W_{\sigma^2 M^{2\gamma,0}}(A)\right)_{A\subset \R^d}.
\end{equation}
where  $\sigma^2=\sigma^2(d)$ and $W$ is a standard complex Gaussian measure on $\R^d$ with intensity $\sigma^2 M^{2\gamma,0}$. The above convergence holds in the sense of convergence in law of the finite dimensional distributions.
\end{itemize}

\end{theorem}

Concerning the frontier II/III, we will need to use the results in \cite{Rnew7,Rnew12}. Further assumptions must be made in order to make sure that one can construct the derivative martingale and prove the Seneta-Heyde renormalization  (see \cite[Section D and Remark 31]{Rnew12}):
\begin{assB}  We consider the following assumptions:
\begin{description}\label{assB}
\item[B1.] $k$ is  nonnegative,
\item[B2.] $k$ has compact support,
\item[B2'.] $k$ admits a nonnegative convolution square root $g$, i.e. $k(x)=\int_{\R^d}g(y)g(x+y)\,\dd y$, such that $g$ is nonnegative, $g$ and $\partial g$ are integrable, and for some constants $C>0$, $\alpha>1$
$$\sup_{|x|\geq 2}g(x)+|\partial g(x)|<+\infty,  \,\,\, \int_1^{\infty}v^{-1}\int_{B(0,\ln^\alpha v)}g^2(y)\,\dd y\,dv<\infty,  \,\,\,   k(x)+|\partial k(x)|\leq Ce^{-  |x|^{1/\alpha}}.$$
\end{description}
\end{assB}
 
\begin{theorem}\label{thfrontier23}
Assume that $k$ satisfies Assumptions (A), B1 and either B2 or B2', then $\gamma= \sqrt{d/2}$ and $\gb^2+\gamma^2>d$, we have
 \begin{equation}
\left((-\ln \gep)^{1/4} M^{\gamma,\gb}_\gep(A)\right)_{A\subset\R^d} \Rightarrow \left(W_{\sigma^2 M'}(A)\right)_{A\subset\R^d}.
\end{equation}
with  $$\sigma^2=\sqrt{\frac{2}{\pi}} \sigma^2(\gb^2+d/2),$$
and the law of $W_{\sigma^2 M'}(\cdot)$ is, conditionally to $X$, that of a complex Gaussian random measure with intensity $\sigma^2 M' $.   
The above convergence holds in the sense of convergence in law of the finite dimensional distributions.

In dimension $1$, 
the law of $(W_{\sigma^2 M'}([0,t]))_{t\ge 0}$ coincides with $(B_{\sigma^2 M'([0,t])})_t$, where $B$ is a   complex  Brownian motion  independent of $X$.
\end{theorem}

 To carry out the proof, we compute the limit of the moments conditionally to $X$ and prove that they match those of the prescribed Gaussian random measure. 
 The first important step is to identify the conditional second moment $\bbE\left[ \left|M^{\gamma,\gb}_\gep(A)\right|^2 \ |  \ X\right]$ in each case. This work is done in Sections \ref{slesmomentos} to
 \ref{frontier23}.
 
 \medskip
 
 Then in a second step, we conclude the proof of the above theorems in Section \ref{twoksec}. The main idea is to consider the higher order moments and to show that they are those of a Gaussian variable. The main technical material on these moments is gathered in appendix \ref{Controlmom}.
 
\begin{rem}
In dimension 1, it is enough to assume $k$ to be twice differentiable on some interval $[0,\delta]$ for some $\delta>0$. The left and right derivative at $0$ need not be the same. For instance, one can consider the kernel $k(u)=(1-|u|)_+$ in dimension $1$.
\end{rem}

\subsection{The second moments}\label{slesmomentos}
 
Let us first focus on computing the second moment.
We want to prove the following: 

\begin{proposition}\label{lesmomentos}
 Let $A$ be some compact ball with a non-empty interior.  \\
$\bullet$ When $\gamma^2+\gb^2>d$, $\gamma<\sqrt{\frac{d}{2}}$,
  we have the following $\bbL_1$-convergence
\begin{equation}\label{supercritical}
    \lim_{\gep\to 0}\bbE\left[\left|\gep^{\gamma^2-d/2} M^{\gamma,\gb}_\gep (A)\right|^2\ | X\right]
    =\sigma^2(\gamma^2+\gb^2) M^{2\gamma,0}(A).
 \end{equation}
 where 
 \begin{equation}\label{subcritsigma}
 \sigma^2(\gamma^2+\gb^2)= \lim_{\gep\to 0} \varepsilon^{\gamma^2+\beta^2-d}\int_{|z| \leq 1} \frac {\dd z }{G_{\gep}(z)^{\beta^2+\gamma^2}}  .
 \end{equation}
$\bullet$ When $\gamma^2+\gb^2=d$, $\gamma< \sqrt{\frac{d}{2}}$,
  we have  the following $\bbL_1$-convergence
\begin{equation}\label{critical}
    \lim_{\gep\to 0}\bbE\left[\left|{\gep}^{\gamma^2-d/2}|\log (\gep)|^{-1/2} M^{\gamma,\gb}_\gep (A)\right|^2\ | X\right]
    =\sigma^2(d) M^{2\gamma,0}(A).
 \end{equation}
 where 
  \begin{equation}\label{critsigma}
 \sigma^2(d)= \lim_{\gep\to \infty} |\log \varepsilon|^{-1}\int_{|z|\le 1}\frac {\dd z }{G_{\gep}(z)}.
 \end{equation}
 
\end{proposition}

\begin{proposition}\label{lesmomentos2}
 Let $A$ be some compact ball with non empty interior. When $\gamma^2+\gb^2>d$, $\gamma=\sqrt{\frac{d}{2}}$,
  we have the following convergence in probability
  
\begin{equation}\label{lignedeutroi}
    \lim_{\gep\to 0}\bbE\left[\left | |\log \gep|^{1/4} M^{\gamma,\gb}_\gep (A)\right|^2\ | X\right]
    =\sigma^2(\gamma^2+\gb^2) \sqrt{\frac{2}{\pi}} M'(A).
 \end{equation}
\end{proposition}

First, we point out that computing the limits  \eqref{subcritsigma} and \eqref{critsigma} yields the expression of Definition \ref{sigmasquare}. This point will be shown in our proofs.

The proof of the above two propositions will be the object of Sections \ref{slesmomentos} to
 \ref{frontier23}. For notational simplicity, we  consider in the proof only the case $d=1$. The 
 general case does not present additional difficulties. 
Computing the variance of $M(A)$ for an arbitrary interval is not more difficult than dealing with the case $A=[0,1]$ so that  
we will only consider this case in the proof.
Before going to the core of the proof (in which we will have to deal with the II/III case separately) let us exhibit a more 
user-friendly expression for $\bbE[ |M^{\gamma,\gb}_\gep (A)|^2 | \mathcal{F}^X]$:
\begin{equation}\label{secondmom}
 \bbE\left[\left|\gep^{\gamma^2-1/2} M^{\gamma,\gb}_\gep ([0,1])\right|^2\ | \mathcal{F}^X\right]=
\gep^{2\gamma^2+\gb^2-1}\int_{[0,1]^2} \frac{   M^{\gamma,0}_\gep(\dd x) M^{\gamma,0}_\gep(\dd y)}{G_\gep(x-y)^{\gb^2}}.
\end{equation}
The right-hand side of \eqref{secondmom} can be rewritten as
\begin{equation}\label{secondmom2}
\gep^{\gamma^2+\beta^2-1} \int_{[0,1]^2} 
\frac{\exp\left(\gamma(X_\gep(x)+X_\gep(y) ) -\frac{\gamma^2}{2}
\E[(X_\gep(x)+X_\gep(y))^2] \right) }{G_{\gep}(x-y)^{\beta^2+\gamma^2} } \dd x \dd y.
\end{equation}

By symmetry the integral is equal to $2\int_{\{y\ge x\}}$ (of course we cannot do anything of the kind when $d\ge 2$ but this step is just performed for notational convenience here and there is nothing crucial about it). With the change of variables $y-x\to r$ and setting 
\begin{equation}\label{def.omega}
X_{\gep,r} (z)= X_\varepsilon(z+r)+X_\varepsilon(z),
\end{equation}
we obtain that \eqref{secondmom2} is equal to
\begin{multline}\label{secondmom3}
2\int_0^1\frac 1 {G_{\gep}(r)^{\beta^2+\gamma^2}}\left(\int_{0}^{1-r} 
\exp\left(\gamma X_{\varepsilon,r}(z) -\frac{\gamma^2}{2}
\E[(X_{\varepsilon,r}(z))^2] \right) \dd z\right) \dd r
=2\int_{0}^1\frac 1 {G_{\gep}(r)^{\beta^2+\gamma^2}}\hat M^\gamma_{\gep,r} \dd r.
\end{multline}
where we have defined 
\begin{equation}\begin{split}
 \hat M^{\gamma}_{\gep,r}&= \int_{0}^{1-r} 
\exp\left(\gamma X_{\varepsilon,r}(z) -\frac{\gamma^2}{2} 
\E[(X_{\varepsilon,r}(z))^2]\right) \dd z,\\
 \text{ and } \quad M^{\gamma}_{\gep,r}&= \int_{0}^1 
\exp\left(\gamma X_{\varepsilon,r}(z) -\frac{\gamma^2}{2}
\E[(X_{\varepsilon,r}(z))^2]\right) \dd z.
\end{split}
\end{equation}
As a conclusion we have
\begin{equation}\label{subcritical}
  \bbE\left[\left|\varepsilon^{\gamma^2-1/2} M^{\gamma,\gb}_\varepsilon ([0,1])\right|^2\ | \mathcal{F}^X\right]
  = \varepsilon^{\gamma^2+\beta^2-1}\int_0^1\frac 2 {G_{\gep}(r)^{\beta^2+\gamma^2}} \hat M^{\gamma}_{\gep,r} \dd r.
\end{equation}

Similarly when $\gamma^2+\gb^2=1$, we have

\begin{equation}\label{untroiz}
  \bbE\left[\left|\varepsilon^{\gamma^2-1/2} |\log\gep|^{-1/4}M^{\gamma,\gb}_\varepsilon ([0,1])\right|^2\ | \mathcal{F}^X\right]
  = |\log \gep|^{-1/2}\int_0^1\frac 2 {G_{\gep}(r)} \hat M^{\gamma}_{\gep,r} \dd r.
\end{equation}

The first key tool of the argument relies on the following lemma, the proof of which is straightforward and thus left to the reader:
\begin{lemma}\label{greenk}
We can write:
$$G_\gep(\gep t)=\gep f(t)g_\gep(t)$$
with $$f(t)=e^{\int_0^{|t|}\frac{1-k(u)}{u}\,\dd u}\quad \text{ and }\quad g_\gep(t)=e^{-\int_0^{|\gep t|}\frac{1-k(u)}{u}\,\dd u}.$$
The function $f$ is continuous on $\R$ and $f(t)\simeq \ln |t|$ for $t$ large. Furthermore, for each fixed $t$, $g_\gep(t)\to 1$ as $\gep \to 0$ and $\sup_{\gep |t|\leq 1}|\ln g_\gep(t)|= C<\infty$.
\end{lemma}

In \eqref{subcritical} and \eqref{untroiz}, we integrate $\hat M^{\gamma}_{\gep,r}$ w.r.t. to a measure on $[0,1]$: either $\varepsilon^{\gamma^2+\beta^2-1}\frac 2 {G_{\gep}(r)^{\beta^2+\gamma^2}}\dd r$
or $|\log \gep|^{-1}\frac 2 {G_{\gep}(r)^{1}}\dd r$.
In both cases the total mass of the measure is of order one and converges when $\gep\to 0$. Furthermore when $\gep\to 0$
most of the mass of these measures is supported by small $r$, even by $r$ that are of order $\gep$ in the case $\gamma^2+\gb^2>1$.

\medskip

Now the idea is that  $X_{\gep,r} (z)\approx 2X_\gep$ when $r$ is small and hence that $\hat M^{\gamma}_{\gep,r}$ (or $M^{\gamma}_{r,\gep}$) can be replaced by $M^{2\gamma}_\gep([0,1])$ in the integral.
This corresponds to proving the following lemma

\begin{lemma}\label{convergence}
For $\gamma<1/\sqrt{2}$, we have the following convergence in $\bbL_1$,
\begin{equation}
 \lim_{\gep,s\to 0}\hat M^{\gamma}_{\gep,s} =\lim_{\gep,s\to 0}M^{\gamma}_{\gep,s}= M^{2\gamma,0}([0,1])
\end{equation}
\end{lemma}

When $\gamma=1/\sqrt{2}$, $M^{2\gamma,0}([0,1])$ has to be replaced by the derivative martingale, which is not in $\bbL_1$ so that there is no hope for an equivalent statement to hold. The case is treated separately in Section \ref{frontier23}.

\begin{proof}[Proof of Proposition \ref{lesmomentos} using Lemma \ref{convergence}]

We first show \eqref{subcritsigma} and \eqref{critical}.
Using the definition of $G_\gep$ we can write 
\begin{equation}
\varepsilon^{\gamma^2+\beta^2-1}\int_{0}^1\frac {\dd r }{G_{\gep}(r)^{\beta^2+\gamma^2}}
=\gep^{-1}\int_0^1 \dd r \exp\left( (\gamma^2+\beta^2)\int_1^{\gep^{-1}} \frac{k(ur)-1}{u}\dd u\right).
\end{equation}
Then we use first the change of variables $r\to r'=r\gep^{-1}$ and $u\to u'=u\gep$ and the above integral becomes
\begin{equation}
\int^{1/\gep}_0 \dd r ' \exp\left( (\gamma^2+\beta^2)\int^1_{\gep} \frac{k(u'r')-1}{u'}\dd u'\right),
\end{equation}
and then we obtain  \eqref{subcritsigma} by taking the limit when $\gep\to 0$ and applying standard integration theorems.

\medskip

For \eqref{critical}, we have 
\begin{multline}
 |\log \varepsilon|^{-1}  \int_0^1\frac {\dd r }{G_{\gep}(r)}= |\log \varepsilon|^{-1}\int_{0}^{\gep|\log \gep|^{1/2}}\frac {\dd r }{G_{\gep}(r)}+
 |\log \varepsilon|^{-1}\int_{|\log \gep|^{-1/2}}^1\frac {\dd r }{G_{\gep}(r)}\\
 + |\log \varepsilon|^{-1} \int_{\gep|\log \gep|^{1/2}}^{|\log \gep|^{-1/2}}\frac {\dd r }{r}\exp\left( \int_1^{\gep^{-1}} \frac{k(ur)-\ind_{[0,r^{-1}]}(u)}{u}\dd u \right).
 \end{multline}
The first term tends to zero because $G_{\gep}(r)\ge C \gep$, the second term as well because $G_{\gep}(r)\ge C/r$.  To prove the convergence of the third  term, we perform the change of variables $u\to u'=ur$ to obtain   
\begin{equation}
\int_{\gep|\log \gep|^{1/2}}^{|\log \gep|^{-1/2}}\frac {\dd r }{r}\exp\left( \int_r^{\gep^{-1}r} \frac{k(u')-\ind_{[0,1]}(u')}{u'}\dd u' \right).
\end{equation}
The term in the exponential converges uniformly to 
$$\int_0^{\infty} \frac{k(u')-\ind_{[0,1]}(u')}{u'}\dd u'$$ on the domain of integration when $\gep$ tends to zero and then integrating with respect to $r$ cancels the 
$|\log \varepsilon|^{-1}$ in front of the integral.

\medskip

Now what remains to prove is that in $\bbL_1$:
\begin{equation}\begin{split}
\lim_{\gep\to 0} 
\gep^{\gamma^2+\beta^2-1} \int_0^1\frac 1 {G_{\gep}(r)^{\beta^2+\gamma^2}} | \hat M^{\gamma}_{\gep,r}- M^{2\gamma,0}([0,1])|  \dd r&=0, \quad \text{when } \gga+\gb<1\\
\lim_{\gep\to 0} 
|\log \gep|^{-1} \int_0^1\frac 1 {G_{\gep}(r)} | \hat M^{\gamma}_{\gep,r}- M^{2\gamma,0}([0,1])|  \dd r&=0 \quad \text{when } \gga +\gb=1
\end{split}\end{equation}
Let $\delta$ be arbitrarily small. According to Lemma \ref{convergence} 
we can find $\eta$ such that if   
$\max(r,\gep)\le \eta$, then 
$$|\hat M^{\gamma}_{\gep,r}- M^{2\gamma,0}([0,1])|\le \delta.$$
Then for all $\gep<\eta$ we have
\begin{multline}
 \bbE\left[\gep^{\gamma^2+\beta^2-1}\int_0^1\frac 2 {G_{\gep}(r)^{\beta^2+\gamma^2}} | \hat M^{\gamma}_{\gep,r}- M^{2\gamma,0}([0,1])|  \dd r\right]\\
 \le \delta \gep^{\gamma^2+\beta^2-1}\int_0^1\frac 2 {G_{\gep}(r)^{\beta^2+\gamma^2}}\dd r
 + \int_{r\in [\eta,1]}\frac 2 {G_{\gep}(r)^{\beta^2+\gamma^2}}\dd r.
\end{multline}
According the proof of \eqref{subcritsigma}, the first term is smaller than a constant times $\delta$ and the second one can be made arbitrary small for a fixed $\eta$ if $\gep$ is taken sufficiently small.
This is enough to conclude as $\delta$ is arbitrary.
The other limit can be proven in the same manner.\end{proof} 

\subsection{Proof of Lemma \ref{convergence}}
First let us notice that it is sufficient to prove the result for $M^{\gamma}_{\gep,s}$ as 
\begin{equation}
\bbE\left[|M^{\gamma}_{\gep,s}-\hat M^{\gamma}_{\gep,s}|\right]=s
\end{equation}
and thus tends to zero.

\medskip

If $\gep\le \sqrt{s}$, we write 
\begin{equation}\label{bool}
 |M^{\gamma}_{\gep,s}-M^{2\gamma,0}|\le 
 |M^{\gamma}_{\gep,s}-M^{\gamma}_{\sqrt{s},s}|+
 |M^{\gamma}_{\sqrt{s},s}-s^{\gamma^2}M^{2\gamma,0}_{\sqrt{s}}[0,1]|
 +|s^{\gamma^2} M^{2\gamma,0}_{\sqrt{s}}[0,1]-M^{2\gamma,0}[0,1]|
\end{equation} 
and if $\gep>\sqrt{s}$, 
\begin{equation}\label{booly}
  |M^{\gamma}_{\gep,s}-M^{2\gamma,0}[0,1]|\le 
 |M^{\gamma}_{\gep,s}-\gep^{2 \gamma^2}M^{2\gamma,0}_{\gep}[0,1]|
 +|\gep^{2 \gamma^2} M^{2\gamma,0}_{\gep}[0,1]-M^{2\gamma,0}[0,1]|.
\end{equation}

The last terms in the r.h.s. of both \eqref{bool} and \eqref{booly} converge to $0$ in $\bbL_1$ as the martingale $( \gep^{2 \gamma^2} M^{2\gamma,0}_{\gep})_{\gep\ge 0}$ is uniformly integrable (see e.g. \cite{cf:Kah}) 
and we just have to care about the other terms.
For the second term in \eqref{bool}, we use the following result

\begin{lemma}\label{lesgaussiennes}
 Let $(X,Y)$ be a centered Gaussian vector. There exists a universal constant $C$ such that
 \begin{equation}
  \bbE[\left| e^{X-\bbE[X^2]/2}-e^{Y-\bbE[Y^2]/2} \right|]\le C\sqrt{\bbE[(X-Y)^2]}.
 \end{equation}

\end{lemma}
\begin{proof}
 We use the Girsanov formula for the measure tilted by $X$ and we obtain
 \begin{multline}
    \bbE\left[\left| e^{X-\bbE[X^2]/2}-e^{Y-\bbE[Y^2]/2} \right|\right]=
    \bbE\left[e^{X-\bbE[X^2]/2}\left|1-e^{Y-X-\bbE[Y^2]/2+\bbE[X^2]/2}\right|\right]\\
    = \bbE\left[\left|1-e^{Y+\bbE[XY]-X-\bbE[X^2]-\bbE[Y^2]/2+\bbE[X^2]/2}\right|\right]
    =\bbE\left[\left|1-e^{(Y-X)-\bbE[(Y-X)^2]/2}\right|\right].
 \end{multline}
Then the reader can check that the last term can be bounded by $C\sqrt{\bbE[(X-Y)^2]}$
(it is sufficient to check it when $(X-Y)$ has small variance because it is always smaller than $2$).
\end{proof}

By Jensen's inequality and stationarity, we have
\begin{align}
  |M^{\gamma}_{\sqrt{s},s}-s^{\gamma^2} M^{2\gamma,0}_{\sqrt{s}}[0,1]|\le &
  \int_0^1\bbE\left[\left| e^{\gamma \go_{\sqrt{s},s}(x)-\frac{\gamma^2}{2}\bbE[\go_{\sqrt{s},s}(x)^2]}-
 e^{2\gamma X_{\sqrt{s}}(x)- 2\gamma^2 \bbE[X_{\sqrt{s}}(x)^2]}\right|\right]\dd x\\
 \le & C\sqrt{\bbE\left[(\go_{\sqrt{s},s}(0)-2\go_{\sqrt{s}}(0))^2\right]}\le C' s^1/4.
\end{align}
which tends to zero.
The same computation allows to control also the first term in \eqref{booly}

\medskip

Finally, we control the $|M^{\gamma}_{\gep,s}-M^{\gamma}_{\sqrt{s},s}|$ term in \eqref{bool} by proving the following Lemma that ensures that 
the  $(M^{\gamma}_{\gep,s})_{\gep\ge 0}$ is uniformly Cauchy in $\bbL^q$ for some $q>1$ (and hence in $\bbL^1)$ when $\gep\to 0$.

\begin{lemma}
For any $\gamma<1/\sqrt 2$, for any 
$q\in (1,\min(2, \frac{1}{2 \gamma^2}))$,
there exists $C>0$ and $\alpha(q,\gamma)>0$ (not depending on $s$, but only on $q$) such that  
 for any $\gep>\gep'>0$ we have
\begin{equation}
\bbE\left[ |M^{\gamma}_{\gep',s}-M^{\gamma}_{\gep,s}|^q \right]^{1/q}\le C \gep^{\alpha}
\end{equation}
\end{lemma}

\begin{proof}
In the proof we always consider that $\gep$ is small enough. 
It is sufficient to prove the result for $\gep'\in[\gep/2,\gep)$ as then we can use telescopic sums i.e. 
if $\gep'\in [2^{-(m+1)}\gep,2^{-m}\gep)$, we have
\begin{align}
 \bbE\left[ |M^{\gamma}_{\gep',s}-M^{\gamma}_{\gep,s}|^q \right]^{1/q} 
 \le &\sum_{k=0}^{m-1}  \bbE\left[ |M^{\gamma}_{2^{-(k+1)}\gep,s}-M^{\gamma}_{2^{-k}\gep,s}|^q \right]^{1/q}
 +\bbE\left[ |M^{\gamma}_{\gep',s}-M^{\gamma}_{2^{-m}\gep,s}|^q \right]^{1/q}
\nonumber \\ \le &C \gep^{\alpha} \sum_{k=0}^{m} 2^{-k\alpha}.
\end{align}
Recall that $\mathcal F_\gep^X$ is the $\sigma$-algebra generated by $X_{t}$, $t\ge \gep$. We have for all $q\in(1,2)$:
\begin{equation}\label{conditional}
 \bbE\left[ |M^{\gamma}_{\gep',s}-M^{\gamma}_{\gep,s}|^q \right]
 \le \bbE\left[ \bbE\left[(M^{\gamma}_{\gep',s}-M^{\gamma}_{\gep,s})^2 \ | \ \mathcal F_{\gep}^X \right]^{q/2}\right]
\end{equation}
The conditional expectation $\bbE\left[(M^{\gamma}_{\gep',s}-M^{\gamma}_{\gep,s})^2 \ | \ \mathcal F_\gep^X \right]$
is in fact a conditional variance as $\bbE [ M^{\gamma}_{\gep',s} \ | \  \mathcal F_\gep^X]=M^{\gamma}_{\gep,s}$.
Then the conditional variance can be expanded as a double integral: 
\begin{multline}\label{covdev}
\int_{[0,1]^2} \dd x \dd y \ e^{\gamma (X_{\gep,s}(x)+X_{\gep,s}(y))-\frac{\gamma^2}{2}\bbE\left[X^2_{\gep,s}(x)+X^2_{\gep,s}(y)\right]} \\
\times \Cov\left(e^{\gamma (X_{\gep',s}- X_{\gep,s})(x)-\frac{\gamma^2}{2}\bbE\left[(X_{\gep',s}-X_{\gep,s})^2(x)\right]} ,
e^{\gamma (X_{\gep',s}-X_{\gep,s})(y)-\frac{\gamma^2}{2}\bbE\left[(X_{\gep',s}- X_{\gep,s})^2(y)\right]}\right)
\end{multline}
Now we try to control the covariance term in $x$ and $y$, with a simple function.
We have
\begin{multline}
\Cov\left(e^{\gamma (X_{\gep',s}- X_{\gep,s})(x)-\frac{\gamma^2}{2}\bbE\left[(X_{\gep',s}-X_{\gep,s})^2(x)\right]} ,
e^{\gamma (X_{\gep',s}-X_{\gep,s})(y)-\frac{\gamma^2}{2}\bbE\left[(X_{\gep',s}-(X_{\gep,s})^2(y)\right]}\right)
\\=e^{\gamma^2\E\left((X_{\gep',s}- X_{\gep,s})(x) (X_{\gep',s}-X_{\gep,s})(y)\right)}-1 \\
\le C \E\left((X_{\gep',s}- X_{\gep,s})(x) (X_{\gep',s}-X_{\gep,s})(y)\right),
\end{multline}
where the last inequality uses the fact that $\gep/2 \le \gep'$ (so that the covariances are uniformly bounded by $4\log 2$).
Then using translation invariance we have 
\begin{multline}
\E\left((X_{\gep',s}- X_{\gep,s})(x) (X_{\gep',s}-X_{\gep,s})(y)\right)\\
= 2\E\left((X_{\gep'}- X_{\gep})(x) (X_{\gep'}-X_{\gep})(y)\right)+ \E\left((X_{\gep'}- X_{\gep})(x+s)(X_{\gep'}-X_{\gep})(y)\right)\\
+\E\left((X_{\gep'}- X_{\gep})(x-s)(X_{\gep'}-X_{\gep})(y)\right).
\end{multline}
The first term is equal to 
\begin{equation}\label{lechemg}
\int_{1/\gep}^{1/\gep'} \frac{k(u(x-y))}{u}\dd u \le \log 2 \max_{v \ge \frac{|x-y|}{\gep}} |k(v)| :=g((x-y)/\gep),
\end{equation}
and one can find similar bounds for the two others. Note that by our assumption 
(\textbf{A.2}) on $k$, $g$ is an integrable function.
Using the inequality $ab\le a^2/2+b^2/2$ together with the symmetry in $x$ and $y$ for the exponential term in \eqref{covdev} and combining it with 
 the bound we have just obtained for the covariance we obtain 
\begin{multline}
\bbE [ (M^{\gamma}_{\gep',s}-M^{\gamma}_{\gep,s})^2\ | \  \mathcal F_\gep^X]
\le C \int_{[0,1]^2}e^{2\gamma X_{\gep,s}(x)-\gamma^2\bbE\left[X^2_{\gep,s}(x)\right]}\\
\times \left(2g\left(\frac{x-y}{\gep}\right)+g\left(\frac{x-y-s}{\gep}\right)+g\left(\frac{x-y+s}{\gep}\right)\right)\dd x \dd y.
\end{multline}
Now we can extend the above integral to $y\in \bbR$ to obtain an upper bound 
(and then three terms in the second line above give the same contribution by translation invariance), 
and make the change of variable $z=\frac{y-x}{\gep}$ to obtain
\begin{align}
\bbE [ (M^{\gamma}_{\gep',s}-M^{\gamma}_{\gep,s})^2\ | \  \mathcal F_\gep^X]\le &
C\gep \left(\int_0^1  e^{2\gamma X_{\gep,s}(x)-\gamma^2\bbE\left[X^2_{\gep,s}(x)\right]}\dd x \right)
\left(\int_{-\infty}^{\infty} g(z)\dd z\right)\nonumber\\
\le & C'  \gep 
\left(\int^1_0 e^{2\gamma X_{\gep,s}(x)-\gamma^2\bbE\left[X^2_{\gep,s}(x)\right]}\dd x\right).
\end{align}
Plugging this inequality into \eqref{conditional}, we have
\begin{equation}\label{deconnepas} 
 \bbE \left[|M^{\gamma}_{\gep',s}-M^{\gamma}_{\gep,s}|^q \right]
  \le C\gep^{q/2} \bbE\left[\left(\int^1_0 e^{2\gamma X_{\gep,s}(x)-\gamma^2\bbE\left[X^2_{\gep,s}(x)\right]}\dd x\right)^{q/2} \right].
 \end{equation}
Now we split the interval $[0,1]$ into $K=\lceil \gep^{-1}\rceil$ disjoint intervals of length $1/K$ setting
$$I_{\gep}^j:=[(j-1)/K,j/K], \quad j=1..K $$
By subadditivity we have
\begin{equation}\label{mec}
\bbE\left[\left(\int^1_0 e^{2\gamma X_{\gep,s}(x)-\gamma^2\bbE\left[X^2_{\gep,s}(x)\right]}\dd x\right)^{q/2} \right]
\le K \bbE\left[\left(\int_{I^1_\gep} e^{2\gamma X_{\gep,s}(x)-\gamma^2\bbE\left[X^2_{\gep,s}(x)\right]}\dd x\right)^{q/2} \right].
\end{equation}

Using H\"older's inequality we have, for any $p>1$ (let $p'=p/(p-1)$ denote its H\"older conjugate),
\begin{multline}\label{cpadlablag}
\bbE\left[\left(\int_{I^1_\gep} e^{2\gamma X_{\gep,s}(x)-\gamma^2\bbE\left[X^2_{\gep,s}(x)\right]}\dd x\right)^{q/2} \right] \\
\le  (1/K)^{q/2}\bbE\left[ e^{\gamma pq X_{\gep,s}(0)-\frac{pq\gamma^2}{2}\bbE\left[X^2_{\gep,s}(x)\right]}\right]^{1/p}
\bbE\left[\left(K\int_{I^1_\gep} e^{2\gamma (X_{\gep,s}(x)- X_{\gep,s}(0)})\dd x\right)^{p'q/2}\right]^{1/p'}.
\end{multline}
By Jensen's inequality, if $p'q/2>1$ (which holds whenever $p\le 2$ as $q>1$) the last factor is smaller than 
$$
\bbE\left[K\int_{I^1_\gep} e^{\gamma p' q (X_{\gep,s}(x)- X_{\gep,s}(0))}\dd x\right]^{1/p'}
\le \exp\left(\frac{\gamma^2 p' q^2}{2} \max_{x\in [0,1/K]}\bbE[ (X_{\gep,s}(x)- X_{\gep,s}(0))^2]\right).
$$
which is bounded above by a constant that does not depend on $\gep$.
On the other hand the first term is equal to 
$$e^{\frac{\gamma^2q(pq-1)}{2}\bbE\left[X^2_{\gep,s}(x)\right]}\le \gep^{-2\gamma^2q(pq-1)} $$
because $\bbE[\go_{\gep,s}(x)^2]\le 4 \bbE[\go_{\gep}(x)^2]=4|\log \gep|$. 
In the end, combining \eqref{deconnepas}, \eqref{mec} and \eqref{cpadlablag} (recall that $K\le 2\gep^{-1}$), we get that
\begin{equation}
 \bbE\left[|M^{\gamma}_{\gep',s}-M^{\gamma}_{\gep,s}|^q \right]\le C \gep^{q-1-2\gamma^2q(pq-1)}.
\end{equation}
Then the result is proved provided one can find $p>1$ such that  $q-1>2\gamma^2q(pq-1)$, 
which is possible whenever $q<1/(2\gamma^2)$.
\end{proof}

\subsection{Proof of Proposition \ref{lesmomentos2}} \label{frontier23}

Recall \eqref{subcritical}. The aim of this Section is to show that in probability
\begin{multline}\label{grogro}
\lim_{\gep\to 0} \gep^{\gb^2-1/2}\int_0^1 \frac{2\dd r}{G_{\gep}(r)^{\gb^2+1/2}} |\log \gep|^{1/2}\hat M^{1/\sqrt{2}}_{\gep,r}\\
=
\left(\lim_{\gep\to 0}  \gep^{\gb^2-1/2} \int_0^1 \frac{2\dd r}{G_\gep(r)^{\gb^2+1/2}}\right)\sqrt{\frac{2}{\pi}}M'[0,1]=\sigma^2(1/2+\gb^2)\sqrt{\frac{2}{\pi}}M'[0,1].
\end{multline} 

The second equality is in fact \eqref{subcritsigma} and has already been proved. In fact, we can replace $\hat M^{1/\sqrt{2}}_{\gep,r}$ by $M^{1/\sqrt{2}}_{\gep,r}$ in the l.h.s.
as 

\begin{equation}
\int_0^1  \gep^{\gb^2-1/2} \frac{2\dd r}{G_\gep(r)^{\gb^2+1/2}} |\log \gep|^{1/2}\bbE[|\hat M^{1/\sqrt{2}}_{\gep,r}-M^{1/\sqrt{2}}_{\gep,r}|]=
\int_0^1  \gep^{\gb^2-1/2} \frac{2r \dd r}{G_\gep(r)^{\gb^2+1/2}} |\log \gep|^{1/2}.
\end{equation}
which tends to $0$ thanks to  Lemma \ref{greenk}. Hence what remains to prove is that 
\begin{equation}
\lim_{\gep\to 0} \gep^{\gb^2-1/2}\left|\int_0^1 \frac{\dd r}{G_{\gep}(r)^{\gb^2+1/2}} \left( |\log \gep|^{1/2}M^{1/\sqrt{2}}_{\gep,r}-\sqrt{\frac{2}{\pi}}M'[0,1] \right)\right|=0
\end{equation}
or with a change of variable $r\to s=r/\gep$
\begin{equation} \label{patropmechan}
\lim_{\gep\to 0}\left|\int_0^{\gep^{-1}} \frac{\dd s}{(G_{\gep}(\gep s)/\gep)^{\gb^2+1/2}} \left( |\log \gep|^{1/2}M^{1/\sqrt{2}}_{\gep,\gep s}-\sqrt{\frac{2}{\pi}}M'[0,1] \right)\right|=0.
\end{equation}

We will split the proof of \eqref{patropmechan} in two lemmas: one taking care of the contribution of the small $s$ in the integral (which is the most important step)
and another one controlling the contribution of the larger $s$:

\begin{lemma}\label{lem:crit2}

We have in probability
\begin{equation}\label{dandy}
\lim_{A\to \infty}\lim_{\gep\to 0}\left|\int_A^{\gep^{-1}} 
\frac{\dd s}{(G_{\gep}(\gep s)/\gep)^{\gb^2+1/2}} \left( |\log \gep|^{1/2}M^{1/\sqrt{2}}_{\gep,\gep s}-\sqrt{\frac{2}{\pi}}M'[0,1]\right)\right|=0
\end{equation}
\end{lemma}

\begin{lemma}\label{lem:crit1}
For all $A$, we have in probability
\begin{equation}\label{refff}
\lim_{\gep\to 0}\left|\int_0^{A} 
\frac{\dd s}{(G_{\gep}(\gep s)/\gep)^{\gb^2+1/2}} \left( |\log \gep|^{1/2}M^{1/\sqrt{2}}_{\gep,\gep s}-\sqrt{\frac{2}{\pi}}M'[0,1]\right)\right|=0
\end{equation}
\end{lemma}

The limit \eqref{patropmechan} results from a combination of \eqref{dandy} and \eqref{refff}.

 Before proceeding with the proofs of these results, let us introduce some notation. In what follows, for some $\kappa>0$, 
 we consider the sets
\begin{equation}\begin{split}
B_\kappa\index{Bkappa@$B_\kappa$}&=\{\sup_{x\in[0,1]}\sup_{u\in ]0,1]} X_u(x)+\sqrt{2}\ln u\leq \kappa\},\\
B_{\kappa,\gep}&=\{\sup_{x\in[0,1]}\sup_{u\in [\gep,1]} X_u(x)+\sqrt{2}\ln u\leq \kappa\}.
\end{split}\end{equation}
Obviously $B_\kappa\subset B_{\kappa,\gep}$ for all $\gep>0$, and it is proved in \cite[Proposition 19]{Rnew7} that 
$$\Pb(B_\kappa)\to 1\quad \text{ as }\kappa\to \infty.$$
For $x\in [0,1]$ and $\kappa>0$, we also introduce the stopping time
$$\tau^\kappa_x\index{taukappa@$\tau^\kappa_x$}=\sup\{u\in ]0,1];X_u(x)+\sqrt{2}\ln u> \kappa\}.$$
It is readily seen that, on $B_\kappa$, we have  $\tau^\kappa_x=0$ for all $x\in [0,1]$.
\vspace{2mm}

\noindent {\it Proof of Lemma \ref{lem:crit2}.} 
The reader can check using Lemma \ref{greenk} that 
$$\lim_{A\to \infty}\lim_{\gep\to 0}\int_A^{\gep^{-1}} \frac{\dd s}{(G_{\gep}(\gep s)/\gep)^{\gb^2+1/2}}=0.$$
Hence what we have to prove is   
\begin{equation}
\lim_{A\to \infty}\lim_{\gep\to 0}\int_A^{\gep^{-1}} \frac{\dd s}{(G_{\gep}(\gep s)/\gep)^{\gb^2+1/2}} |\log \gep|^{1/2}M^{1/\sqrt{2}}_{\gep,\gep s}=0 \quad\text{in probability}.
\end{equation}
To do so, it is enough to show that the expectation of the above expression restricted to the event $B_{\kappa}$ vanishes, for all $\kappa$, as $A\to \infty$ and $\gep\to 0$. The key point is to
prove that there exists a constant $C$ such that for all $t$:
\begin{equation}\label{integsurB}
\bbE[ M^{1/\sqrt{2}}_{\gep,\gep t}\ind_{B_{\kappa}}] \le \sqrt{\frac{2}{\pi}}\frac{\kappa+\sqrt{2}/2( 1+\log(Ct))}{|\log \gep|^{1/2}}.
\end{equation}
Indeed, this implies that for all $\kappa>0$ 
\begin{align*}
& \lim_{A\to \infty}\lim_{\gep\to 0}\bbE\left[ \ind_{B_\kappa}\int_A^{\gep^{-1}} \frac{\dd s}{(G_{\gep}(\gep s)/\gep)^{\gb^2+1/2}} |\log \gep|^{1/2}M^{1/\sqrt{2}}_{\gep,\gep s}\right]\\
&\le  \ \lim_{A\to \infty}\lim_{\gep\to 0} \sqrt{2/\pi} \int_A^{\gep^{-1}} \frac{\kappa+\sqrt{2}/2( 1+\log(Cs))}{(G_{\gep}(\gep s)/\gep)^{\gb^2+1/2}}\dd s \\
& = 0 .
\end{align*}
Let us now prove \eqref{integsurB}.
\begin{align*}
\bbE[ M^{1/\sqrt{2}}_{\gep,\gep t}\ind_{B_{\kappa}}]=&\int_0^1 \bbE\left[ \ind_{B_{\kappa}}e^{{1/\sqrt{2}}\, X_{\gep,\gep t} (z)  -\frac{1}{4}\E[X_{\gep,\gep t} (z)^2]}\right]\dd z
\\ \le &\int_0^1 \bbE\left[ \ind_{\{\forall u\in[\gep,1] X_u(z)+\sqrt{2}\ln u\leq \kappa \}}e^{\frac 1 {\sqrt{2}} \, X_{\gep,\gep t} (z)  -\frac{1}{4}\E[X_{\gep,\gep t} (z)^2]}\right]\dd z.
\end{align*}

By the Girsanov formula, the expectation in the r.h.s.\ of the above expectation is equal to (recall that $K_\gep$ denotes the covariance kernel of $X_\gep$): 
\begin{align}\label{crodin}
 \bbP \Big(\sup_{u\in [\gep,1]}X_u(z)+ \sqrt{2}/2\ln u+\sqrt{2}/2 K_u(\gep t))\leq \kappa \Big).
\end{align}
Because of Assumption (A), we may find a constant $C$ such that $|k(0)-k(x)|\leq C|x|$ for all $x\in\R$. Thus we have for $u\in[\gep,1]$:
\begin{align*}
\ln \frac{1}{u}-  K_u(\gep t)= &\int_1^{u^{-1}} \frac{1-k(vt\gep)}{v}\dd v
\le \int_{\gep}^{1} \frac{1-k(vt)}{v}\dd v\le \int_{\gep}^{(Ct)^{-1}}Ct\dd v+   \int_{(Ct)^{-1}}^{1} \frac{\dd v}{v}
\\ \le& 1+\log(Ct).
\end{align*}
For some given $z$, the process $X_u(z)_{u\in]0,1]}$ has the same law of a time changed standard 
Brownian motion $(B_{\sqrt{|\ln u|}})_{u\in(0,1]}$  and hence
\eqref{crodin} is smaller than
\begin{align*}
 \bbP \Big(\sup_{s\in [ 0,|\ln \gep| ]} B_s \leq \kappa+\sqrt{2}/2( 1+\log(Ct))\Big)
 =\sqrt{\frac{2}{\pi}}\frac{\kappa+\sqrt{2}/2( 1+\log(Ct))}{|\log \gep|^{1/2}},
\end{align*}
which is enough to conclude.\qed

\vspace{3mm}

\noindent {\it Proof of Lemma \ref{lem:crit1}.} We set $\gep''=\sqrt{\gep}$, $\gep'=\gep \log\log (\gep^{-1})$.
From \cite{Rnew12}, we know that $|\ln \gep|^{1/2}\gep M^{\sqrt{2},0}_{\gep'}[0,1]$ converges towards $\sqrt{2/\pi} \, M'[0,1]$ (note that 
$\log \gep'/\log \gep\to 1)$ when $\gep$ tends to zero and thus  it is enough
to prove \eqref{refff} with $\sqrt{2/\pi} M'[0,1]$ replaced by $|\ln \gep|^{1/2}\gep M^{\sqrt{2},0}_{\gep'}[0,1]$.

\medskip

We introduce $\tilde X_{\gep,\gep',s}$ which is an interpolation between $X_{\gep,s}$ and $X_\gep$ (when $\gep\le \gep'\le 1$):
$$\tilde X_{\gep,\gep',s}:= 2X_{\gep'}+(X_{\gep,s}-X_{\gep',s}).$$
We then use the following decomposition 
\begin{align*}
 &\left|\int_0^{A} 
\frac{|\ln \gep|^{1/2} \dd s}{(G_{\gep}(\gep s) / \gep)^{\gb^2+1/2}}\left( M^{1/\sqrt{2}}_{\gep,\gep s}-\gep M^{\sqrt{2},0}_{\gep'}\right)\right|\\
&\le \left|\int_0^{A} 
\frac{|\ln \gep|^{1/2} \dd s}{(G_{\gep}(\gep s)/ \gep)^{\gb^2+1/2}}\left(
\int\limits_{0}^{1} e^{(1/\sqrt{2}) X_{\gep,\gep s} (z)  -\frac{1}{4}
\E[X_{\gep,\gep s} (z)^2]} -  e^{(1/\sqrt{2}) \tilde X_{\gep,\gep'',\gep s}(z)-\frac{1}{4}\E[ \tilde X_{\gep,\gep'',\gep s}(z)^2]}
\,\dd z\right)\right|
\\
&+  \left|\int_0^{A} 
\frac{|\ln \gep|^{1/2} \dd s}{(G_{\gep}(\gep s)/ \gep)^{\gb^2+1/2}}\left(
\int\limits_{0}^{1} e^{(1/\sqrt{2}) \tilde X_{\gep,\gep'',\gep s}(z)-\frac{1}{4}\E[ \tilde X_{\gep,\gep'',\gep s}(z)^2]}
-e^{(1/\sqrt{2}) \tilde X_{\gep,\gep',\gep s}(z)-\frac{1}{4}\E[ \tilde X_{\gep,\gep',\gep s}(z)^2]}
\,\dd z\right)\right|\\
&+ \left|\int_0^{A} 
\frac{|\ln \gep|^{1/2} \dd s}{(G_{\gep}(\gep s) / \gep)^{\gb^2+1/2}}\left(
\int\limits_{0}^{1} e^{(1/\sqrt{2}) \tilde X_{\gep,\gep',\gep s}(z)-\frac{1}{4}\E[ \tilde X_{\gep,\gep',\gep s}(z)^2]}
- e^{\sqrt{2} X_{\gep}(z)-\E[X_{\gep}(z)^2]} \,\dd z\right)\right|
\end{align*}
and show that each of the three terms converges to zero in probability.

The first term converges in $\bbL_1$ norm. Indeed, from Lemma \ref{lesgaussiennes} we have
\begin{align*}
 \bbE&\left[\left|\int_0^{A} 
\frac{|\ln \gep|^{1/2} \dd t}{(G_{\gep}(\gep t) / \gep)^{\gb^2+1/2}}\left(
\int\limits_{0}^{1} e^{(1/\sqrt{2}) X_{\gep,\gep t} (z)  -\frac{1}{4}
\E[X_{\gep,\gep t} (z)^2]} - e^{(1/\sqrt{2}) \tilde X_{\gep,\gep'',\gep t}(z)-\frac{1}{4}\E[ \tilde X_{\gep,\gep'',\gep t}(z)^2]}
\,\dd z\right)\right|\right]
\\
\le & \int_0^{A} 
\frac{|\ln \gep|^{1/2} \dd t}{(G_{\gep}(\gep t) / \gep)^{\gb^2+1/2}}
\int\limits_{0}^{1} \bbE\left[ |e^{(1/\sqrt{2}) X_{\gep,\gep t} (z)  -\frac{1}{4}
\E[X_{\gep,\gep t} (z)^2]} - e^{(1/\sqrt{2}) \tilde X_{\gep,\gep'',\gep t}(z)-\frac{1}{4}\E[ \tilde X_{\gep,\gep'',\gep t}(z)^2]}|\right]
\,\dd z \\
\le &C\int_0^{A} 
\frac{|\ln \gep|^{1/2} \dd t}{(G_{\gep}(\gep t) / \gep)^{\gb^2+1/2}}\sqrt{\bbE[  (X_{\gep,\gep t}(z)-\tilde X_{\gep,\gep'',\gep t}(z))^2]}
\le C'|\ln \gep|^{1/2}\gep^{1/4}.
\end{align*}

The second term
converges to zero in probability. 
This is a bit more tricky to show because we do not have $\bbL_1$ convergence but we can manage to obtain it by restricting ourselves to the event
$B_{\kappa,\gep''}$ with say $\kappa=\log\log|\log \gep|$ (whose probability tends to one).
By Jensen's inequality, the expectation of the second term on the event $B_{\kappa,\gep''}$
is smaller than
 \begin{multline}\label{tirangoli}
\int_0^{A} 
\frac{|\ln \gep|^{1/2} \dd t}{(G_{\gep}(\gep t) / \gep)^{\gb^2+1/2}}\\
\times
\int\limits_{0}^{1} \bbE\left[ \ind_{B_{\kappa,\gep''}}|e^{(1/\sqrt{2}) \tilde X_{\gep,\gep'',\gep t}(z)-\frac{1}{4}\E[ \tilde X_{\gep,\gep'',\gep t}(z)^2]}
-e^{ (1/\sqrt{2}) \tilde X_{\gep,\gep',\gep t}(z)-\frac{1}{4}\E[ \tilde X_{\gep,\gep',\gep t}(z)^2]}|\right]
\,\dd z .
\end{multline}
 Then, by independence of the increments of $(X_{\gep})_{\gep\ge 0}$, we have
\begin{multline}\label{lfzr}
\bbE  \left[ \ind_{B_{\kappa,\gep''}}\Big| \int\limits_{0}^{1} e^{(1/\sqrt{2}) \tilde X_{\gep,\gep'',\gep t}(z)-\frac{1}{4}\E[ \tilde X_{\gep,\gep'',\gep t}(z)^2]}- e^{ (1/\sqrt{2})  \tilde X_{\gep,\gep',\gep t}(z)-\frac{1}{4}\E[ \tilde X_{\gep,\gep',\gep t}(z)^2]}\,\dd z\Big|\right]\\
= \int\limits_{0}^{1} \bbE[ \ind_{B_{\kappa,\gep''}} e^{\sqrt{2} X_{\gep''}(z)- \E[ X_{\gep''}(z)^2]}  ]  \bbE\Big[\Big|e^{(1/\sqrt{2}) (\tilde X_{\gep,\gep'',\gep t}-2X_{\gep''})(z)-\frac{1}{4} 
\bbE[ (\tilde X_{\gep,\gep'',\gep t}-2X_{\gep''})^2(z)]}\\
- e^{(1/\sqrt{2}) (\tilde X_{\gep,\gep',\gep t}-2X_{\gep''})(z)-\frac{1}{4} \bbE[(\tilde X_{\gep,\gep',\gep t}-2X_{\gep''})^2(z)]}\Big|\Big] \dd z.
\end{multline}
By a Girsanov transform one sees that the first factor in the r.h.s.\ of \eqref{lfzr} is equal to
\begin{equation}
\bbP\left[\max_{u\in (\gep'',1]}X_u(z)\le \kappa\right]
=\bbP[ |X_{\gep''}(z)|\le \kappa]= \sqrt{\frac 2 {\pi}}\frac{\kappa}{(\log \gep)^{1/2}}.
\end{equation}
Using Lemma \ref{lesgaussiennes} one sees that the second factor is smaller than
$$
C\sqrt{\bbE\left[(\tilde X_{\gep,\gep'',\gep t}-\tilde X_{\gep,\gep',\gep t})^2\right]}\le C\sqrt{A (\log |\log \gep|)^{-1}},$$
Hence \eqref{tirangoli} is smaller than

\begin{equation}
 C\sqrt{A}\int_0^{A} 
\frac{(\log |\log \gep|)^{-1/2} \dd t}{(G_{\gep}(\gep t) / \gep)^{\gb^2+1/2}}\le C\sqrt A \kappa (\log |\log \gep|)^{-1/2}.
 \end{equation}

\medskip

Finally we show that 
\begin{equation}\label{damusketeer}
\left|\int\limits_{0}^{A}  \frac{|\ln \gep|^{1/2}}{(G_\gep(t\gep) / \gep)^{\beta^2+1/2}}  
\left( \int\limits_{0}^{1} e^{ (1/\sqrt{2}) \tilde X_{\gep,\gep',\gep t}(z)-\frac{1}{4}
\E[ \tilde X_{\gep,\gep',\gep t}(z)^2]}
- e^{\sqrt{2}  X_{\gep'}(z)-\E[ X_{\gep'}(z)^2]}\,\dd z\right)\dd t \right|
\end{equation}
tends to zero in probability. To do so we prove that for some $q<1$ the $q$-th moment converges to zero.
The content of $|\cdot|$ can be rewritten as

\begin{equation}
|\ln \gep|^{1/2} \int\limits_{0}^{1} e^{\sqrt{2}  X_{\gep'}(z)-\E[ X_{\gep'}(z)^2]}\xi(z) \dd z,
\end{equation}
where 
$$\xi(z):=   \int\limits_{0}^{A}  
\frac{ \dd t}{(G_{\gep}(\gep t) / \gep)^{\gb^2+1/2}}\left(
e^{ (1/\sqrt{2})\left(X_{\gep,\gep t}-X_{\gep',\gep t}\right)(z)-\frac{1}{4}\E\left[ \left(X_{\gep,\gep t}-X_{\gep',\gep t}\right)(z)^2\right]}-1 \right). $$
Then using conditional Jensen's inequality, we obtain that the $q$-th moment is smaller than

\begin{equation}\label{lemoment}
|\ln \gep|^{q/2}\bbE\left[ \left( \int_{[0,1]^2} e^{\sqrt{2}  (X_{\gep'}(z)+X_{\gep'}(z'))-2\E[ X_{\gep'}(z)^2]}\bbE[\xi(z)\xi(z')] \dd z \dd z'\right)^{q/2}\right].
\end{equation}
Thus we need to have an upper bound on $\bbE[\xi(z)\xi(z')]$ to conclude.

\begin{lemma}\label{covdexi}
We stick to  the notation of \eqref{lechemg}. For any $z$ and $z'$ we have
$$\bbE[\xi(z)\xi(z')]\le C  \left(\frac{\gep'}{\gep}\right)^2 \log(\gep'/\gep)  g\left(\frac{|z-z'|-\gep A}{\gep'}\right).$$
\end{lemma}
\begin{proof}
Expanding the two integrals we get

\begin{multline}\label{grouf}
\bbE[\xi(z)\xi(z')]=\int\limits_{0}^{A}  \int\limits_{0}^{A}
\frac{1}{(G_{\gep}(\gep t) / \gep)^{\gb^2+1/2}(G_{\gep}(\gep t') / \gep)^{\gb^2+1/2}} \\
\times\left(e^{\E [(X_{\gep,\gep t}-X_{\gep',\gep t})(z)(X_{\gep,\gep t'}-X_{\gep',\gep t'})(z')] /2}-1\right) \dd t\dd t'.
\end{multline}
Using the inequality $e^x-1\le e^K x$ for $x\le K$ and the fact that 
$$\E\left[ (X_{\gep,\gep t}-X_{\gep',\gep t})(z)(X_{\gep,\gep t'}-X_{\gep',\gep t'})(z') \right] \le 4 \bbE\left[ (X_{\gep}-X_{\gep'})^2(0)\right].$$
we get
$$
e^{\E[ (X_{\gep,\gep t}-X_{\gep',\gep t})(z)(X_{\gep,\gep t'}-X_{\gep',\gep t'})(z')  ]/2}-1\le  \left(\frac{\gep'}{\gep}\right)^2
\E\left[ (X_{\gep,\gep t}-X_{\gep',\gep t})(z)(X_{\gep,\gep t'}-X_{\gep',\gep t'})(z')\right]/2
$$
To get a bound on \eqref{grouf}, we need a bound on $\E [(X_{\gep,\gep t}-X_{\gep',\gep t})(z) (X_{\gep,\gep t'}-X_{\gep',\gep t'})(z')] $ that does not depend on $t$ nor $t'$. We do so by noticing that for any values of $t$, $t'$ in $[0,A]$, we have
\begin{equation}
\min(|z-z'|,|z-z'+\gep t|, |z-z'-\gep t'|, |z-z'+\gep t-\gep t'|)\ge |z-z'|-\gep A.
\end{equation}
Hence, using the notation introduced in \eqref{lechemg}, we have
\begin{equation}
\E[(X_{\gep,\gep t}-X_{\gep',\gep t})(z) (X_{\gep,\gep t'}-X_{\gep',\gep t'})(z')] 
\le C \log(\gep'/\gep)  g\left(\frac{(|z-z'|-\gep A)_+}{\gep'}\right).
\end{equation}
The result follows from the combination of the above inequalities  and the fact that 
$\int\limits_{0}^{A}
\frac{ \dd t}{(G_{\gep}(\gep t) / \gep)^{\gb^2+1/2}}$ is bounded uniformly in $A$ and $\gep$ from Lemma \ref{greenk}.
\end{proof}

From Lemma \ref{covdexi} and \eqref{lemoment} we have that the $q$-th moment of \eqref{damusketeer} is smaller than 
\begin{multline}\label{aaaaarrrghhh}
C  |\log \gep|^{q/2}\left(\frac{\gep}{\gep'}\right)^2 \log(\gep/\gep')\\
\times  \bbE\left[ \left( \int_{[0,1]^2} e^{\sqrt{2}  (X_{\gep'}(z)+X_{\gep'}(z'))-2\E[ X_{\gep'}(z)^2]}  g\left(\frac{(|z-z'|-\gep A)_+}{\gep'}\right) \dd z \dd z'\right)^{q/2}\right].
\end{multline}
In the above expectation, by using the relation
$$e^{\sqrt{2}  (X_{\gep'}(z)+X_{\gep'}(z'))-2\E[ X_{\gep'}(z)^2]}\leq  \frac{1}{2}e^{2\sqrt{2}  X_{\gep'}(z)-2\E[ X_{\gep'}(z)^2]} +\frac{1}{2}e^{2\sqrt{2}  X_{\gep'}(z')-2\E[ X_{\gep'}(z')^2]}, $$
 we get the following upper-bound by symmetrization:
\begin{align*}
&\int_{[0,1]^2} e^{\sqrt{2}  (X_{\gep'}(z)+X_{\gep'}(z'))-2\E[ X_{\gep'}(z)^2]}  g\left(\frac{(|z-z'|-\gep A)_+}{\gep'}\right) \dd z \dd z'\\
&\le  \int_{[0,1]^2} e^{2\sqrt{2}  X_{\gep'}(z)-2\E[ X_{\gep'}(z)^2]}  g\left(\frac{(|z-z'|-\gep A)_+}{\gep'}\right) \dd z \dd z'\\
&\le  \gep' \int^1_0 e^{2\sqrt{2}  X_{\gep'}(z)-2\E[ X_{\gep'}(z)^2]}\dd z \int_{\bbR} g\left((|z''|-(\gep A)/\gep')_+ \right) \dd z''
\end{align*}
where the last line was obtained by a change of variables and expanding the integral over
$\R$. The function in the second integral is smaller than $g\left((|z''|-1)_+ \right)$,
which is integrable.
And hence it remains to show that 
\begin{equation}\label{sezr}
|\log \gep|^{q/2}\left(\frac{\gep'}{\gep}\right)^2 \log(\gep'/\gep) \bbE\left[ \left( \gep' \int^1_0 e^{2\sqrt{2}  X_{\gep'}(z)-2\E[ X_{\gep'}(z)^2]}\dd z\right)^{q/2}\right],
\end{equation}
tends to zero.

Now, we  use Lemma \ref{superhushi} with $\gamma=2\sqrt{2}$ and $\alpha=q/2< 1/2$, 
and we obtain  the following bound
$$ \E\left[ \left( \gep' \int^1_0 e^{2\sqrt{2}  X_{\gep'}(z)-2\E[ X_{\gep'}(z)^2]}\dd z\right)^{q/2}\right]    \le C |\log \gep' |^{-\frac{3q}{2}},$$
so that the expression \eqref{sezr} is smaller than 

$$C |\log \gep|^{-q}\left(\frac{\gep'}{\gep}\right)^2  \log(\gep'/\gep),$$
 which tends to zero according to our definition of $\gep'$. \qed

\subsection{Proof of Theorem \ref{th3} and \ref{thfrontier23} } \label{twoksec}
Now, we conclude the proofs of Theorem \ref{th3} and \ref{thfrontier23}. For simplicity, we consider the case $d=1$. We only treat the proof of Theorem \ref{th3} since Theorem \ref{thfrontier23} can be dealt with similarly. We consider $l$ disjoint intervals $A_1, \cdots, A_l$. We fix k vectors $u_1, \cdots, u_l$ in $\R^2$.
We denote by $u.x$ the Euclidean scalar product on $\R^2$. We also introduce a sequence $(f_j)_{j \geq 1}$ of continuous and bounded functions which is dense in the space of continuous functions with compact support.   We want to show that 
\begin{equation*}
Z_\gep^{u_1, \cdots, u_l}: =\sum_{i=1}^l \gep^{\gamma^2-\frac{1}{2}}  \: 
u_i . M^{\gamma,\gb}_\gep(A_i)
\end{equation*}
converges in law to $Z^{u_1, \cdots, u_l}:=\sum_{i=1}^l W_{\sigma^2 M^{2\gamma,0}}(A_i)$ 
as $\gep$ goes to $0$. If this was not the case, we could find an index $j_0$ 
and a subsequence $\gep_n$ going to $0$ such that 
$\E[ f_{j_0}(Z_{\gep_n}^{u_1, \cdots, u_l})]$ does not converge to 
$\E[ f_{j_0}(Z^{u_1, \cdots, u_l})]$. By 
Proposition \ref{propclesmoments} and a diagonal extraction argument, we can find a subsequence $(n_p)_{p \geq 1}$ such that for all $k \geq 1$ we get the following almost sure convergence (with respect to $X$):
\begin{equation*}        
 \E\left[  (Z_{\gep_{n_p}}^{u_1, \cdots, u_l})^k \ | \ X\right] 
 \underset{p \to \infty}{\rightarrow}   \E\left[  (Z^{u_1, \cdots, u_l})^k \  | \  X\right]
 \end{equation*}   
By the method of moments, we deduce that almost surely we have:
\begin{equation*}       
 \E\left[  f_{j_0}(Z_{\gep_{n_p}}^{u_1, \cdots, u_l})\ | \ X\right]  \underset{p \to \infty}{\rightarrow}   \E\left[  f_{j_0}(Z^{u_1, \cdots, u_l}) \ | \ X\right]
 \end{equation*}  
 Hence by dominated convergence, we get that $ \E\left[  f_{j_0}(Z_{\gep_{n_p}}^{u_1, \cdots, u_l}) \right]  \underset{p \to \infty}{\rightarrow}   \E\left[  f_{j_0}(Z^{u_1, \cdots, u_l})\right]$ which contradicts our assumption.

\section{Conjectures}\label{sec:conj}
\subsection{Reminder: conjectures on $\beta=0$}
In the case $\beta=0$, there are still some open questions about the renormalization of the measures $(M^{\gamma,0}_\gep)_{\gep}$. Some conjectures are stated in \cite{Rnew4,Rnew7} that we recall here:
\begin{conjecture}\label{real}
Assume $\gamma>\sqrt{2d}$ and set $\alpha=\frac{\sqrt{2d}}{\gamma}$. Then
\begin{equation}\label{renormsurcrit}
(-\ln\gep)^{\frac{3\gamma}{2\sqrt{2d}}}\gep^{ \gamma  \sqrt{2d}-d}M^{\gamma,0}_\gep(\dd x)
\stackrel{law}{\to} c_\gamma N_\alpha(\dd x),\quad \text{as }\gep\to 0
\end{equation}
where $c_\gamma$ is a positive constant depending on $\gamma$ and the law of $N_\alpha$ is given, conditioned on the derivative martingale $M'$, by an independently scattered random measure the law of which is characterized by
$$\forall A\in\mathcal{B}(\R^d),\forall q\geq 0,\quad \E[e^{-qN_\alpha(A)}|M']=e^{-q^\alpha M'(A)}. $$
\end{conjecture}

Following this work, the authors of \cite{MRV} proved in fact the above conjecture in the case where $k$ has compact support but also when $X_\gep$ is a specific cut-off scheme for the (massless) GFF in a bounded domain or the massive planar GFF. Note however that the results of \cite{MRV} are valid for specific cut-off approximations and what is lacking is a universality result which establishes convergence \eqref{renormsurcrit} for a wide class of cut-off approximations (it is natural to expect that the constant $c_\gamma$ in \eqref{renormsurcrit} depends on the cut-off approximation).

\subsection{Conjectures on the inner phase II}
 We state here conjectures on this phase that we should be able to prove in the case of discrete cascades thanks to the exact study of the extremal process combined with our argument to establish convergence in law towards a complex Gaussian random measure conditionally on $X$. In the context of Gaussian multiplicative chaos, we have to rely on the conjecture \ref{real} to state:
 
 \begin{conjecture}\label{conjII}
 Let $\beta>0$ and $\gamma>\sqrt{\frac{d}{2}}$ such that $\gamma+\beta>\sqrt{2d}$. Then we get the following convergence in law:
\begin{equation*}
\left( (\ln \frac{1}{\gep})^{\frac{3 \gamma}{2\sqrt{2d}}} \gep^{   \gamma \sqrt{2d}-d} M^{\gamma,\gb}_\gep(A)\right)_{A\subset \R^d} \Rightarrow \left(W_{\sigma^2 N^\alpha_{M'}}(A)\right)_{A\subset \R^d}.   
 \end{equation*}
 where, conditionally on  $N^\alpha_{M'}$,  $W_{\sigma^2 N^\alpha_{M'}}$ is a complex Gaussian random measure with intensity  $N^\alpha_{M'}$ and  $N^\alpha_{M'}$ is a $\alpha$-stable random measure with intensity $M'$ and $\alpha=\sqrt{\frac{d}{2}}\frac{1}{\gamma}$, namely an independently scattered random measure  whose law is characterized by $\E[e^{-q N^\alpha_{M'}(A)}]=e^{-q^\alpha M'(A)}$ for all $u\geq 0$ and $A$ bounded Borel set. The constant $\sigma^2$ depends on $(\gamma,\beta)$.
 \end{conjecture}

Following this work, the authors of \cite{MRV1} proved an analogous result  to conjecture \ref{conjII} in the simpler case of complex Branching Brownian motion (BBM). Recall that BBM can be seen as a log-correlated Gaussian field (on a metric space where the distance relies on a hierarchical structure) and its behaviour is expected (and has been proved in numerous cases) to exhibit similar features to log-correlated Gaussian fields. Therefore, the main result of \cite{MRV1} brings additional evidence for conjecture \ref{conjII}.

\subsection{Triple point}

Concerning the triple point, the situation is a bit more delicate. When looking at the proof of subsection \ref{frontier23}, it is natural to expect:
 \begin{conjecture}
For  $\beta=\gamma=\sqrt{\frac{d}{2}}$,  the following convergence in law holds:
\begin{equation*}
\left( |\ln \gep|^{-\frac{1}{4}}  M^{\gamma,\gb}_\gep(A)\right)_{A\subset \R^d} \stackrel{\gep\to 0}{\Rightarrow} \left(W_{\sigma^2 M'}(A)\right)_{A\subset \R^d}.   
 \end{equation*}
 where, conditionally on  $M'$,  $W_{\sigma^2 M'}(\cdot)$ is a complex Gaussian random measure with intensity  $\sigma^2 M'$, and $\sigma^2$ is a constant.
 \end{conjecture}

The technical difficulty lies in the proof of the second moment, i.e.\ the equivalent of Proposition \ref{lesmomentos}:
in the integral \eqref{grogro}, most of the mass of the integral is supported by the $r$ of order $\gep$ and the rest is negligible.

\medskip

On the contrary at the triple point  all scales $\gep^{a}$, $a\in [0,1]$ contribute equaly to to the integral. 
Our problem is that the convergence of $\hat M^{1/\sqrt{2}}_{\gep,r}$ to the derivative martingale is rather weak:
it holds only in probability but not in $l_1$. For this reason it is difficult to control the convergence of the integral as we have 
to control all scales converge simultaneously.

\subsection{Continuity in dimension $1$ of the limiting process on the frontier I/II}

Looking at the proof of lemma  \ref{capI_II}, one can write the following heuristics for $k \leq n$: 
\begin{align*}
 \E\left[  \left| (2^{-n})^{\frac{\gamma^2}{2}-\frac{\beta^2} {2}}M_{\frac{1}{2^n}}^{\gamma, \beta}[0,2^{-k}]  \right|^2 \  | \  \mathcal{F}^X \right] & = \int_{[0,\frac{1}{2^k}] \times [0,\frac{1}{2^k}]} 
(2^{-n})^{\gamma^2} \frac{M_{\frac{1}{2^n}}^{\gamma,0}(\dd x) M_{\frac{1}{2^n}}^{\gamma,0}(dy)}{G_{\frac{1}{2^n}}(y-x)^{(\sqrt{2}-\gamma)^2}}  \\
& \approx \sum_{j=k}^n 2^{j \beta^2} \sum_{l=1}^{2^{j-k}} \frac{1}{2^{j\gamma^2}}  M_{\frac{1}{2^j}}^{\gamma,0}\left[  \frac{l-1}{2^j}, \frac{l}{2^j}   \right]^2 \\
& \approx \sum_{j=k}^n  \sum_{l=1}^{2^{j-k}} e^{  2 \gamma (X_{2^{-j}} (\frac{l}{2^j}) - \sqrt{2} \ln 2^j )  }
\end{align*}  
Recall that it is conjectured that $\sum_{l=1}^{2^{j-k}} e^{  2 \gamma (X_{2^{-j}} (\frac{l}{2^j})  - \sqrt{2} \ln 2^j )  }$ converges in law to some atomic random measure $\nu$ (see \eqref{renormsurcrit}). Hence, the limit $M^{\gamma,\beta}$ should satisfy the bound $\E[      M^{\gamma,\beta}([ 0, \frac{1}{2^k}   ] )^2 | \mathcal{F}^X   ]  \leq  $ $\nu( [ 0, \frac{1}{2^k}   ])  (k^{ \frac{3 \gamma}{\sqrt{2}} -1})^{-1} $ and more generally:
\begin{equation*}
\E[      M^{\gamma,\beta}([ s,t  ]) ^2 | \mathcal{F}^X  ]  \leq \frac {\nu( [ s,t   ])}  {(\ln_+\frac{1}{|t-s|}^{ \frac{3 \gamma}{\sqrt{2}} -1})} 
\end{equation*}
It is therefore natural to conjecture that in dimension 1, we can reinforce the above result by an almost sure convergence in the space of continuous functions. Indeed, as soon as one can show that the limiting measure $M^{\gamma,\beta}$ is cadlag, the above estimates entail continuity.


\section{Gaussian Free Fields}\label{GFF}

The Gaussian Free field \index{Gaussian free field (GFF)} \index{Massive free field (MFF)} with mass $m\ge 0$ on a set $D\subset \bbR^2$ (for simplicity we can say that $D$ is either a  planar bounded domain or the whole plane) and Dirichlet boundary condition
is the Gaussian field whose covariance function is given by the Green function of the problem
$$\triangle u- 2mu=-2\pi f \text { on }D,\quad u_{|\partial D}=0.$$
Notice the unusual normalization factor $2\pi$ in order to get correlations of the form \eqref{Kintrorev}. When $D=\bbR^2$ we have to consider $m>0$: otherwise the Green function is infinite everywhere. In the case of a bounded domain $D$, we are mostly interested in the case $m=0$ due to conformal invariance.

 The Green function can be written as 
$$g(x,y)=\pi\int_0^\infty e^{-rm} p(r,x,y)\,dr.$$
where $p(t,x,y)$ will denote the transition densities of the Brownian motion on $D$ killed upon touching $\partial D$.
A formal way to define   the complex  Gaussian field $X+iY$ (with $X$ and $Y$ independent GFF\index{Gaussian free field (GFF)}) is to consider  two independent white noises $W^X, W^Y$ on $\R_+\times D$ and define
\begin{equation}\begin{split}\label{GFFXY}
X(x)&=\sqrt{\pi}\int_{0}^{\infty}\int_{D}e^{-mr/2} p(r/2,x,y)\,W^X(dr,dy),\\
Y(x)&=\sqrt{\pi}\int_{0}^{\infty}\int_{D}e^{-mr/2} p(r/2,x,y)\,W^Y(dr,dy).
\end{split} \end{equation}
To define the exponential of the field $X+iY$, we need to use a cut-off procedure. So we define the approximations $X_{\gep}$ and $Y_{\gep}$ (respectively of the fields  $X$ and $Y$) by integrating over $(\gep^2,\infty)\times D$ in \eqref{GFFXY} instead of $(0,\infty)\times D$. The covariance function for these approximations is given by 
\begin{equation}\label{covar1}
\E[X_\gep(x)X_{\gep'}(y)]=\E[Y_\gep(x)Y_{\gep'}(y)]=\pi\int_{\gep^2\vee\gep'^2}^{\infty} e^{-rm} p(r,x,y)\,dr.
\end{equation}

\subsection{Massive Gaussian Free Field in the plane}
Observe that, in the case of the massive Gaussian Free Field on $\bbR^2$ (see subsection \ref{extension}), the kernel $p$ is translation invariant and has a simple expression
$$p(t,x,y)=\frac{1}{2\pi t}e^{-\frac{|x-y|^2}{2t}}.$$ 
The whole plane massive Green function then takes the form
\begin{equation} 
\forall x,y \in \R^2,\quad G_m(x,y)=\int_0^{\infty}e^{-mu-\frac{|x-y|^2}{2u}}\frac{\dd u}{2 u}.
\end{equation}
 and can be rewritten as 
\begin{equation} 
G_m(x,y)=\int_{1}^{+\infty}\frac{k_m(u(x-y))}{u}\,du.
\end{equation}
 for the continuous covariance kernel $k_m=\frac{1}{2}\int_0^\infty e^{-\frac{m}{v}|z|^2-\frac{v}{2}}\,dv$.  Therefore the whole plane massive free field strictly enters the framework of the first part of our paper.

\subsection{Gaussian Free Field \index{Gaussian free field (GFF)} on planar bounded domains}\label{GFFplanar}

The case of the GFF\index{Gaussian free field (GFF)} on a planar bounded domain is a bit more delicate (but nothing too serious) as $p(t,x,y)$ is not translation invariant this time.
We use the change of variables $r\to r^{-2}$ in the integral in the r.h.s. of \eqref{covar1} to find something closer to the setup that we have worked with in the previous sections and obtain:
\begin{equation}\label{covar2}
\E[X_\gep(x)X_{\gep'}(y)]=\E[Y_\gep(x)Y_{\gep'}(y)]=2\pi\int_0^{\gep^{-1}\wedge\gep^{'-1}} e^{-m / r^2}  r^{-3} p(r^{-2},x,y)\,dr.
\end{equation}
The kernel $2\pi e^{-m / r^2}r^{-3} p(r^{-2},x,y)$ will have to play the role of $\frac{k(r(x-y))}{r}$.

We have 
\begin{equation}\label{dominp}
2\pi r^{-3} p(r^{-2},x,y)\le \frac{e^{\frac{-(r|x-y|)^2}{2}}}{r}, \quad  \forall (x,y, r).
\end{equation}
Furthermore the two kernels are asymptotically equivalent in the interior of $D$ in the sense that for any compact $K\subset D$:
\begin{equation}
\lim_{r\to \infty} \sup_{x,y \in K} | 2\pi r^{-2} p(r^{-2},x,y)e^{\frac{(r|x-y|)^2}{2}}-1|=0.
\end{equation}
We will also consider the following decomposition of the covariance function
\begin{align*}
g_\gep(x,y):=&\E[X_\gep(x)X_{\gep}(y)]\\
=& \int_1^{\gep^{-1}} 2\pi r^{-3} p(r^{-2},x,y)\,dr+ \int_0^{1} 2\pi r^{-3} p(r^{-2},x,y)\,dr\\
=:&\tilde g_\gep(x,y)+ g'(x,y),
\end{align*}
This corresponds to writing 
\begin{equation}\label{decompo}
X_\gep=\tilde X_\gep+X'
\end{equation} 
where $X'$ and $\tilde X_\gep$  have respective covariance functions $\tilde g_\gep(x,y)$ and $g'(x,y)$.

The conformal radius $C(x,D)$ \index{conformal radius} of a point $x$ in the planar bounded domain $D$ is defined by
\begin{equation}\label{confrad}
C(x,D):=\frac{1}{|\varphi'(x)|}
\end{equation}
where $\varphi$ is any conformal mapping of $D$ to the unit disc such that 
$\varphi(x)=0$. In fact, we will use the following definition which is more useful for our purpose. Let $\varphi$ be any conformal map from $D$ to the upper half plane $\mathbb{H}$. Then we have the following expression for the conformal radius:
\begin{equation}\label{confrad2}
C(x,D)=\frac{2 \text{Im}(\varphi(x))}{|\varphi'(x)|}.
\end{equation}
Set 
$$\bar{G}_\gep(x,y):=e^{-g_\gep(x,y)}$$
 We stress here that $\bar{G}_\gep$ is not to be confused with $G_\gep$ of the previous sections. 

The following claim is proved in the Appendix.
\begin{lemma}\label{greenradius}
For all $x\in D$, we set $C_\gep(x,D)=\gep/\bar{G}_\gep(x,x)$. We have:
\begin{equation}
\lim_{\gep \to 0} C_\gep(x,x)=C(x,D)
\end{equation}
uniformly on the compact subsets of $D$.
\end{lemma}

\vspace{1mm}

We define for $(\gamma,\beta)\in\R_+^2$ and $\gep$ the following operator:
$$\bar{M}^{\gamma,\beta}_\gep(\varphi)=\int_D \varphi(x)  e^{\gamma X_\gep(x)+i\beta Y_\gep (x)}\bar{G}_\gep(x,x)^{\frac{\gamma^2-\beta^2}{2}} C(x,D)^{\frac{\gamma^2}{2}-\frac{\beta^2}{2}}\,\dd x.$$
where $\varphi(x)$ is a bounded measurable function on $D$. Notice that the renormalization term $\bar{G}_\gep(x,x)^{\frac{\gamma^2-\beta^2}{2}}$  is chosen such that $M^{\gamma,\beta}_\gep(\varphi)$ is a martingale in $\gep$.

Given another planar domain $\widetilde{D}$ and a conformal map $\psi:\widetilde{D}\to D$, we will denote by $(X^\psi_\gep)_{\gep\in]0,1]}$ and $(Y^\psi_\gep)_{\gep\in]0,1]}$ the random fields defined by 
$$X^\psi_\gep(x)=X_\gep(\psi(x))\quad \text{and}\quad Y^\psi_\gep(x)=Y_\gep(\psi(x)).$$
These two families form two independent white noise approximating sequences of the GFFs\index{Gaussian free field (GFF)}  $X\circ \psi$ and $Y\circ\psi$ defined on $\widetilde{D}$. Then we define for $\varphi$ defined on $\widetilde{D}$:
$$\bar{M}^{\gamma,\beta,\psi}_\gep(\varphi)=\int_{\widetilde{D}}\varphi(x) e^{\gamma X^\psi_\gep(x)+i\beta Y^\psi_\gep (x)} \bar{G}_\gep(\psi(x),\psi(x))^{\frac{\gamma^2-\beta^2}{2}}C(x,\tilde{D})^{\frac{\gamma^2}{2}-\frac{\beta^2}{2}}|\psi'(x)|^{2\gamma} \,\dd x.$$
When $\psi$ is the identity we simply write $\bar{M}^{\gamma,\beta}_\gep$.
This allows us to define simultaneously the GFF\index{Gaussian free field (GFF)} on every planar bounded domain conformally equivalent to $D$. We also mention the rule, for $x\in\tilde{D}$, $| C( \psi(x),D)  |= |\psi'(x)| | C(x,\tilde{D})  |$ where  $|\psi'(x)|^2$  is the Jacobian  of the mapping $\psi:\widetilde{D}\to D$ (this follows right away from the definition of the conformal radius \eqref{confrad}).

\subsubsection{Phase I and frontier I/II}

Consider a couple $(\gamma,\beta)\in\R_+\times \R$ and define 
\begin{equation*}
 \zeta(p)=\left(2+ \frac{\gamma^2}{2}-\frac{\beta^2} {2}\right)p -\frac{\gamma^2}{2}p^2.
 \end{equation*}
We have the following behavior inside phase I:
\begin{theorem}\label{maringo}
Consider a couple $(\gamma,\beta)\in\R_+\times \R$
in phase I or in the frontier I/II (excluding the extremal points). Consider $p\in (1,2]$ 
such that $\zeta(p)>2$ in the inner phase I or $p\in \left[1,\frac{2}{\gamma}\right)$ on the frontier I/II. Then:
\begin{enumerate}
\item For every bounded planar bounded domain $\widetilde{D}$ and conformal map $\psi:\widetilde{D}\to D$, 
for all bounded functions $\varphi$ defined on $\widetilde{D}$, the martingale:
\begin{equation*} 
 (\bar{M}_\gep^{\gamma, \beta ,\psi}(\varphi) )_\gep
\end{equation*}
is uniformly bounded in $\bbL_p$. 
\item Almost surely, the sequence $\bar{M}_\gep^{\gamma, \beta,\psi}(\cdot)$   converges in the space of distributions of order $2$ towards a limit we will denote by $M^{\gamma, \beta ,\psi}(\cdot)$. We set $M^{\gamma, \beta}(\cdot):=M^{\gamma, \beta,\psi}(\cdot)$ when $\psi$ is the identity map on $D$. The operator norm in the space of distributions of order $2$ of the limiting distribution  $M^{\gamma, \beta ,\psi}(\cdot)$ is $\bbL_p$-integrable. 
\item  For all $q\in [0,p]$ and all functions $\varphi\in C^2_c(\widetilde{D})$:
\begin{equation*}
\E  [  |  M^{\gamma, \beta,\psi}(\varphi(\cdot/r))  |^p ] \underset{r \to 0}{\sim} C_x r^{\zeta(p)}
\end{equation*}
for all $x\in \widetilde{D}$ and some constant $C_x>0$, which is continuous with respect to $x$ on $\widetilde{D}$. 
\end{enumerate}
\end{theorem}

\vspace{2mm}
\noindent {\it Proof.} The proofs of items 1,2,3 are exactly the same as in Section \ref{phaseI}: when computing the $\bbL_p$ norm,
thanks to \eqref{dominp}, we  can use Proposition \ref{kahaha} to compare the capacity with the one of the stationary case $k(x)=e^{-x^2/2}$. \qed

\begin{rem}
As a consequence of the above Theorem 
we have for any $\varphi\in C^2_c(\tilde{D})$,

$$\lim_{\gep\to 0} \gep^{\frac{\gamma^2}{2}-\frac{\beta^2}{2}} \int_{D}\varphi(x)e^{\gamma X_\gep(x)+i\beta Y_\gep(x)}\,\dd x=M^{\gamma, \beta}(\varphi),$$ 
in the $\bbL_p$ sense. 

\medskip

To see this, the reader can check that with our definitions
\begin{multline}
\gep^{\frac{\gamma^2}{2}-\frac{\beta^2}{2}} \int_{D}\varphi(x)e^{\gamma X_\gep(x)+i\beta Y_\gep(x)}\,\dd x -\bar{M}^{\gamma,\gb}_\gep(\varphi) \\
= \int_{D}\left(C_\gep(x,D)-C(x,D)\right)^{\frac{\gamma^2}{2}-\frac{\beta^2}{2}}  \varphi(x)e^{\gamma X_\gep^\psi(x)+i\beta Y_\gep^\psi(x)-\frac{\gamma^2-\gb^2}2 \bar{G}_{\gep}(x,x)}\,\dd x.
\end{multline}
As $ C_\gep(x,D)$ converges uniformly (Lemma \ref{greenradius}), Theorem \ref{phase1} (Item 1) shows that the moment of order $p$ of the above quantity tends to zero.

\end{rem}

\subsubsection{Gaussian Free Field on planar bounded domains: another approach in phase I and frontier I/II}

In this subsection, we also construct the limit $M^{\gamma, \beta}$ of theorem \ref{maringo} by other cut-off procedures, where the approximations $X_\gep$ and $Y_\gep$  are functions of $X$ and $Y$. In particular, this gives a construction of $M^{\gamma, \beta}$ as a function of $X$ and $Y$. This can be useful in applications (for couplings with SLE, etc...).

\paragraph{Circle averages.}
In fact, we can extend the framework of \cite{cf:DuSh} to the complex case in phase I and frontier I/II: this framework enables to construct the corresponding limits as functions of the GFF. Let $X$ and $Y$ be two independent GFFs\index{Gaussian free field (GFF)}  on a domain $D$. 
We introduce the circle averages $(X_\gep)_{\gep\in]0,1]}$ and $(Y_\gep)_{\gep\in]0,1]}$ of radius $\gep$, i.e. $X_\gep(x)$ (resp. $Y_\gep(x)$) stands for the mean value of $X$ (resp. $Y$) on the circle centered at $x$ with radius $\gep$ (cf. \cite{cf:DuSh} for further details). We then consider the  operator:
 $$\varphi\mapsto \tilde{M}^{\gamma,\beta}_\gep(\varphi)= \int_D\varphi(x) e^{\gamma X_\gep(x)+i\beta Y_\gep (x) -(\frac{\gamma^2}{2}-\frac{\beta^2}{2} )  \E[  X_\gep(x)^2   ]  }  C(x,D)^{\frac{\gamma^2}{2}-\frac{\beta^2}{2}} \,\dd x.$$
 Recall also that $ M^{\gamma,\beta}_\gep$ is given by the following expression:
 \begin{equation*}
 \varphi\mapsto M^{\gamma,\beta}_\gep(\varphi)= \int_D\varphi(x) e^{\gamma X_\gep(x)+i\beta Y_\gep (x) }\,\dd x.
 \end{equation*}

 We set $\tilde{G}_{\gep, \gep'}(x,y)=\E[ X_\gep(x) X_{\gep'}(y)   ] $. We can now state the following theorem:
\begin{theorem}\label{circle}
In the inner phase I, we consider $p\in ]1,2]$ such that $\zeta(p)>2$ and on the frontier I/II (excluding the extremal points), we consider $p\in]1,\frac{2}{\gamma}[$. For all  bounded measurable functions $\varphi$ with compact support in $D$, the family $\tilde{M}^{\gamma,\beta}_\gep(\varphi)$ converges in $\bbL_p$ towards a variable we will denote $M^{\gamma,\beta}(\varphi)$, since it has same law as the limit $M^{\gamma,\beta}(\varphi)$ of theorem \ref{maringo}. The sequence $( \gep^{\frac{\gamma^2}{2}-\frac{\beta^2}{2}}   M^{\gamma,\beta}_\gep(\varphi))_\gep$ converges also in $\bbL_p$ towards the same variable $M^{\gamma,\beta}(\varphi)$.
\end{theorem}

\proof
Here $M^{\gamma,\beta}_\gep$  is not a martingale so we cannot content ourselves with proving boundedness in $\bbL_p$: we must show that the sequence is Cauchy.
If $\gep, \gep'>0$, we have the following bounds:
\begin{align}
  &\E[  |     \tilde{M}^{\gamma,\beta}_\gep ([0,1]^2)- \tilde{M}^{\gamma,\beta}_{\gep'}( [0,1]^2 )   | ^p ] \nonumber\\
& \leq \E[    \E [ |\tilde{M}^{\gamma,\beta}_\gep ([0,1]^2)-\tilde{M}^{\gamma,\beta}_{\gep'}( [0,1]^2 )   | ^2  |  X]^{p/2} ]     \nonumber\\
& = \E[  |   A(\gep,\gep)+A(\gep',\gep')-2 A(\gep,\gep')   |^{p/2}  ]  \label{porcasse1} 
\end{align}
where we have set:
\begin{equation*}
A(\gep,\gep')= \int_{([0,1]^2)^2}   e^{\gamma X_\gep(x)  -\frac{\gamma^2}{2}  \E[  X_\gep(x)^2   ] } e^{  \gamma X_{\gep'}(y)  -\frac{\gamma^2}{2}  \E[  X_{\gep'}(y)^2   ]   } e^{\beta^2 \tilde{G}_{\gep, \gep'}(x,y) }  C(x,D)^{\frac{\gamma^2}{2}-\frac{\beta^2}{2}}  C(y,D)^{\frac{\gamma^2}{2}-\frac{\beta^2}{2}}   \dd x \dd y .
\end{equation*}
If $\delta>0$, we further define
\begin{align*}
A(\gep,\gep', \delta )=& \int_{([0,1]^2)^2,|x-y|\leq \delta}   e^{\gamma X_\gep(x)  -\frac{\gamma^2}{2}  \E[  X_\gep(x)^2   ] } e^{  \gamma X_{\gep'}(y)  -\frac{\gamma^2}{2}  \E[  X_{\gep'}(y)^2   ]   } e^{\beta^2 \tilde{G}_{\gep, \gep'}(x,y) }  C(x,D)^{\frac{\gamma^2}{2}-\frac{\beta^2}{2}}  C(y,D)^{\frac{\gamma^2}{2}-\frac{\beta^2}{2}}  \dd x \dd y\\
C(\gep,\gep', \delta )=& \int_{([0,1]^2)^2,|x-y|> \delta}   e^{\gamma X_\gep(x)  -\frac{\gamma^2}{2}  \E[  X_\gep(x)^2   ] } e^{  \gamma X_{\gep'}(y)  -\frac{\gamma^2}{2}  \E[  X_{\gep'}(y)^2   ]   } e^{\beta^2 \tilde{G}_{\gep, \gep'}(x,y) }  C(x,D)^{\frac{\gamma^2}{2}-\frac{\beta^2}{2}}  C(y,D)^{\frac{\gamma^2}{2}-\frac{\beta^2}{2}} \dd x \dd y . 
\end{align*} 
The main idea of what follows is the following: we split the integrals appearing in \eqref{porcasse1} in two regions $|x-y|\leq \delta$ and $|x-y|>\delta$ for some $\delta>0$. On the set $|x-y|> \delta$,  the singularity $e^{\beta^2 \tilde{G}_{\gep, \gep'}(x,y) }$ is bounded by a constant (eventually depending on $\delta$). Therefore, the convergence of the term
$$\E[  |   C(\gep,\gep,\delta)+C(\gep',\gep',\delta)-2 C(\gep,\gep',\delta)   |^{p/2}  ]$$
towards $0$ boils down to establishing the convergence of the family $(\tilde{M}^{\gamma,0}_\gep)$ in $\bbL^p$. This is "almost" proved in \cite{cf:DuSh}: actually, the authors in \cite{cf:DuSh} only prove almost sure convergence. On the other hand, it is plain to check (using Proposition \ref{kahaha} to get a comparison with a stationary field) that this family is uniformly bounded in $\bbL_q$ for some $q>p$. The claim of convergence in $\bbL_p$ follows. We deduce:
\begin{align*}
\limsup_{\gep, \gep' \to 0} &\E[ |    \tilde{M}^{\gamma,\beta}_\gep ([0,1]^2)-\tilde{M}^{\gamma,\beta}_{\gep'}( [0,1]^2 )   | ^p  ]  \\&\leq \limsup_{\gep, \gep' \to 0} \E[  |   A(\gep,\gep,\delta) |^{p/2} ] +  \E[ |A(\gep',\gep',\delta)|^{p/2}] + 2 \E[ | A(\gep,\gep',\delta)   |^{p/2}  ]   .
\end{align*}
By the capacity lemmas \ref{capI} or \ref{capI_II} (depending if we are in the inner phase I or the frontier I/II), the above quantity goes to $0$ as $\delta$ goes to $0$; therefore $(\tilde{M}^{\gamma,\beta}_\gep ([0,1]^2))_\gep$ is a Cauchy sequence in $\bbL_p$.

The fact that $( \gep^{\frac{\gamma^2}{2}-\frac{\beta^2}{2}}   M^{\gamma,\beta}_\gep(\varphi))_\gep$ converges also in $\bbL_p$ towards the same variable $M^{\gamma,\beta}(\varphi)$ is a consequence of the relation $\E[X_\gep(x)^2]= \ln \frac{1}{\gep}+ \ln C(x,D)+ o(1)$ as $\gep$ goes to $0$.

\qed

\paragraph{Orthonormal basis expansion of the GFF\index{Gaussian free field (GFF)}.} 

 As a preliminary, the reader is referred to \cite{She07,cf:DuSh} for further background about the expansion of the GFF\index{Gaussian free field (GFF)} in an orthonormal basis. Let us denote by $H(D)$ the Hilbert space closure of the space $C^\infty_c(D)$ with respect to the inner product
$$(f,g)_\nabla=\frac{1}{2\pi}\int_D\nabla f(x)\cdot\nabla g(x)\,\dd x.$$
Now we want to expand the GFF\index{Gaussian free field (GFF)} along a given orthonormal basis of $H(D)$ to produce another way of defining the limiting random variable $M^{\gamma,\beta}$. We will also show that the limit obtained with this procedure does not depend on the choice of the   orthonormal basis.

So we consider an orthonormal basis $(f_k)_{k\geq 1}$ of $H(D)$ made up  of continuous functions. We consider the projections of $X$ and $Y$ onto this orthonormal basis, namely we define the sequence of i.i.d. Gaussian random variables:
\begin{equation*}
\gep_k= \frac{1}{2\pi}\int_D \nabla X(x) \nabla f_k(x) \dd x,\quad \text{and}\quad  \gep'_k=\frac{1}{2\pi} \int_D \nabla Y(x) \nabla f_k(x) \dd x.
\end{equation*}
The projections  of $X$ and $Y$ onto the span of $\{f_1,\dots,f_n\}$ are given by:
$$X_n(x)= \sum_{k=1}^n  \gep_k f_k(x)\quad \text{and}\quad Y_n(x)= \sum_{k=1}^n  \gep'_k f_k(x).$$
In this context, we set:
\begin{equation}\label{defbon}
\bar{M}^{\gamma,\beta}_n(A)= \int_A e^{\gamma X_n(x)+i\beta Y_n  (x) -(\gamma^2/2-\beta^2/2) \E[X_n(x)^2]  } C(x,D)^{\frac{\gamma^2}{2}-\frac{\beta^2}{2}} \,\dd x.
\end{equation}

We have the following result: 
\begin{theorem}\label{cv:bon}
In the inner phase I, we consider $p\in ]1,2]$ such that $\zeta(p)>2$ and on the frontier I/II (excluding the extremal points), we consider $p\in]1,\frac{2}{\gamma}[$. For all  bounded measurable functions $\varphi$ with compact support in $D$, the sequence $(\bar{M}^{\gamma,\beta}_n(\varphi))_n$ converges almost surely and in $\bbL_p$ to $M^{\gamma,\beta}(\varphi)$, i.e. the same limit as the circle average approximations of Theorem \ref{circle}. 
\end{theorem}

\proof First observe that the sequence $(M^{\gamma,\beta}_n(\varphi))_n$ is a martingale uniformly bounded in $\bbL_p$ (because of Lemma \ref{capI} or Lemma \ref{capI_II} depending if we are in the inner phase I or the frontier I/II). Thus it converges almost surely and in $\bbL_p$.

Let $\gep>0$. For $n\geq 1$, we denote by $X_{n,\gep}(x)$ (resp. $Y_{n,\gep}(x)$)  the circle average of $X_{n}(x)$ (resp. $Y_{n}(x)$), i.e. the mean value of $X_n$ (resp. $Y_n$) along the circle centered at $x$ with radius $\gep$. For all $n \leq m$, we have the following:
\begin{align*}
& \E[    \int_{D}\varphi(x)  e^{\gamma X_{m,\gep}(x)+i\beta Y_{m,\gep}  (x) -(\gamma^2/2-\beta^2/2) \E[X_{m,\gep}(x)^2]  }   C(x,D)^{\frac{\gamma^2}{2}-\frac{\beta^2}{2}}  \,\dd x     |  (\gep_k,\gep'_k)_{k \leq n}   ] \\
& = \int_{D}\varphi(x)   e^{\gamma X_{n,\gep}(x)+i\beta Y_{n,\gep}  (x) -(\gamma^2/2-\beta^2/2) \E[X_{n,\gep}(x)^2]  }  C(x,D)^{\frac{\gamma^2}{2}-\frac{\beta^2}{2}} \,\dd x 
\end{align*}
 Now, we take the limit as $m\to \infty$ and get that:
\begin{equation*}
\E[   \tilde{M}^{\gamma,\beta}_\gep (\varphi)   |  (\gep_k,\gep'_k)_{k \leq n}   ] =  \int_{D}\varphi(x) e^{\gamma X_{n,\gep}(x)+i\beta Y_{n,\gep}  (x) -(\gamma^2/2-\beta^2/2) \E[X_{n,\gep}(x)^2]  }  C(x,D)^{\frac{\gamma^2}{2}-\frac{\beta^2}{2}} \,\dd x.
\end{equation*}
Since the variable $\tilde{M}^{\gamma,\beta}_\gep$ converges in $\bbL_p$, we can take the limit in the above identity as $\gep \to 0$ hence getting:
\begin{equation*}
\E[   M^{\gamma,\beta} (\varphi)   |  (\gep_k,\gep'_k)_{k \leq n}   ] =  \bar{M}^{\gamma,\beta}_n(\varphi) .
\end{equation*}
Now, we conclude that $\bar{M}^{\gamma,\beta}_n(\varphi) $ is a martingale bounded in $\bbL_p$ which converges to $M^{\gamma,\beta}(\varphi)  $. \qed

\subsubsection{Phases II and III,  frontier I/III and II/III}

 One can adapt the proofs of Theorems \ref{th3} and \ref{thfrontier23}
  to the case of the GFF\index{Gaussian free field (GFF)} with Dirichlet boundary condition. Here, we work with the approximations $X_\gep, Y_\gep$ given by \eqref{covar2} 
 
 \medskip
 
 Concerning the corresponding statements, we have   to specify what  the value of $\sigma$ and  the intensity measure are.
 It appears through the computations that the natural thing to do is to renormalize 
 $e^{\gamma X_\gep+i\gb Y_\gep}$ by a power of $\gep$ (i.e. by considering Wick ordering\index{Wick ordering}). Recall that
$$M^{\gamma,\beta}_\gep(A) =
\int_{A} e^{\gamma X_\gep (x)+i\beta Y_\gep (x)}\,\dd x$$ 
for all measurable bounded sets $A\subset D$. In this context, one can show (the computation of the constant $\tilde{\sigma}^2$ below is left to the reader and presents no special difficulty):
 
 \begin{theorem}\label{th3GFF}
\begin{itemize}
\item When $\gamma\in[0, 1[$ and $\gb^2+\gamma^2>2$, we have
 \begin{equation}\label{foum}
\left(\gep^{\gamma^2-1} M^{\gamma,\gb}_\gep(A)\right)_{A\subset \R^2} \Rightarrow \left(W_{\tilde{\sigma}^2  M^{2\gamma,0}}(A)\right)_{A\subset \R^2}.
\end{equation}
with
$$\tilde{\sigma}^2=\tilde{\sigma}^2(\gb^2+\gamma^2):=2\pi \int_0^\infty \exp\Big(-(\gamma^2+\beta^2)\int_0^1\frac{1-e^{-(ur)^2/2}}{u}\,\dd u\Big)\,{ r\dd r},$$
where recall that 
 \begin{equation*}
   M^{2\gamma,0}(\dd x):=\lim_{\gep\to 0} \gep^{2\gamma^2}e^{2\gamma X_\gep(x)}\dd x
 =C(x,D)^{2\gamma^2} \lim_{\gep\to 0}e^{2\gamma X_\gep(x)-2\gamma^2\bbE[X^2_\gep(x)]}\dd x
 \end{equation*}
 and $W _{\tilde{\sigma}^2  M^{2\gamma,0}} $ is a standard complex Gaussian measure on $\R^2$ with intensity $\tilde{\sigma}^2 M^{2\gamma,0}$. The above convergence holds in the sense of convergence in law of the finite dimensional marginals.
\item  When $\gamma\in[0,1[$ and $\gb^2+\gamma^2=2$, we have
 \begin{equation}
\left(\gep^{\gamma^2-1}|\log \gep|^{-1/2} M^{\gamma,\gb}_\gep(A)\right)_{A\subset \R^2} \Rightarrow \left(W_{\tilde{\sigma}^2  M^{2\gamma,0}}(A)\right)_{A\subset \R^d}.
\end{equation}
with 
$$\tilde{\sigma}^2=\tilde{\sigma}^2(2):= 2\pi\exp\left(\int_0^\infty \frac{e^{-u^2/2}-\ind_{[0,1]}(u)}{u} \dd u \right).
$$
and $W_{  \tilde{\sigma}^2 M^{2\gamma,0}}$ is a standard complex Gaussian measure on $\R^2$ with intensity $\tilde{\sigma}^2 M^{2\gamma,0}$. The above convergence holds in the sense of convergence in law of the finite dimensional distributions.
\end{itemize}
 \end{theorem}

\begin{theorem}\label{thfrontier23GFF}
When $\gamma= 1$ and $\gb^2+\gamma^2>2$, we have
 \begin{equation}
\left((-\ln \gep)^{1/4} M^{\gamma,\gb}_\gep(A)\right)_{A\subset\R^2} \Rightarrow \left(W_{\tilde{\sigma}^2  M'}(A)\right)_{A\subset\R^2}.
\end{equation}
with  $$\tilde{\sigma}^2=\sqrt{\frac{2}{\pi}} \tilde{\sigma}^2(\gb^2+1).$$  
Convergence holds in the sense of convergence in law of the finite dimensional distributions and the law of $W_{\tilde{\sigma}^2  M'}(\cdot)$ is  that of a complex Gaussian random measure with intensity $\tilde{\sigma}^2  M' $ where 
\begin{equation*}
  M'(\dd x):=\lim_{\gep\to 0}\sqrt{\frac{\pi}{2}} \gep^{2}|\log \gep|^{1/2}e^{2 X_\gep(x)}\dd x=C(x,D)^{2} \lim_{\gep\to 0}[2\bbE[X^2_\gep(x)]-X_\gep(x)]  e^{2 X_\gep(x)-2\bbE[X^2_\gep(x)]}\dd x.
 \end{equation*}
\end{theorem}
 
 \begin{rem}  Note that the expression that we find for $\tilde{\sigma}$ is that of $\sigma$ in Theorem \ref{th3} where $k$ is taken equal to $e^{-u^2/2}$.
 \end{rem}

\subsubsection{Other GFFs\index{Gaussian free field (GFF)}}
There are other possible choices of the underlying GFF\index{Gaussian free field (GFF)}. One may for instance consider Neumann boundary conditions instead of Dirichlet's, or consider a GFF\index{Gaussian free field (GFF)} with vanishing mean on the sphere or the torus. Since everything works essentially the same as in the case of the GFF with Dirichlet boundary conditions, we will not give the details here.

%

\section{Applications in $2D$-string theory}\label{LQG}

\subsection{Introduction} 
Euclidian  quantum gravity is an attempt to quantize general relativity 
based on Feynman's functional integral and on the 
Einstein-Hilbert action.   The main motivation for considering $2D$ (i.e. two-dimensional) quantum gravity comes from the fact  that the Einstein-Hilbert action becomes trivial in $2D$ as it reduces to a topological term via the Gauss-Bonnet theorem. 
The basic approach is to couple a Conformal Field Theory  (CFT)  with central charge $c$ to gravity. A famous example is the coupling of $c$ free scalar matter fields to gravity (in this case, $c$ is an integer), leading to an interpretation of such a specific theory of $2D$-Liouville Quantum Gravity as a bosonic string  theory   in $c$ dimensions \cite{Pol}.
 
 The following discussion will focus on the coupling of one free scalar matter field to gravity: the CFT (here the free scalar field) is then said to have a central charge \index{central charge} $c=1$ and this corresponds to the bosonic string \index{bosonic string} in $1$ dimension, or equivalently $2D$ string theory. For a central charge $c=1$, it is shown in \cite{Pol,cf:KPZ,cf:Da} that the action of the bosonic string in one dimension factorizes as a tensor product: the fluctuations of the metric are governed by the Liouville quantum field theory (LFT for short, see \cite{DKRV} for a mathematical construction) and are  independent of the CFT. More precisely,  the random metric    takes on the form \cite{Pol,cf:KPZ,cf:Da} (we consider an Euclidean background metric for simplicity)
\begin{equation}\label{tensor}
g(x)=e^{2 X(x)}dx^2,
\end{equation}
where the   field $X$ is a Liouville field  and the CFT becomes an independent free field, which we may call $Y$. When neglecting the cosmological effects (or treating the LFT as a free  field theory),  the Liouville action turns the field $X$ in \eqref{tensor} into a Free Field, with appropriate boundary conditions. For an extensive review on $2D$ string theory, 
we refer to Klebanov's lecture notes \cite{Kle}. As expressed by Klebanov in \cite{Kle}: "Two-dimensional string theory is the kind of toy model which possesses a remarkably simple structure but at the same time incorporates 
some of the physics of string theories embedded in higher dimensions". The reader is also referred to \cite{david-hd,cf:Da,cf:DuSh,DFGZ,DistKa,bourbaki,LBM,GM,glimm,Kle,cf:KPZ,Nak,Pol} for more insights on $2D$-Liouville quantum gravity. 

 Therefore, in the following section, we first  review the basic notions of CFT with central charge $c=1$ 
 on its own (Section \ref{oio}) and then  coupled to gravity, i.e. 
 two-dimensional string theory (Section \ref{apres}). In particular, we show that our work enables to define mathematically the so-called Tachyon fields.

\subsection{Conformal Field theory with central charge $c=1$}\label{oio}
 We consider a domain $D$ and a GFF\index{Gaussian free field (GFF)} $Y$ on $D$ with Dirichlet boundary conditions. In the physics literature, one considers the conformally invariant action
   \begin{equation}\label{ActionGFF} 
  S(Y)= \frac{1}{4 \pi}\int_D | \nabla Y(x) | d^2x 
  \end{equation}
 and all averages of functionals $F(Y)$ are denoted formally as
 \begin{equation*}
 \E[F(Y)  ]= \frac{\int F(Y) e^{-S(Y)} \dd Y}{\int e^{-S(Y)} \dd Y}.
 \end{equation*}
Note that the normalization in the definition of $S$ ensures that
\begin{equation*}
\E[Y(y)Y(x)] \underset{|y-x| \to 0}{\sim} \ln \frac{1}{|y-x|}. 
\end{equation*}
In this context, Conformal Field Theory (CFT) \index{conformal field theory (CFT)} with central charge  \index{central charge} $c=1$ involves defining and studying operators (or fields) formally constructed as functions of the GFF\index{Gaussian free field (GFF)}: see the mathematically oriented article \cite{KanMak} for more on this.  Since the GFF\index{Gaussian free field (GFF)} is a distribution (generalized function), this is often not straightforward mathematically. Of particular importance are the so-called vertex operators  \index{vertex operator} denoted $V_\alpha= e^{\alpha Y}$  ($\alpha\in \C$) by physicists  that we will rather denote formally in the following way
\begin{equation}\label{eqformel}    
V_{\alpha}(Y(x),x)= C(x,D)^{\frac{\alpha^2}{2}}e^{ \alpha Y(x) -\frac{\alpha^2}{2}  \E[  Y(x)^2 ] }
\end{equation}
where $C(x,D)$ is the conformal radius. This formal definition is more accurate to denote what physicists of CFT or Quantum gravity call normal or Wick ordering \index{Wick ordering} of $e^{\alpha Y}$ (in other fields, in the Wick ordering \index{Wick ordering} of $e^{\alpha Y}$, the conformal radius does not appear in expression \ref{eqformel}). When integrated under this form, we recognize Gaussian multiplicative chaos if $\alpha$ is a nonnegative real less than $2$. If $\alpha$ is complex with $|\alpha|<\sqrt{2}$, one can define $V_{\alpha}$ as  a random distribution (see \cite{KanMak}). The construction is obvious because $|\alpha|<\sqrt{2}$ ensure the fields are $\bbL_2$-integrable. 

In this paper and in the following discussion, we consider the case $\alpha=i \beta$ where $\beta$ is real.   
The conformal dimension of $V_{i \beta}(Y,z)$ measures how the field changes when one switches to another parametrization. More precisely, let $\psi:  \tilde{D} \rightarrow D$ be a conformal map and set $\tilde{Y}(\tilde{x})= (Y \circ \psi)(\tilde{x})$. Since the action (\ref{ActionGFF}) maps $Y \rightarrow \tilde{Y}$ under $\psi$, the conformal dimension  \index{conformal dimension} $\Delta_{i\beta}$ of $V_{i \beta}$ is defined by  
\begin{equation}\label{eqdimformelle}
 V_{i \beta}(Y( \psi(\tilde{x})), \psi(\tilde{x}))= |\psi'(\tilde{x})|^{-2 \Delta_{i \beta}} V_{i \beta}(\tilde{Y} (\tilde{x}), \tilde{x}) . 
\end{equation}
In fact, this corresponds to the definition of conformal dimension for spinless operators.
 By the rule $| C( \psi(\tilde{x}),D)  |= |\psi'(\tilde{x})| | C(\tilde{x},\tilde{D})  | $, we get that $\Delta_{i \beta}=\frac{\beta^2}{4}$. In particular, we get the following rigorous relation by integrating the formal relation  (\ref{eqdimformelle}) for some compact set $\tilde{K}$ ($K=\psi(\tilde{K})$)
 \begin{align*} 
\int_{\tilde{K}}  V_{i \beta}(\tilde{Y} (\tilde{x}), \tilde{x}) d \tilde{x} & = \int_{\tilde{K}}  V_{i \beta}(Y( \psi(\tilde{x})), \psi(\tilde{x})) |\psi'(\tilde{x})|^{2 \Delta_{i \beta}}   d \tilde{x}       \\ 
& =  \int_{K}  V_{i \beta}(Y(x), x) |\psi'(\psi^{-1}(x))|^{2 \Delta_{i \beta}-2}   d x   
\end{align*}
where $K$ is some compact set.

\subsection{CFT with central charge $c=1$ coupled to gravity}\label{apres}
 In the special case $c=1$ (hence the case of a GFF\index{Gaussian free field (GFF)} $Y$ with Dirichlet boundary conditions on some domain $D$), the Polyakov action \index{Polyakov action} 
 can be written in the following tensor form
 \begin{equation}\label{actionPol} 
  S(X,Y)= \frac{1}{4 \pi}\int_D | \nabla Y(x) |^2 \dd^2x + \frac{1}{4 \pi} \int_D | \nabla X (x) |^2+  Q R(x)  X(x) \dd^2x
  \end{equation}      
where the first term is the classical GFF\index{Gaussian free field (GFF)} action (CFT \index{conformal field theory (CFT)} with $c=1$) and the second is the classical Liouville action (where we have set the cosmological constant to $0$, $R$ is the curvature and $Q=2$). 

Following the physics literature (see \cite{DistKa,GM,Nak}), we consider the following equivalence class of random surfaces: if $\psi:  \tilde{D} \rightarrow D$ is a conformal map then we get the following rule for the fields $(X,Y)$:
\begin{equation*}
  (X,Y) \rightarrow (X \circ \psi +2 \ln | \psi' | , Y \circ \psi).
\end{equation*}
This equivalence class is a generalization of \cite{cf:DuSh} in the sense that now we incorporate  the matter field $Y$. The reparametrization rule for $Y$ is a just a consequence of the conformal invariance of the Free Field, i.e. the action $\int_D | \nabla Y(x) | d^2x$ is conformally invariant. 
Since Liouville Quantum Gravity \index{Liouville Quantum Gravity} is a conformal field theory\index{conformal field theory (CFT)}, the relevant operators are the ones which are invariant under the above rule; this ensures that the operators are stable under reparametrization, i.e. are independent of the underlying background metric which is used to define the theory. For consistency reasons (no conformal anomaly), one can only consider conformally invariant dressed operators within Liouville Quantum Gravity\index{Liouville Quantum Gravity}. The simplest of such operators are    the so-called
{\bf tachyon fields}  \index{tachyon field} (see the reviews \cite{Kle,Nak}). More precisely, the tachyon fields \index{tachyon field} are the CFT vertex operators \index{vertex operator} $e^{i\beta Y }$ with gravitational dressing of the form $e^{\gamma X}$ that are conformally invariant under the action (\ref{actionPol}). Hence they are formally of the form $e^{\gamma X(x)+i\beta Y(x) }$. 
 Here, we stress the fact that $e^{\gamma X(x)+i\beta Y(x) }$ is a function of the two GFFs\index{Gaussian free field (GFF)}  $X,Y$ in order to determine the way it 
changes under reparametrization. In what follows, we will thus consider the Wick ordering of the field   $e^{\gamma X(x)+i\beta Y(x) }$, i.e. the limit $\lim_{\gep\to 0} \gep^{\frac{\gamma^2}{2}-\frac{\beta^2}{2}} e^{\gamma X_\gep(x)+i\beta Y_\gep(x)}\,\dd x$ (when it exists) and see if it properly defines a conformally invariant operator  under the action (\ref{actionPol}).

\subsubsection{The special point $\gamma=2,\beta=0$}
At the special point $(\gamma=2,\beta=0)$, we recover the special tachyon field \index{tachyon field}
$$M^{\gamma,\beta}_{X,Y}(A)= M'(A)$$
where $M'$ is the derivative martingale\index{derivative multiplicative chaos} defined in \cite{Rnew7, Rnew12}:
$$M'(dz)= \, C(z,D)^2 (2\E[X (z)^2]-X )e^{2X (z)-2\E[X (z)^2]}\,dz.$$

\subsubsection{Tachyons within phase I and frontier I/II}
It is natural to first look for  other tachyons fields \index{tachyon field} in phase I together with the frontier of phases I/II (excluding the extremal points) because in this case the Wick ordering \index{Wick ordering} of the field $e^{\gamma X(x)+i\beta Y(x) } $ converges to $M^{\gamma, \beta}$ defined in subsection \ref{GFFplanar}. In order to stress the dependence of $M^{\gamma, \beta}$ with respect to $X,Y$, we denote it by $M^{\gamma,\beta}_{X,Y}$. In view of subsection \ref{GFFplanar}, it is natural to give the following formal expression to $M^{\gamma,\beta}_{X,Y}$:    
$$M^{\gamma,\beta}_{X,Y}(\dd x)=e^{\gamma X(x)+i\beta Y(x) -\frac{\gamma^2}{2}\E[X(x)^2]+\frac{\beta^2}{2}\E[Y(x)^2]}   C(x,D)^{\frac{\gamma^2}{2}-\frac{\beta^2}{2}}  \dd x.$$
Now, we are looking for the couples $(\gamma, \beta)$   satisfying $$M_{X,Y}^{\gamma,\beta} ( \varphi\circ \psi^{-1}))= M^{\gamma,\beta}_{X \circ \psi +2 \ln | \psi' |,Y\circ \psi} (\varphi)$$ for every function $\varphi\in C^2_c(\tilde{D})$. 

  Now, we get that:
  \begin{align*}
  & M_{X \circ \psi +2 \ln | \psi' |,Y \circ \psi}^{\gamma,\beta} (\varphi) \\
  & =  \int_{\tilde{D}}\varphi(\tilde{x}) e^{\gamma (X(\psi(\tilde{x})) + 2 \ln | \psi'(\tilde{x}) |  )+ i \beta Y(\psi(\tilde{x}))-  \frac{\gamma^2}{2} \E[X(\psi(\tilde{x}))^2]+ \frac{\beta^2}{2} \E[Y(\psi(\tilde{x}))^2]} | C(\tilde{x},\tilde{D})  |^{\gamma^2/2- \beta^2/2} d^2 \tilde{x}  \\
  & =  \int_{\tilde{D}}\varphi(\tilde{x}) e^{\gamma X(\psi(\tilde{x}))+ i \beta Y(\psi(\tilde{x}))- \frac{\gamma^2}{2} \E[X(\psi(\tilde{x}))^2]+\frac{\beta^2}{2} \E[Y(\psi(\tilde{x}))^2]} | C(\psi(\tilde{x}),D)  |^{ \frac{\gamma^2}{2}- \frac{\beta^2}{2}}   |\psi'(\tilde{x})|^{  2\gamma - \frac{\gamma^2}{2}+ \frac{\beta^2}{2}  }   d^2 \tilde{x} \\
 & =  \int_{D} \varphi(\psi^{-1}(x)) e^{\gamma X(x)+ i \beta Y(x)-  \frac{\gamma^2}{2} \E[X(x)^2]+\frac{\beta^2}{2}\E[Y(x)^2]} | C(x,D)  |^{ \frac{\gamma^2}{2}- \frac{\beta^2}{2}}   |\psi'(\psi^{-1}(x))|^{ 2 \gamma-  \frac{\gamma^2}{2}+ \frac{\beta^2}{2} -2  }   d^2 x.
  \end{align*}
  Though the above computation is formal, it is not difficult to make it rigorous by making the same computation with the regularized fields $X_\gep,Y_\gep$ of subsection \ref{GFFplanar} and then taking the limit as $\gep$ goes to $0$.  
  
  Hence, for the field  to be conformally invariant, one must solve the equation $2 \gamma = \gamma^2/2- \beta^2/2+2$, yielding the points of  the frontier of phases I/II: $\gamma\pm\beta=2$ for $\gamma \in]1,2[$). This partially confirms predictions of physicists (see \cite{Kle} and references therein) in the sense that it was claimed that we get tachyons on the segment $\gamma\pm\beta=2$ for $\gamma \in [0,2[$). Yet, our approach does not allow us to interpret the tachyon fields  as random distributions (in the sense of Schwartz)  for the values $\gamma\pm\beta=2$ and $\gamma \in [0,1]$ by this Wick ordering procedure as   we enter phase III and the triple point, which involves non standard renormalization that highly perturbs the behaviour of the exponent of the conformal radius, yielding operators that are not conformally invariant. 

\subsubsection{KPZ formula for the tachyon fields}
Consider the vertex operators  \index{vertex operator}   $V_{i\beta}= e^{i\beta Y}$, with $\beta\in [0,1)$, of a Conformal Field Theory (CFT) \index{conformal field theory (CFT)} with central charge  \index{central charge} $c=1$, i.e. $Y$ is a GFF. We have seen that the conformal dimension  \index{conformal dimension} $\Delta^0_{i\beta}$ of $V_{i \beta}$ is $\frac{\beta^2}{4}$.  The quantum dimension \index{quantum dimension} $\Delta^q_{i\beta} $ of the operator  $V_{i\beta}$ is defined as the value such that the operator
$$e^{2(1-\Delta^q_{i\beta})X}V_{i\beta}$$ is conformally invariant within the theory, i.e. becomes a tachyon field. We have seen that we must choose 
$$2(1-\Delta^q_{i\beta})+\beta=2$$ in order for this field to be conformally invariant, yielding
$$\Delta^q_{i\beta}=\frac{\beta}{2}.$$
We thus recover the celebrated KPZ formula \index{KPZ formula}(see \cite{cf:KPZ}) 
$$\Delta^0_{i\beta}=\Delta^q_{i\beta}+\Delta^q_{i\beta}(\Delta^q_{i\beta}-1) $$
in its original derivation for (critical) Liouville quantum gravity with a $c=1$ ($b=2$) central charge.
  

 \appendix 
 \appendixpage
\vspace{2mm}

\section{Control of moments of order $2k$}  \label{Controlmom}

In this appendix, we gather technical estimates on the convergence of the higher order moments in phase III and its frontiers with the other phases. The main purpose of this appendix is to prove proposition \ref{clesmoments} below. The following results have straightforward analogs in all dimensions: for the sake of clarity, we state and prove them in dimension 1.

\begin{proposition}\label{propclesmoments}
Let $k$ and $k'$ be natural integers. Then for any interval $J$, in the phase $III$, $I/III$ and $II/III$ 
we have the following convergence in probability:

\begin{equation}\label{clesmoments}
\lim_{\gep\to 0} \frac{ \bbE\left[(M^{\gamma,\gb}_\gep(J))^k (\overline{ M^{\gamma,\gb}_\gep(J)})^{k'} \ | \ \mathcal{F}^X  \right]}
{\bbE \left[ M^{\gamma,\gb}_\gep(J)\overline {M^{\gamma,\gb}_\gep(J)}  \ | \  \mathcal{F}^X  \right]^{k+k'/2}}=k! \ind_{k=k'},
                                                  \end{equation}
where $\overline{M^{\gamma,\gb}_\gep(J)}$ denotes the complex conjugate of $M^{\gamma,\gb}_\gep(J)$.                                               

For all $l \geq 2$ and all $2l$-tuple of natural integers $(k_1,\dots,k_{2l})$, for any collection of disjoint intervals $J_1, \cdots, J_l$, we have the following convergence in probability: 
\begin{multline}\label{clesmoments2}
\lim_{\gep\to 0} \frac{ \bbE\left[ \prod_{1 \leq i \leq l}(M^{\gamma,\gb}_\gep(J_i))^{k_{2i-1}} (\overline{ M^{\gamma,\gb}_\gep(J_i)})^{k_{2i}}
\ | \ \mathcal{F}^X  \right]}
{\prod_{1 \leq i \leq l} \E \left[ M^{\gamma,\gb}_\gep(J_i) \overline {M^{\gamma,\gb}_\gep(J_i)}  \ | \  \mathcal{F}^X  \right]^{(k_{2i-1}+k_{2i})/2}}= 
\prod_{1 \leq i \leq l} \:  k_{2i-1}! \ind_{k_{2i-1}= k_{2i}}   
.
                                                  \end{multline}
                                                  
\end{proposition}

\subsection{Optimal matching between two finite sets in $\bbR^d$}
Before starting the proof, we introduce a matching procedure. This algorithm was introduced by Gale and Shapley \cite{cf:Marriage} to provide a solution to the stable marriage problem.

Given $k\le k'$,
let $(x_1,\dots,x_k)$ and $(y_1,\dots,y_{k'})$ be two sets of points in $\bbR^d$ such that all the pairwise distances 
$|x_i-y_j|$ are distinct.
The {\bf optimal matching} of ${\bf x}$ with ${\bf y}$ is an injective application:
$$\sigma({\bf x},{\bf y}): \{1,\dots,k\}\to \{1,\dots,k'\}$$ 
obtained by the following procedure:
\begin{itemize}
 \item [(i)] If $x_i$ and $y_j$ are mutually closest, i.e. if: 
 $$\forall i', j' \:  |x_i-y_j|<|x_{i'}-y_j| \text{ and } |x_i-y_j|<|x_{i}-y_{j'}|$$
 then we set $\sigma(i)=j$.
 \item [(ii)] We delete the points that have been matched in step $(i)$.
 \item [(iii)] We iterate the procedure until all the $x_i$'s have been matched.
\end{itemize}

\subsection{Proof of proposition \ref{propclesmoments} for matching indices $k=k'$}
We prove (\ref{clesmoments}) for $J=[0,1]$. The general case can be proved along the same lines. We introduce the following notation:

$$M^{\gamma,\gb}_\gep (\dd {\bf x}):= \prod_{j=1}^k M^{\gamma,\gb}_\gep (\dd x_j)$$

The moment $\bbE\left[\left| M^{\gamma,\gb}_\gep ([0,1])\right|^{2k}\ | \  \mathcal{F}^X  \right]$ is given by
 
\begin{equation}\label{theintegral}
\gep^{k\gb^2}\int_{[0,1]^{2k}} 
\frac{M^{\gamma,0}_\gep(\dd {\bf x}\dd {\bf y}) 
\left(\prod_{1\le i<j\le k}G_\gep(x_i-x_j) \prod_{1\le i<j\le k}G_\gep(y_i-y_j)\right)^{\gb^2}}{\prod_{i,j=1}^k G_\gep(x_i-y_j)^{\gb^2}}.
\end{equation}

Hence \eqref{clesmoments} for $k=k'$ corresponds to proving 

\begin{equation}\label{wouhou}
\lim_{\gep\to 0} \frac{\int_{[0,1]^{2k}} 
\frac{M^{\gamma,0}_\gep(\dd {\bf x}\dd {\bf y}) 
\left(\prod_{1\le i<j\le k}G_\gep(x_i-x_j) \prod_{1\le i<j\le k}G_\gep(y_i-y_j)\right)^{\gb^2}}{\prod_{i,j=1^k} G_\gep(x_i-y_j)^{\gb^2}}}
{\left(\int_{[0,1]^{2}} M^{\gamma,0}_\gep(\dd { x}\dd { y})  G_\gep(x-y)^{-\gb^2}\right)^k}=k!.
\end{equation}

Let  $\sigma({\bf x},{\bf y})$ be the permutation obtained by the optimal matching procedure described in the previous section (which is Lebesgue almost-everywhere 
well defined) and set

$$B:=\{ ({\bf x},{\bf y}) \ | \  \sigma({\bf x},{\bf y})=\ind \},$$
where $\ind$ denotes the identity.
By symmetry of the indices, we can rewrite \eqref{theintegral} (divided by $\gep^{k\gb^2}$) as 
 
 \begin{equation}
 k!\int_B 
 \frac{M^{\gamma,0}_\gep(\dd {\bf x}\dd {\bf y}) \left(\prod_{1\le i<j\le k}G_\gep(x_i-x_j) \prod_{1\le i<j\le k}G_\gep(y_i-y_j)\right)^{\gb^2}}
 {\prod_{i,j=1}^k G_\gep(x_i-y_j)^{\gb^2}}.
 \end{equation}

A reformulation of \eqref{wouhou} is that one can find a $\delta$ that tends to zero with $\gep$ which is such that for all $k$ with high probability
\begin{multline}\label{lagrossineq}
(1-\delta)\left( \int_{[0,1]^2} G_\gep(x-y)^{-\gb^2} M^{\gamma,0}_\gep(\dd { x}\dd { y})\right)^k\\
\le \int_{B} \frac{   M^{\gamma,0}_\gep(\dd {\bf x}\dd {\bf y}) \left(\prod_{1\le i< j\le k}
G_\gep(x_i-x_j) \prod_{1\le i<j\le k}G_\gep(y_i-y_j)\right)^{\gb^2}}{\prod_{i,j=1}^k G_\gep(x_i-y_j)^{\gb^2}}
\\\le (1+\delta)\left( \int_{[0,1]^2} G_\gep(x-y)^{-\gb^2} M^{\gamma,0}_\gep(\dd { x}\dd { y})\right)^k.
\end{multline}
Let us consider functions $a(\gep)$ and $ b(\gep)$ such that
 $\gep \ll a(\gep)\ll b(\gep)\ll 1$. We set
\begin{equation} \label{defa}
A:= \{√Ç¬†({\bf x},{\bf y})\in [0,1]^{2k}\  | \forall i,\ |x_i-y_i|\le a(\gep), \  \forall i\ne j,\  |x_i-x_j|\ge b(\gep),\ |y_i-y_j|\ge b(\gep)\}.
\end{equation}

Note that $A \subset B$ (because when $({\bf x},{\bf y})\in A$ all the pairs $(x_i,y_i)$ are mutually closest)

\begin{lemma}\label{lezinek}
When $({\bf x},{\bf y})\in A$ we have
\begin{multline}\label{ineqa}
(1-\delta)  \prod_{i=1}^k G_\gep(x_i-y_i)^{-\gb^2}\\
 \le \frac{\prod_{1\le i<j\le k}G_\gep(x_i-x_j)^{\beta^2} \prod_{1\le i<j\le k}G_\gep(y_i-y_j)^{\beta^2}}{\prod_{i,j=1}^k G_\gep(x_i-y_j)^{\gb^2}}
 \\
 \le (1+\delta)  \prod_{i=1}^k G_\gep(x_i-y_i)^{-\gb^2}
\end{multline}
where $\delta=\delta(\gep)$ tends to zero when $\gep$ does.

For $({\bf x},{\bf y})\in B$ we have a general upper bound
\begin{equation}\label{ineqB}
\frac{\prod_{1\le i<j\le k}G_\gep(x_i-x_j)^{\beta^2} \prod_{1\le i<j\le k}G_\gep(y_i-y_j)^{\beta^2}}{\prod_{i,j=1}^k G_\gep(x_i-y_j)^{\gb^2}}
 \le C(k,\gb) \prod_{i=1}^k G_\gep(x_i-y_i)^{-\gb^2}.
\end{equation}
\end{lemma}

\begin{proof}
We prove \eqref{ineqB} by induction on $k$.
We can assume that $x_1=y_1$ are mutually closest (there is at least one pair of mutually closest vertices and we exchange the indices if needed).
Then we have for all $j \in \lbrace 2,\dots,k \rbrace$
\begin{equation}\label{grominet}\begin{split}
 |y_1-y_j|&\le |x_1-y_1|+|x_1-y_j|\le 2|x_1-y_j|\\
 |x_1-x_j|&\le |x_1-y_1|+|y_1-x_j|\le 2|y_1-x_j|.
\end{split}\end{equation}
We have from Lemma \ref{greenk}
$$\sup_{s\in [0,1],t\le 2s,\gep\le 1}\frac{G_{\gep}(t)}{G_{\gep}(s)}=C_1<\infty,$$
and hence, using the inequalities \eqref{grominet} we obtain that 
\begin{multline}\label{ineqBinter}
\frac{\prod_{1\le i<j\le k}G_\gep(x_i-x_j) \prod_{1\le i<j\le k}G_\gep(y_i-y_j)}{\prod_{i,j=1}^k G_\gep(x_i-y_j)}
\\
 \le C_1^{2(k-1)} \frac{1}{G_\gep(x_1-y_1)}
\frac{ \prod_{2\le i<j\le k}G_\gep(y_i-y_j)}{\prod_{i,j=2}^k G_\gep(x_i-y_j)}.
\end{multline}
Then using the induction hypothesis, (note that the identity is still the optimal matching once $(x_1,y_1)$ have been deleted)
we obtain \eqref{ineqB} with $C(k,\gb)=C_1^{\gb^2k(k-1)}$.

\medskip

The inequality \eqref{ineqa} is obtained by noticing that on the set $A$ for all $i<j$
\begin{equation}\label{petitminet}\begin{split}
 \big | |y_i-y_j|-|y_i-x_j|\big|&\le |x_j-y_j|\le a(\gep)\\
 \big | |x_i-x_j|-|x_i-y_j|\big|&\le |x_j-y_j|\le a(\gep)
\end{split}\end{equation}

We have from Lemma \ref{greenk}
$$\lim_{\gep\to 0}\sup_{s\ge b(\gep), t\in (s-a(\gep),s+a(\gep))}\left| \frac{G_{\gep}(s)}{G_{\gep}(t)}-1 \right|=0.$$
Let $\delta_1(\gep)$ be the quantity in the limit. Using \eqref{petitminet} we obtain \eqref{ineqa} with
$(1\pm \delta_1(\gep))^{\gb^2k(k-1)}$ instead of $(1\pm \delta)$.
\end{proof}

Now, we state the following lemma whose proof is postponed to the next subsection.

\begin{lemma}\label{lesdebris}
For all $k$, we have the following convergence in probability 
\begin{equation}\label{michto}
\lim_{\gep\to 0}\frac{\int_{[0,1]^{2k}\setminus A}
   M^{\gamma,0}_\gep(\dd {\bf x}\dd {\bf y})\prod_{i=1}^k G_\gep(x_i-y_i)^{-\gb^2}}{\left(\int_{[0,1]^2} M^{\gamma,0}_\gep(\dd { x}\dd { y})G_\gep(x-y)^{-\gb^2}\right)^k}=0.
 \end{equation}
 \end{lemma}
 
 With this lemma and Lemma \ref{lezinek}, we can conclude the proof of \eqref{clesmoments}. Let us first prove the lower bound in \eqref{lagrossineq}. 
First we replace the domain integration $B$ by $A$ which is smaller. Then we use \eqref{ineqa} and obtain that 
\begin{align*}
\int_{B}& \frac{   M^{\gamma,0}_\gep(\dd {\bf x}\dd {\bf y}) \left(\prod_{1\le i< j\le k}G_\gep(x_i-x_j) \prod_{1\le i< j\le k} G_\gep(y_i-y_j)\right)^{\gb^2}}{\prod_{i,j=1}^k G_\gep(x_i-y_j)^{\gb^2}}\\
&\ge (1-\delta) \int_A  M^{\gamma,0}_\gep(\dd {\bf x}\dd {\bf y}) \prod_{i=1}^k G_\gep(x_i-y_i)^{-\gb^2}\\ &\ge 
(1-\delta') \left(\int_{[0,1]^2} M^{\gamma,0}_\gep(\dd { x}\dd { y})G_\gep(x-y)^{-\gb^2}\right)^k.
\end{align*}
where the last line holds with high probability according to Lemma \ref{lesdebris}.

\medskip

For the upper bound in \eqref{lagrossineq}, we remark that from \eqref{ineqa} we have:
\begin{multline}
\int_{A} \frac{   M^{\gamma,0}_\gep(\dd {\bf x}\dd {\bf y}) \left(\prod_{1\le i< j\le k}
G_\gep(x_i-x_j) \prod_{1\le i<j\le k}G_\gep(y_i-y_j)\right)^{\gb^2}}{\prod_{i,j=1}^k G_\gep(x_i-y_j)^{\gb^2}}\\
\le (1+\delta)\int_{A} M^{\gamma,0}_\gep(\dd {\bf x}\dd {\bf y})\prod_{i=1}^k G_\gep(x_i-y_i)^{-\gb^2}.
\end{multline}
Thus it is sufficient to control  the contribution of $B\setminus A$ to conclude.
From \eqref{ineqB} we have: 
\begin{multline}
\int_{B\setminus A} \frac{   M^{\gamma,0}_\gep(\dd {\bf x}\dd {\bf y}) \left(\prod_{1\le i< j\le k}
G_\gep(x_i-x_j) \prod_{1\le i<j\le k}G_\gep(y_i-y_j)\right)^{\gb^2}}{\prod_{i,j=1}^k G_\gep(x_i-y_j)^{\gb^2}}\\
\le C(k,\gb^2) \int_{[0,1]^{2k}\setminus A} M^{\gamma,0}_\gep(\dd {\bf x}\dd {\bf y})\prod_{j} G_\gep(x_j-y_j)^{-\gb^2}.
\end{multline}
According to \eqref{michto}, the r.h.s. is smaller than $\delta  \left(\int_{[0,1]^2} M^{\gamma,0}_\gep(\dd { x}\dd { y})G_\gep(x-y)^{-\gb^2}\right)^k$
provided $\gep$ is chosen sufficiently small.

\subsection{Proof of Lemma \ref{lesdebris}}

We decompose the set $[0,1]^{2k}\setminus A$ as a union of non-disjoint events as follows
\begin{align*}
[0,1]^{2k}\setminus A&:=\bigcup_{i=1}^k \{({\bf x},{\bf y})\in[0,1]^{2k}\  | \ |x_i-y_i|> a(\gep)\}\\
& \quad \quad \quad\cup\bigcup_{1\le i< j\le k}\{  ({\bf x},{\bf y})\in[0,1]^{2k} \  | \ |x_i-x_j|< b(\gep)\}\\
&\quad \quad \quad \cup\bigcup_{1\le i< j\le k}\{  ({\bf x},{\bf y})\in[0,1]^{2k} \  | \ |y_i-y_j|< b(\gep)\}\\
&=:\bigcup_{i=1}^k\bar A_i\cup
\bigcup_{1\le i<j\le k}
\bar A_{i,j} \bigcup_{1\le i<j\le k}
\bar A'_{i,j}
\end{align*}

Then by permutation of the indices and symmetry in $x,y$ we have: 
\begin{align*}
\int_{[0,1]^{2k}\setminus A}
   M^{\gamma,0}_\gep (\dd {\bf x}\dd {\bf y})\prod_{i=1}^k G_\gep(x_i-y_i)^{-\gb^2}  &\le k
   \int_{\bar A_1} M^{\gamma,0}_\gep(\dd {\bf x}\dd {\bf y})\prod_{i=1}^k G_\gep(x_i-y_i)^{-\gb^2}\\
   &\quad +k(k-1)
   \int_{\bar A_{1,2}} M^{\gamma,0}_\gep(\dd {\bf x}\dd {\bf y})\prod_{i=1}^k G_\gep(x_i-y_i)^{-\gb^2}.
\end{align*}
Hence it is sufficient to show \eqref{michto} with $[0,1]^{2k}\setminus A$ replaced by ${\bar A_1}$ and ${\bar A_{1,2}}$ in the numerator's integrand.
After simplification it amounts to showing two things (see the next lemma).
One sets 
\begin{equation}
\begin{split}
D_2&:=\{(x,y)\in[0,1]^2 \ | \ |x-y|>a(\gep)\}\\
D_4&:=\{({\bf x},{\bf y})\in[0,1]^4 \ | \ |x_1-x_2|<b(\gep)\}.
\end{split}
\end{equation}

\begin{lemma}
The two following convergences hold in probability:
\begin{equation}\begin{split}\label{crimou}
 \lim_{\gep\to 0} \frac{\int_{D_2} G_\gep(x-y)^{-\gb^2} M^{\gamma,0}_\gep(\dd x\dd y)}{ \int_{[0,1]^2} G_\gep(x-y)^{-\gb^2} M^{\gamma,0}_\gep(\dd x\dd y) }
 &=0 ,\\
  \lim_{\gep\to 0} \frac{\int_{D_4} \prod_{i=1}^2 G_\gep(x_i-y_i)^{-\gb^2} M^{\gamma,0}_\gep(\dd {\bf x}\dd {\bf y})}{\left( \int_{[0,1]^2} G_\gep(x-y)^{-\gb^2} M^{\gamma,0}_\gep(\dd x\dd y) \right)^2 }&=0.
 \end{split}\end{equation}

\end{lemma}

\begin{proof}
Using the results of section \ref{slesmomentos} concerning convergence of the second moment, one can make the following replacement in the denominator of \eqref{crimou} 
\begin{equation}
\int_{[0,1]^2} G_\gep(x-y)^{-\gb^2} M^{\gamma,0}_\gep(\dd x\dd y) \approx \begin{cases}
\gep^{1-2\gamma^2-\gb^2} & \text{ if } \gamma<1/\sqrt{2} \text{ √Ç¬†and } \gamma^2+\gb^2>1, \\
\gep^{-\gamma^2}|\log(\gep)| & \text{ if } \gamma<1/\sqrt{2} \text{ √Ç¬†and } \gamma^2+\gb^2=1,\\
\gep^{1-2\gamma^2-\gb^2}|\log(\gep)|^{-1/2} & \text{ if } \gamma=1/\sqrt{2} \text{ √Ç¬†and } \gamma^2+\gb^2>1,
\end{cases}
\end{equation}
in the sense that the above ratios converge in probability towards a positive random variable.
Then the first line of \eqref{crimou} comes from an easy $\bbL_1$ computation
\begin{equation}
\bbE\left[\int_{D_2} G_\gep(x-y)^{-\gb^2} M^{\gamma,0}_\gep(\dd x\dd y)\right]\approx \begin{cases}
\gep^{-\gamma^2} a(\gep)^{1-\gamma^2-\gb^2} & \text{ if } \gamma\le 1/\sqrt{2} \text{ √Ç¬†and } \gamma^2+\gb^2>1, \\
\gep^{-\gamma^2}|\log( a(\gep))| & \text{ if } \gamma<1/\sqrt{2} \text{ √Ç¬†and } \gamma^2+\gb^2=1.
\end{cases}
\end{equation}
which is good enough as $a(\gep)>0$ (we only need to use the Markov inequality).
When  $\gamma= 1/\sqrt{2}$ we require   $(a(\gep)/\gep)^{\gamma^2+\gb^2-1}\gg |\log \gep|^{1/2} $ to make things work.

\medskip

A similar computation works for $\gamma<\frac{1}{2}$ for the second line of \eqref{crimou} but miserably fails in the other cases.
We present below a method that works in all cases.
Set $B(\gep)=b(\gep)^{-1}$ (and we assume that $b$ is defined so that $B$ is an integer).
For $j=1, \cdots, B-1$, we define $I_j:=[b(j-1),b(j+1)]$.

We have
\begin{equation}
\int_{D_4} \prod_{i=1}^2 G_\gep(x_i-y_i)^{-\gb^2} M^{\gamma,0}_\gep(\dd {\bf x}\dd {\bf y})
\le \sum_{j=1}^{B-1}\left(\int_{I_j\times [0,1]} G_\gep(x-y)^{-\gb^2} M^{\gamma,0}_\gep(\dd {x}\dd {y}) \right)^2.
\end{equation}
Now we treat the case $ \gamma<1/\sqrt{2} \text{ √Ç¬†and } \gamma^2+\gb^2>1$, the others can be dealt with similarly. We set 
\begin{equation*}
\bar{\bar{M}}^{2\gamma,0}_\gep(I)= \gep^{2\gamma^2+\beta^2-1} \int_{I \times [0,1]} G_\gep(x-y)^{-\gb^2} M^{\gamma,0}_\gep(\dd {x}\dd {y})
\end{equation*}
and therefore get:
\begin{equation}
( \gep^{2\gamma^2+\beta^2-1} )^2 \int_{D_4} \prod_{i=1}^2 G_\gep(x_i-y_i)^{-\gb^2} M^{\gamma,0}_\gep(\dd {\bf x}\dd {\bf y})
\le   \sup_{1 \leq j \leq B}  ( \bar{\bar{M}}^{2\gamma,0}_\gep[I_j] )  \:  \bar{\bar{M}}^{2\gamma,0}_\gep([0,1])
\end{equation}

By the results of section \ref{slesmomentos}, the random measures $\bar{\bar{M}}^{2\gamma,0}_\gep$ converge in probability in the space of Radon measures to the measure $M^{2\gamma,0}$. By extracting a subsequence, we can assume that almost sure convergence holds. Let $b >0$ be fixed. We have:
\begin{equation*} 
\underset{\gep \to 0}{\overline{\lim}}  \sup_{1 \leq j \leq B}  ( \bar{\bar{M}}^{2\gamma,0}_\gep[I_j] )  \:  \bar{\bar{M}}^{2\gamma,0}_\gep([0,1]) \leq \sup_{1 \leq j \leq b} ( M^{2\gamma,0} (I_j ) )  \:  M^{2\gamma,0}([0,1]) 
\end{equation*}
Now one can conclude by letting $b$ go to $0$ in the above inequality and using the fact that $M^{2\gamma,0}$ has no atoms.

Hence we have proved that: 
\begin{equation}
  \lim_{\gep\to 0} 
  \frac{\int_{D_4} \prod_{i=1}^2 G_\gep(x_i-y_i)^{-\gb^2} M^{\gamma,0}_\gep(\dd {\bf x}\dd {\bf y})}{\left( \int_{[0,1]^2} G_\gep(x-y)^{-\gb^2} M^{\gamma,0}_\gep(\dd x\dd y) \right)^2 }=0.
\end{equation} \end{proof}
 
 \begin{rem} Note that this last argument that cannot be adapted in phase 2. Not only the limiting measure possesses 
  atoms so that the proof doe not work but also the integral over $[0,1]^{2k}\setminus A$ is not negligible.
 \end{rem}

 \subsection{The case $k<k'$}
 
 This case is easier and only uses the tools developed for the $k=k'$ case.

An easy computation shows that   \eqref{clesmoments} for $k<k'$ corresponds to proving 

\begin{equation}\label{wouhou2}
\lim_{\gep\to 0} \frac{\int_{[0,1]^{k+k'}} 
\frac{M^{\gamma,0}_\gep(\dd {\bf x}\dd {\bf y}) \left(
\prod_{1\le i<j\le k}G_\gep(x_i-x_j) \prod_{1\le i<j\le k'}G_\gep(y_i-y_j)\right)^{\gb^2}}{\prod_{i=1}^k\prod_{j=1}^{k'}  G_\gep(x_i-y_j)^{\gb^2}}}
{\left(\int_{[0,1]^{2}} M^{\gamma,0}_\gep(\dd { x}\dd { y})  G_\gep(x-y)^{-\gb^2}\right)^{(k+k')/2}}=0.
\end{equation}

 Let $\sigma$ denote the function obtained from the matching procedure of ${\bf x}\in \bbR^k$, ${\bf y}\in \bbR^{k'}$. Set
 $$B:=\{ ({\bf x},{\bf y})\in [0,1]^{k+k'} \ | \  \sigma({\bf x},{\bf y})(i)=i, \ \forall i\in \{1,\dots, k\} \}.$$
 We have by invariance under permutation of the indices that the numerator above is equal to 
 
 \begin{equation}
 \frac{k'!}{(k'-k)!}\int_{B} 
\frac{M^{\gamma,0}_\gep(\dd {\bf x}\dd {\bf y}) {
\left(\prod_{1\le i<j\le k}G_\gep(x_i-x_j) \prod_{1\le i<j\le k'}G_\gep(y_i-y_j)\right)^{\gb^2}}}{\prod_{i=1}^{k}\prod_{j=1}^{k'}  G_\gep(x_i-y_j)^{\gb^2}}.
 \end{equation}
 Now we can adapt the proof of \eqref{ineqB} in Lemma \ref{lezinek} and show that 
 for all $({\bf x}, {\bf y})\in B$
 \begin{equation}
 \frac{\left(\prod_{1\le i<j\le k}G_\gep(x_i-x_j) \prod_{1\le i<j\le k'}G_\gep(y_i-y_j)\right)^{\gb^2}}{\prod_{i=1}^k \prod_{j=1}^{k'}  G_\gep(x_i-y_j)^{\gb^2}}
 \le C_1^{\gb^2 k(k'-1)} \frac{\prod_{k\le i<j\le k'}G_\gep(y_i-y_j)^{\gb^2}}{\prod_{i=1}^k G_\gep(x_i-y_i)^{\gb^2}}.
 \end{equation}
Then using Lemma \ref{greenk}, we see that  the numerator of the r.h.s. above is bounded by a constant.
Hence there is a constant $C$ such that 
\begin{align*}
\int_{B}& 
\frac{M^{\gamma,0}_\gep(\dd {\bf x}\dd {\bf y}) 
\left(\prod_{1\le i<j\le k}G_\gep(x_i-x_j) \prod_{1\le i<j\le k'}G_\gep(y_i-y_j)\right)^{\gb^2}}{\prod_{i=1}^k\prod_{j=1}^{k'}  G_\gep(x_i-y_j)^{\gb^2}}
\\
&\le C \int_{[0,1]^{k+k'}}\frac{M^{\gamma,0}_\gep(\dd {\bf x}\dd {\bf y})}{\prod_{i=1}^{k} G_\gep(x_i-y_i)^{\gb^2}}
\\&=
\left(\int_{[0,1]^{2}} M^{\gamma,0}_\gep(\dd { x}\dd { y})  G_\gep(x-y)^{-\gb^2}\right)^{k}\left(M^{\gamma,0}_\gep([0,1])\right)^{k'-k}.
\end{align*}
Hence to prove \eqref{wouhou2}, it is sufficient to prove that 

\begin{equation}
\lim_{\gep\to 0} \frac{M^{\gamma,0}_\gep([0,1])}{\sqrt{\int_{[0,1]^{2}} M^{\gamma,0}_\gep(\dd { x}\dd { y})  G_\gep(x-y)^{-\gb^2}}}=0.
\end{equation}
This is a simple consequence of the results of Section \ref{slesmomentos}.

\subsection{Proof of \eqref{clesmoments2}}

We treat only the case of two intervals, more intervals meaning only more notational problems.
We further assume that $J_1$ and $J_2$ are at a positive distance from one another for simplicity.

\medskip

If this is not the case, as when $J_1=[-1,0]$ and $J_2=[0,1]$, we can split $J_2$ into two intervals $J'_2=[0,\delta]$ and $J''_2=[\delta,1]$ 
and expand the factor $M(J_2)=M(J'_2)+M(J''_2)$ in the product.
Then using  H\"older's inequality together with \eqref{clesmoments} we prove that all the term where $J'_2$ appear have a negligible contribution when
 $\delta$ goes to zero.
 
 \medskip

The moment in the numerator is equal to
\begin{multline}
\gep^{k\gb^2}\int_{{\bf x}^i\in I^{k_i}} M^{\gamma,0}_\gep(\dd {\bf x})
\frac{\left(\prod_{l=1}^4 \prod_{1\le i<j\le k_l}G_\gep(x^l_i-x^l_j)\right)^{\gb^2}}{\left(\prod_{(l,m)\in\{ (1,2),(3,4)\}}\prod_{i=1}^{k_l}\prod_{j=1}^{k_m}G_\gep(x^l_i-x^m_j)\right)^{\gb^2}}\\
\times
\frac{
\left(\prod_{i=1}^{k_1}\prod_{j=1}^{k_3}G_\gep(x^1_i-x^3_j)\right)^{\gb^2}\left(\prod_{i=1}^{k_2}\prod_{j=1}^{k_4}G_\gep(x^2_i-x^4_j)\right)^{\gb^2}}{
\left(\prod_{(l,m)\in\{ (1,4),(2,3)\}}\prod_{i=1}^{k_l}\prod_{j=1}^{k_m}G_\gep(x^l_i-x^m_j)\right)^{\gb^2}}
\end{multline}

Now if $I$ and $J$ are at a positive distance, the term appearing on the second line is uniformly bounded, and hence we obtain the result for free when either 
$k_1\ne k_2$ or $k_3\ne k_4$.
In the case 
$k_1= k_2$ or $k_3= k_4$ we have to show that the cross term
$$\frac{
\left(\prod_{i=1}^{k_1}\prod_{j=1}^{k_3}G_\gep(x^1_i-x^3_j)\right)^{\gb^2}\left(\prod_{i=1}^{k_2}\prod_{j=1}^{k_4}G_\gep(x^2_i-x^4_j)\right)^{\gb^2}}{
\left(\prod_{(l,m)\in\{ (1,4),(2,3)\}}\prod_{i=1}^{k_l}\prod_{j=1}^{k_m}G_\gep(x^l_l-x^m_j)\right)^{\gb^2}}$$
is roughly equal to one on the subset of $J_1^{2k_1}\times J_2^{2k_3}$ that really matters.
First we remark that multiplying by $k_1!k_3!$ we can restrict to the set 
$$B_1\times B_2:=
\{ ({\bf x}^1, {\bf x}^2,{\bf x}^3, {\bf x}^4)\in J_1^{2k_1}\times J_2^{2k_3} \ | \ \sigma({\bf x}^1, {\bf x}^2)=\ind, \sigma({\bf x}^3, {\bf x}^4)=\ind\}.$$

Then we show that on the set $A_1\times A_2$ where $A_1$ and $A_2$ are defined similarly to $A$ in \eqref{defa}
we show that the cross term is in the interval $(1-\delta,1+\delta)$.
Then one can conclude by using Lemma \ref{lesdebris}.

\section{Differentiability of Gaussian processes}
In this section, we state a few results about Gaussian processes, mainly about differentiability   and convergence in law in the sense of distributions. These results are certainly not new but we have not found any proper reference. 

In what follows, $D$ stands for a compact subset of $\R^d$ and  let $X$ be a $\R^m$-valued stochastically continuous stochastic process defined on $D$.

First we recall the classical 
\begin{proposition}\label{kolm1}
If, for some $\beta,\alpha,C>0$: 
$$\forall x,z\in D,\quad \E[|X_x-X_z|^q]\leq C|x-z|^{d+\beta}.$$
For all $\gamma\in ]0,\frac{\beta}{q}[$, we set $L=\sup_{x\not = z}\frac{|X_x-X_z|}{|x-z|^\gamma}$. Then, for all $p<q$, $\E[L^\beta]\leq 1+\frac{C p2^{\beta-q\gamma}}{(q-p)(2^{\beta-q\gamma}-1)}$.
\end{proposition} 

Now we claim the following:
\begin{proposition}\label{kolm2}
Assume that $X$ is a centered Gaussian process with covariance kernel $K$. Then:\\
1) if for some $\alpha>0$ and $\forall x,y\in D$, $K(x,x)+K(z,z)-2K(x,z)\leq C|x-z|^\alpha$ then $X$ admits a $\gamma$-H\"older modification for all $\gamma\in]0,\alpha/2[$. \\
2) if, for some $k\in\N^*$ and $\alpha>0$, $K$ is of class $C^{2k}$ and $\forall x,y\in D$:
 \begin{equation}\label{regKkolm}
 \partial_x^k\partial_z^kK(x,x)+\partial_x^k\partial_z^kK(z,z)-2\partial_x^k\partial_z^kK(x,z)\leq C|x-z|^\alpha,
 \end{equation}
 then $X$ is of class $C^k$. Furthermore the $k$-th derivative is $\gamma$-H\"older   for all $\gamma\in]0,\alpha/2[$ and the H\"older constant is in $L^p$ for all $p>0$. 
\end{proposition} 

\noindent {\it Proof.} The first claim is just an application of Proposition \ref{kolm1}.  We outline the second claim. We consider a mollifying sequence $(\rho_n)_n$ and we define the Gaussian process   by convolution $X_n=X\star \rho_n$. For each $n$, this process is infinitely differentiable and the covariance structure is readily seen to be given, after integration by parts,  by:
$$\E[\partial^k_{x}X_n(x)\partial_{x}^kX_n(z)]=\int \int \rho_n(y)\rho_n(y')\partial_{x}^k\partial_{z}^kK(x-y,z-y')\,dydy'\stackrel{\text{def}}{=}K_{k,n}(x,z).$$
Because of \eqref{regKkolm}, it is plain to see that there exists a constant $C$ such that, for all $n$ and $x,y\in D$
$$\E[|\partial^k_{x_i}X_n(x)-\partial^k_{x_i}X_n(z)|^2]\leq C|x-y|.$$
 
The same argument holds for all the derivatives of order $k'$ for $k'=0,\dots,k$. The Kolmogorov criterion (Prop. \ref{kolm1}) ensures that the sequence $(X_n)_n$ is tight in $C^k$ equipped with the topology of uniform convergence of all derivatives of order $k'$ for $k'=0,\dots,k$. Since $(X_n)_n$ converges almost surely in $C(D)$ towards $X$, we deduce that $X$ is of class $C^k$.\qed

We deduce:

\begin{corollary}\label{kolm3}
Consider a sequence $(X_n)_n$   of centered Gaussian processes with respective covariance kernel $K_n$ of class $C^{2k}$ such that all the derivatives up to order $2k$ uniformly converges on $D$ towards a covariance kernel $K$. Then the sequence $(X_n,\partial_xX_n,\dots\partial^k_x X_n)$ converges in law in $C(D)$ towards $(X,\partial_xX,\dots\partial^k_x X)$, where $X$ is a centered Gaussian process on $D$.
\end{corollary}
\section{Auxiliary results of section \ref{GFF}}

\noindent {\it Proof of Lemma \ref{greenradius}.} We suppose that $B(x,\delta) \subset D$. We set $G_D$ to be the Green function in $D$. We have for all functions $F \ge 0$:
\begin{align*}
\int_D F(y)\left( \int_{\gep^2}^{\infty} p_{D}(t,x,y)\dd t\right) dy & =   \E^x[   \int_{\gep^2}^{\infty} F(B_t)  \ind_{\{\tau_D>t\}} \dd t ]  =   \E^x[   \E^{B_{\gep^2}}  [ \int_{0}^{\infty} F(B_t)  \ind_{\tau_D>t} \dd t]  \ind_{\{\tau_D>\gep^2\}}   ]  \\
& =   \E^x[   \int_D G_D(B_{\gep^2},y ) F(y)  dy \ind_{\{\tau_D>\gep^2\}}  ]      = \int_D  F(y)   \E^x[ G_D(B_{\gep^2},y ) \ind_{\{\tau_D>\gep^2\}} ] dy.  
\end{align*}
Hence we have$ \int_{\gep^2}^{\infty} p_{D}(t,x,y)\dd t = \E^x  [   G_D(B_{\gep^2},y) \ind_{\{\tau_D>\gep^2\}}  ] $. Now, we extend $G_D(x,y)$ to all $\R^2 \times \R^2$ by setting it equal to 0 as soon as $x$ or $y$ are not in the domain $D$. Recall that there exists some constant $C>0$ such that for all $x,y$ in $\R^2$, $
G_D(x,y) \le \ln_{+} \frac{1}{|y-x|}+C$. Now, we have:
\begin{align*}
\big|\E^x  [   G_D(B_{\gep^2},x)   ]-\E^x  [   G_D(B_{\gep^2},x) \ind_{\{\tau_D>\gep^2\}} ] \big|  
& = \E^x  [   G_D(B_{\gep^2},x) \ind_{\{\tau_D\le  \gep^2\}} ]   \\
& \le  \E^x  [   (\ln_{+} \frac{1}{|B_{\gep^2}-x|} + C)\ind_{\{\sup_{u \leq \gep^2} |B_u-x|   \ge \delta\}}  ]     \\  
& \le \E^x  [   (\ln_{+} \frac{1}{|B_{\gep^2}-x|} + C)^2]^{1/2} \Pb( \sup_{u \leq \gep^2} |B_u-x|   \ge \delta )^{1/2}  \\
& \le C (\ln \frac{1}{\gep})^2 e^{- C \delta^2/{\gep^2}}.
\end{align*}
Hence, we can replace $\E^x  [   G_D(B_{\gep^2},x) \ind_{\tau_D>\gep^2}  ] $ by $\E^x  [   G_D(B_{\gep^2},x) ] $. Similarly, we can replace $\E^x  [   G_D(B_{\gep^2},x) ] $ by $\E^x  [   G_D(B_{\gep^2},x) \ind_{\{|B_{\gep^2}-x| \le \delta\}}  ]  $. 
Therefore we get
\begin{align*}
\E^x  [   G_D(B_{\gep^2},x) \ind_{\{\tau_D>\gep^2\}}   ]  &  = \E^x  [   G_D(B_{\gep^2},x) \ind_{\{|B_{\gep^2}-x| \le \delta\}}  ]  +o(1) \\
& = \int_{|u| \le \frac{\delta}{\gep} }   \frac{e^{-|u|^2/2}}{2 \pi}  G_D(x+\gep u,x) \dd u +o(1)  .
\end{align*}
Now, by conformal invariance, we get that:
\begin{align*}
\int_{|u| \le \frac{\delta}{\gep} }   \frac{e^{-|u|^2/2}}{2 \pi}  G_D(x+\gep u,x) \dd u & = \int_{|u| \le \frac{\delta}{\gep} }   \frac{e^{-|u|^2/2}}{2 \pi}  G_{\mathbb{H}}(\varphi(x+\gep u), \varphi(x)) \dd u \\
& \approx   \int_{|u| \le \frac{\delta}{\gep} }   \frac{e^{-|u|^2/2}}{2 \pi}  G_{\mathbb{H}}(\varphi(x)+\gep \varphi'(x) u), \varphi(x)) \dd u 
\end{align*}
where the last line can be made rigorous by using the explicit expression of $G_{\mathbb{H}}$:
$$G_{\mathbb{H}}(x,y)=\ln\frac{1}{|x-y|}-\ln\frac{1}{|x-\bar{y}|}. $$
Therefore we get:
\begin{align*}
\int_{\gep^2}^{\infty} p_{D}(t,x,y)\dd t 
& = \int_{|u| \le \frac{\delta}{\gep} }   \frac{e^{-|u|^2/2}}{2 \pi}  G_{\mathbb{H}}(\varphi(x)+\gep \varphi'(x) u), \varphi(x)) \dd u +o(1)\\
& =   \int_{|u| \le \frac{\delta}{\gep} }   \frac{e^{-|u|^2/2}}{2 \pi}  \ln\frac{1}{|\gep \varphi'(x) u|} \, \dd u \\
& \quad  \quad    -\int_{|u| \le \frac{\delta}{\gep} }   \frac{e^{-|u|^2/2}}{2 \pi}  \ln\frac{1}{|\varphi(x)-\bar{\varphi}(x)+\gep \varphi'(x) u|} \, \dd u +o(1)\\
&=\ln\frac{1}{\gep}+ \ln\frac{1}{| \varphi'(x)|} +\ln 2{\rm Im}(\varphi(x))+o(1).
\end{align*}
The result follows via \eqref{confrad2}. \qed


\end{document}